\documentclass{article}
\usepackage{amssymb,amsthm,amsmath,mathrsfs, mathtools}
\usepackage[dvipsnames]{xcolor}
\usepackage{bbm} 
\usepackage{geometry}     
\usepackage{hyperref} 
\hypersetup{colorlinks=true,allcolors=blue}
\usepackage{hypcap}

\usepackage{imakeidx}
\makeindex[title=Index of terminology, columns=2] 
\makeindex[name=symbols, title=Index of symbols, columns=3] 

\usepackage[utf8]{inputenc}
\usepackage{indentfirst}

\newtheorem{theorem}{Theorem}[section]
\newtheorem{lem}[theorem]{Lemma}
\newtheorem{prop}[theorem]{Proposition}
\newtheorem{corol}[theorem]{Corollary}
\newtheorem{defin}[theorem]{Definition}
\newtheorem{fait}[theorem]{Fact}

\theoremstyle{remark}
\newtheorem{rem}[theorem]{Remark}

\renewcommand{\rm}[1]{\mathrm{#1}}
\renewcommand{\cal}[1]{\mathcal{#1}}
\newcommand{\bb}[1]{\mathbb{#1}}

\renewcommand{\frak}[1]{\mathfrak{#1}}
\newcommand{\dd}{ \; \mathrm{d}}
\newcommand{\dis}{\mathrm{d}}
\newcommand{\adis}{\mathrm{d}_{\underline{a}}}

\newcommand{\R}{\bb R}

\newcommand{\Z}{\bb Z}

\newcommand{\N}{\bb N}

\renewcommand{\P}{\bb P}

\newcommand{\calM}{\cal M}
\newcommand{\calF}{\cal F}
\newcommand{\calG}{\cal G}
\newcommand{\calO}{\cal O}
\newcommand{\calW}{\cal W}

\newcommand{\calL}{\cal L}
\newcommand{\calP}{\cal P}

\newcommand{\calD}{\cal D}

\newcommand{\frakS}{\frak S}
\newcommand{\frakD}{\frak D}

\newcommand{\sld}{\rm{SL}_d(\Z)}
\newcommand{\slr}{\rm{SL}_d(\R)}

\newcommand{\inj}{{inj} }

\DeclareMathOperator{\vol}{vol}
\DeclareMathOperator{\leb}{Leb_{\mathfrak{a}}}

\renewcommand{\c}{c}

\newcommand{\hypG}{Let $G$ be a connected, real linear, semisimple Lie group of non-compact type. }

\title{Equidistribution and counting of periodic tori in the space of Weyl chambers}
\author{Nguyen-Thi Dang and Jialun Li}
\date{}

\begin{document}

\maketitle
\begin{abstract}
Let $G$ be a semisimple Lie group without compact factor and $\Gamma < G$ a torsion-free, cocompact, irreducible lattice.
According to Selberg, periodic orbits of regular Weyl chamber flows live on tori. 
We prove that these periodic tori equidistribute exponentially fast towards the quotient of the Haar measure. From the equidistribution formula, we deduce a higher rank prime geodesic theorem.

\end{abstract}

\section{Introduction}

Let $G$ be a semisimple, connected, real linear Lie group without compact factor. 
Let $K$ be a maximal compact subgroup, $A$ be a maximal $\mathbb{R}$-split torus, $A^{+}\subset A$ a closed positive chamber such that the Cartan decomposition $G=KA^+K$ holds.
Denote by $M:=Z_K(A)$ the centralizer of $A$ in $K$. 

Let $\Gamma < G$ be a torsion-free, cocompact lattice. 
The double coset space $\Gamma \backslash G/M$ is called the \emph{space of Weyl chambers} of the symmetric space $\Gamma \backslash G/K$.
We study the counting and equidistribution of the compact right $A$-orbits in the space of Weyl chambers.

\subsection{Pioneering works on hyperbolic surfaces}
In this case, $G= \mathrm{PSL}(2,\mathbb{R})$ is the isometry group of the Poincaré half-plane $\mathbb{H}^2$, the space of Weyl chamber is the unit tangent bundle of the hyperbolic surface $\Gamma \backslash \mathbb{H}^2$ and the right action of $A$ on $\Gamma \backslash G/M$ corresponds to the geodesic flow.
Periodic orbits of the geodesic flow project in the surface to primitive closed geodesics.

\paragraph{Prime geodesic theorems} 
In 1959, Huber \cite{huberZurAnalytischenTheorie1959} proved a prime geodesic theorem for compact hyperbolic surfaces.
He obtained an estimate of the number of primitive closed geodesics as their length grows to infinity. 
More precisely, let $N(T)$ be the number of primitive closed geodesics of length less than $T$ on a hyperbolic surface. 
He proved that as $T$ tends to infinity, 
$$N(T) \sim e^T/T.$$
This term is similar to the asymptotic $x/\log x$ given by the prime number 
theorem\footnote{See Pollicott's research statement §1.2 \cite{pollicott}}
for the number of primes less than $x$. 
In 1969, using dynamical methods, Margulis \cite{margulisCertainApplicationsErgodic1969} extended the prime geodesic theorem to negatively curved compact manifolds. 
He proved that the exponential growth rate of $N(T)$ is equal to the topological entropy of the geodesic flow. 
Later on, relying on Selberg's Trace formula, Hejhal 
\cite{hejhal_selberg_1976} and Randol \cite{randol_asymptotic_1977} obtained a precise asymptotic development of the counting function in terms of the spectrum of the Laplace-Beltrami operator. 
In 1980, Sarnak \cite{sarnak1980} extended their precise asymptotic development to finite area surfaces.

Let us state one of the various equivalent formulations of the prime geodesic theorem.
For a closed geodesic $\c$ on $\Gamma\backslash \mathbb{H}^2$, denote by $\ell(\c)$ the length of this geodesic. Let $c_0$ be the primitive closed geodesic underlying $\c$.
Then as $T\rightarrow +\infty$
\begin{equation}\label{equ-prime}
\sum_{\c_0}\left\lfloor \frac{T}{\ell(\c_0)} \right\rfloor\ell(c_0)=\sum_{\c,\ell(c)\leq T}\ell(c_0)\sim e^T,
\end{equation}
where the first sum is over all primitive closed geodesics, the second sum is over all closed geodesics. 
This sum is similar to the second Chebyshev function: the weighted sum of the logarithms of primes less than a given number, where the weight is the highest power of the prime that does not exceed the given number. 
The second Chebyshev function is essentially equivalent to the prime counting function and their asymptotic behaviour is similar. 

\paragraph{Equidistribution of closed geodesics}
Margulis in his 1970 thesis\footnote{See Parry's review \cite{parry}} and Bowen \cite{bowenPeriodicOrbitsHyperbolic1972}, \cite{bowen_equidistribution_1972} independently studied the spatial distribution of the closed orbits of the geodesic flow. They proved that closed orbits uniformly equidistribute towards a measure of maximal entropy as their period tends to infinity. 
In the second 1972 paper, Bowen proved the uniqueness of the measure of maximal entropy for the geodesic flow.
As a consequence, the measure of maximal entropy of the geodesic flow is equal to the quotient of the Haar measure.
Later, Zelditch \cite{zelditchSelbergTraceFormulae1992} generalized Bowen's equidistribution theorem to finite area hyperbolic surfaces.

Let us recall Bowen and Margulis' result for a compact hyperbolic surface. 
For every primitive periodic orbit $F \subset  \Gamma \backslash \mathrm{PSL}(2,\mathbb{R})$, denote by $\cal{P}_{F}$ the unique probability measure invariant under the geodesic flow supported on $c$.
For every $T >0$, we denote by $\calG_p(T)$ the set of primitive periodic orbits of minimal period less that $T$.
Bowen and Margulis proved that for every bounded smooth function $f$,
$$ \frac{T}{e^{ T}} \sum_{F \in \calG_p(T) } \int f \; \dd\cal{P}_{F} \xrightarrow[T \rightarrow \infty]{} \int f \; \dd m_{\Gamma},$$
where $m_{\Gamma}$ is the measure of maximal entropy, which also corresponds in our case to the quotient measure of the Haar measure on $ \Gamma \backslash \mathrm{PSL}(2,\R)$.

The following non exhaustive list \cite{degeorge_length_1977}, \cite{gangolli_zeta_1980}, \cite{parryAnaloguePrimeNumber1983},
\cite{roblin}, \cite{naud2005expanding}, \cite{margulis_closed_2014} provides some of the many subsequent works tackling the counting and equidistribution problem in several different rank one generalisations.

\subsection{Main results}

In this article, we focus on the higher rank case\footnote{more precisely, we do not have restrictions on the rank of $G$} for $G$,
meaning that $\dim _{\mathbb{R}} A \geq 2$.  
Denote by $\frak a:= \mathrm{Lie} \; A$ the Cartan subspace, by $ \frak{a}^+$ the closed positive chamber in the Lie algebra and by $\frak{a}^{++}$ its interior.  

\begin{defin}[Periodic flat tori]\label{defin-flat-per-tori}
For any right $A$-orbit $F$ in $\Gamma\backslash G/M$, we define the set of \emph{periods} of $F$ as
$$ \Lambda (F):= \lbrace Y \in \frak{a} \; \vert \; ze^Y=z ,\;   \forall z \in F \rbrace. $$
A period $Y$ in $\Lambda(F)$ is called \emph{regular} if $Y\in \frak a^{++}$.
When $\Lambda(F)$ is a maximal grid of $\frak{a}$, we say $F$ a \emph{periodic flat torus} or a \emph{compact $A$-orbit}. 
\end{defin}
Denote by $C(A)$ the set of compact $A$-orbits in $\Gamma\backslash G/M$. 
For every $F \in C(A)$, we denote by $L_F$ the quotient measure on $F$ of $\leb$, the Lebesgue measure on $\frak{a}$. 
Note that $L_F$ is not a probability measure. Its total mass, denoted by $\vol_{\frak{a}} (F)$, is the Lebesgue measure of any fundamental domain in $\mathfrak{a}$ of the grid $\Lambda (F)$.

\paragraph{Main counting result}

We use $\vol$ to denote the Haar measure on $G$ whose quotient on the symmetric space $X:=G/K$ equals the measure induced by the Riemannian metric. 
Denote by $\Vert \; \Vert$ the Euclidean norm on $\frak{a}$ coming from the Killing form on $\mathfrak{g}$ and by $B_\frak{a}$ the balls for this norm.
For every $T>0$, set $B_{\frak a}^{++}(0,T):=B_{\frak a}(0,T)\cap\frak a^{++}$ and
$D_T:= K \exp\big( B_{\frak{a}}(0,T) \big) K$, which is the preimage by the quotient map $G\rightarrow X$ of the ball of radius $T$ centered at $eK$ in the symmetric space $X$.

\begin{theorem}\label{corol-counting}
Let $G$ be a semisimple, connected, real linear Lie group without compact factor and $\Gamma < G$ be a torsion-free, cocompact irreducible lattice. 
Then there exist constants $C_G>0$ and $u>0$ such that for $T>0$
\begin{equation}
	\sum_{F\in C(A)}|\Lambda(F)\cap B_{\frak a}^{++}(0,T)| \;
	\vol_{\frak{a}}(F)
	= \vol(D_T)(C_G+O(e^{-u T})).
\end{equation}
There exists $C_G'>0$ such that for any non-degenerate parallelotope domain $\mathcal{P}\subset \mathfrak{a}^{+}$ whose faces are parallel to the walls of the Weyl chamber, there exists an entropy $\delta_{\mathcal{P}}>0$, a gap $u_{\mathcal{P}}>0$ such that for $T>0$
\begin{equation}\label{equ:parallelotope domain}
	\sum_{F\in C(A)}|\Lambda(F)\cap T \mathcal{P} \cap \mathfrak{a}^{++} | \;
	\vol_{\frak{a}}(F)
	= e^{\delta_{\mathcal{P}} T} (C_G'+O(e^{-u_{\mathcal{P}} T})).
\end{equation}
\end{theorem}
We deduce this counting result from the subsequent equidistribution statement.
\paragraph{Main equidistribution result}
Denote by
$\pi : G \rightarrow \Gamma \backslash G/M$ the projection and by $\widetilde m_{\Gamma}$ the quotient measure of the Haar measure $\vol$.
We normalise $\widetilde m_{\Gamma}$ to obtain a probability measure that we denote by $m_{\Gamma}$. 

We obtain a higher rank version of the Bowen-Margulis equidistribution formula with an exponential rate of convergence. 

\begin{theorem}\label{thm-introequid}
Under the same hypothesis and for the same constants $C_G>0$ and $u>0$ as in the previous Theorem \ref{corol-counting}, for all $T>0$ and every Lipschitz function $f$ on $\Gamma\backslash G/M$ we have

\begin{equation}\label{equ:ball domain}
	\frac{1}{ \vol(D_T)}\sum_{F\in C(A)}|\Lambda(F)\cap B_{\frak a}^{++}(0,T)|\int_F f \; \dd L_F= C_G \int_{\Gamma\backslash G/M} f \; \dd m_{\Gamma}+O(e^{-u T}|f|_{Lip}).
\end{equation}
Additionally, for the same $C_G'>0$, parallelotope domain $\mathcal{P}\subset \mathfrak{a}^+$ and constants $\delta_{\mathcal{P}}, u_{\mathcal{P}}>0$, for all $T>0$ and every Lipschitz function $f$ on $\Gamma\backslash G/M$ we have
\begin{equation}
e^{-\delta_{\mathcal{P}}T}\sum_{F\in C(A)}|\Lambda(F)\cap T \mathcal{P}\cap \mathfrak{a}^{++}|\int_F f \; \dd L_F= C_G' \int_{\Gamma\backslash G/M} f \; \dd m_{\Gamma}+O(e^{-u_{\mathcal{P}} T}|f|_{Lip}).
\end{equation}
\end{theorem}

The asymptotic behaviour of the main term for the ball domain is
$ \vol(D_T)\sim C_0 T^{\frac{ \dim  A -1}{2} } e^{\delta_0 T} $, where $\delta_0>0$ is determined by the root system of $\mathfrak{g}$, the Lie algebra of $G$ and $C_0>0$ is given by the Harish-Chandra formula.
Without the error term, we deduce the following convergence where $r:= \dim  A$.
\begin{equation}
 \frac{C_{G,0} e^{-\delta_0 T}}{T^{\frac{r-1}{2}}}\sum_{F\in C(A)}|\Lambda(F)\cap B_{\frak a}^{++}(0,T)|\int_F f \; \dd L_F \xrightarrow[T \rightarrow \infty]{}	\int_{\Gamma\backslash G/M} f \; \dd m_{\Gamma} .
\end{equation}

Note that in the rank one case, any periodic flat torus $F$ corresponds to a primitive closed geodesic.
Furthermore, both $\vol_{\frak a}(F)$ and its smallest regular period correspond to the length of the geodesic. 
Therefore Theorem \ref{corol-counting} is a higher rank version of the prime geodesic theorem \eqref{equ-prime}. 

\begin{figure}[h!]
    \centering
    \includegraphics[width=10cm]{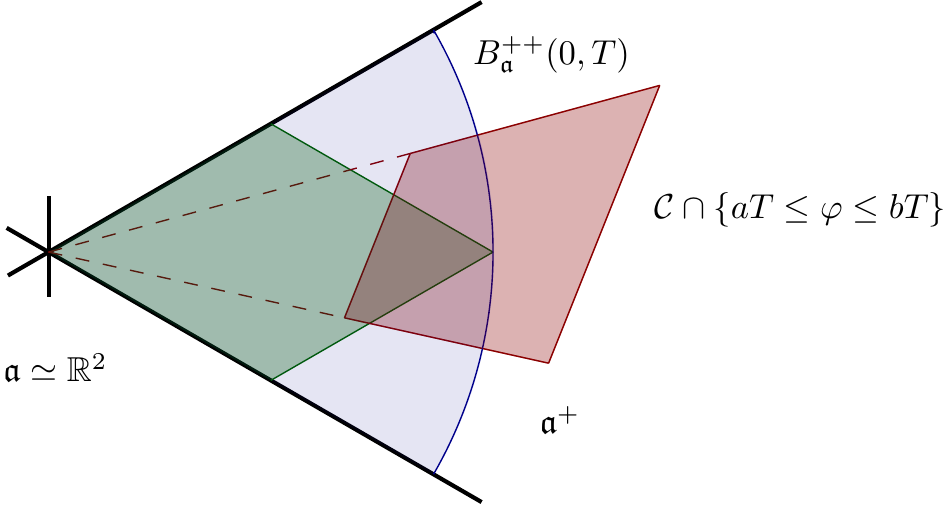}
    \caption{This is a positive Weyl chamber for $\mathrm{SL}(3,\mathbb{R})$ and $T>0$ is large. 
    In blue, our counting region $B_{\frak a}^{++}(0,T)$.
    In green, Deitmar-Gon-Spilioti's \cite{deitmar_prime_2019} counting region.
    In red, Guedes Bonthonneau--Guillarmou--Weich's \cite{bonthonneau_srb_2021} counting region where $\mathcal{C}$ is a convex cone strictly inside $\mathfrak{a}^{++}$ delimited by the red dashed lines, $0<a<b$ are real numbers and $\varphi$ is a linear form strictly positive on $\mathfrak{a}^+$.
    }
    \label{fig:domains}
\end{figure}

In the compact case, Spatzier in his thesis \cite{spatzier83} computed, using the root spaces of the Lie algebra of $G$, the topological entropy of every regular Weyl chamber flows: right action of $\exp(\R Y)$ on $\Gamma \backslash G/M$, where $Y\in \mathfrak{a}^{++}$ is non zero. 
Furthermore, $\delta_0$, the exponential growth rate of $ \vol(D_T)$, is a sharp upper bound of the topological entropy of regular Weyl chamber flow.
He also proved that $\delta_0$ is equal to the exponential growth rate of 
the sum over periodic flat tori of the smallest regular period less than $t$ of $\vol_{\frak a}(F)$, as $t$ goes to infinity. 
Knieper \cite{knieper2005uniqueness} studied the equidistribution of periodic orbits of regular Weyl chamber flows in the same setting. 
He obtained an equidistribution formula towards the measure of maximal entropy of the most chaotic regular Weyl chamber flow, whose topological entropy is $\delta_0$. In the finite volume case, Oh \cite{oh} proved that the number of periodic flat tori of bounded volume is always finite. 

In the compact case, Deitmar \cite[Theorem 3.1]{deitmar} used a Selberg trace formula and methods from analytical number theory to give the main term \eqref{equ:parallelotope domain} in Theorem \ref{corol-counting} (Using \eqref{equ:conjugacy} to connect the counting of conjugacy classes in Theorem 3.1 in \cite{deitmar} with Theorem \ref{corol-counting}). Actually, Deitmar's result is more general that he only needs each edge of the parallelotope goes to infinite. For parallelotope domain, Theorem \ref{corol-counting} and \ref{thm-introequid}  provides an equidistribution result and an exponential speed of convergence, which is new compared to \cite{deitmar}. 

\begin{rem}
Recently and for the compact case, Guedes Bonthonneau--Guillarmou--Weich \cite[Theorem 2, equation (0.3)]{bonthonneau_srb_2021} obtained a weighted equidistribution formula. 
Each period point is weighted by a dynamical determinant and the region where they count the period points is defined using any convex non-degenerate closed cone $\mathcal{C}$ strictly inside $\mathfrak{a}^{++}$, any choice of positive numbers $0<a<b$ and any linear form $\varphi$ that takes positive values in $\mathfrak{a}^+$ as shown in red in Figure \ref{fig:domains}.
They take a different approach, relying on the spectral properties of the $A$-action via their previous study of Ruelle-Taylor resonances with Hilgert \cite{bonthonneau_taylor_2020}.

For the non-compact, finite volume case $\mathrm{SL}(3,\mathbb{Z}) \backslash \mathrm{SL}(3,\R)$, Einsiedler--Lindenstrauss--Michel--Venkatesh in \cite{einsiedler_distribution_2011} use the classification of diagonal invariant measures and subconvexity estimates to deduce an equidistribution result for the following collection of tori. 
They take sets of periodic tori of the same volume and prove that the sum of Lebesgue measures on those tori, normalised by the total mass, equidistributes towards the quotient measure of the Haar measure as the volume goes to infinity.

Deitmar later on generalised his counting result to the non-compact finite volume case $\mathrm{SL}(3,\Z) \backslash \mathrm{SL}(3,\R)$, in joint work with Gon and Spilioti in \cite{deitmar_prime_2019}, with a different summation region in the Weyl chamber, the one in green in Figure \ref{fig:domains}.
\end{rem}

\begin{rem}
1. Our counting region in \eqref{equ:ball domain} is different (shown in blue in Figure \ref{fig:domains}), so our first asymptotic term is new in the higher rank case. 

None of the above works provides estimates on the speed of convergence. 
The decay rate $u$ in Theorem \ref{thm-introequid} only depends on a parameter $n(G,\Gamma)$ from spectral gaps, so it is uniform over all congruence subgroups.

2. In a forthcoming paper, we will prove the same counting and equidistribution results for irreducible non-cocompact higher rank lattices. The extra ingredient for the non-cocompact case results from the non-escape of mass for periodic tori. The case for $\sld$ is written in a previous arXiv version.
\end{rem}

\paragraph{Counting conjugacy classes}
We deduce an asymptotic formula of loxodromic conjugacy classes with a weight given by the volume of the corresponding periodic torus. See Section \ref{sec:equidistribution} for more details.

One application is an upper bound of the growth of conjugacy classes. Set $[\Gamma]$ be the set of conjugacy classes and let 
$$[\Gamma](T):=\{[\gamma]\in[\Gamma] |\,\lambda(\gamma)\in B_{\frak a}(0,T) \}.$$ 
Because $\Gamma$ is torsion free, there is a one-to-one correspondence between conjugacy classes and free homotopic classes of closed geodesics of the locally symmetric space $\Gamma\backslash X$. 
So $|[\Gamma](T)|$ is the number of free homotopic classes of closed geodesics of length less than $t$. The exponential growth rate of $|[\Gamma](T)|$ for higher rank lattice is still unknown \cite{knieper2005uniqueness}. Here we give an upper bound with a non-trivial polynomial term.
\begin{theorem}
\label{corol_conjugacy_count}
	Let $\Gamma$ be a cocompact irreducible lattice without torsion. Then 
	\[|[\Gamma](T)|\ll \frac{e^{\delta_0 T}}{T^{(r+1)/2}} ,\]
	where $r=\dim  A$.
\end{theorem}
The interesting point is that the polynomial term $T^{-(r+1)/2}$ in the upper bound depends on the real rank of the group $G$. We hope this polynomial term is the correct asymptotic for $|[\Gamma](T)|$. This upper bound also hints at Knieper's question in \cite[Remark in page 175]{knieper2005uniqueness}.

\subsection{Overview of the proofs}

The first step of the proof is to rewrite the sum of ``delta masses on the tori" (Cf. §4) using conjugacy class of loxodromic elements in the discrete group and their Jordan projection.
In the $\mathrm{SL}(2,\R)$ case, periodic orbits of the geodesic flow are in one to one correspondence with conjugacy class of hyperbolic elements in the discrete group.
In the higher rank case, every regular period $Y \in \Lambda(F) \cap \frak{a}^{++}$ of any periodic flat torus $F$ corresponds to a regular Weyl chamber flow $z \mapsto z e^{t Y}$ that admits for all $z \in F$ a $1$-periodic orbit.
Now, instead of hyperbolic elements and the translation length, we use loxodromic elements and the Jordan projection $\lambda$ (Cf. Definition \ref{def-Jordan-proj}).
The conjugacy class of any loxodromic element $[\gamma]$ in the discrete subgroup is in a one to one correspondence with $(F_{[\gamma]}, \lambda(\gamma) )$, where $F_{[\gamma]}$ is an $A$-orbit and $\lambda(\gamma)$ one of its regular periods (Cf. Proposition \ref{lem-gaglox}).
Denote by $\Gamma_c^{lox}$ the subset of loxodromic elements whose conjugacy class corresponds to a compact $A$-orbit and rewrite the sum as follows.
\begin{equation}\label{reorganising-sum}
 \sum_{F \in C(A)} \vert \Lambda(F) \cap B_{\frak a}^{++}(0,T) \vert \; L_F = \underset{\Vert \lambda(\gamma) \Vert \leq T}{ \sum_{[\gamma] \in [\Gamma_c^{lox}]} } L_{F_{[\gamma]}} .   
\end{equation}

In the cocompact case, by a Selberg's Lemma in \cite{selberg} (Cf. Lemma \ref{lem-selberg}) then $\Gamma_c^{lox}=\Gamma^{lox}$, i.e. every $1$-periodic orbit of a regular Weyl chamber flow lives in a periodic flat torus.

\subsubsection{Local equidistribution}

In a second step of the proof, we follow Roblin's strategy \cite[Chapter 5]{roblin} to get a local equidistribution in the cover: the space of Weyl chambers $G/M$ of the symmetric space $X=G/K$.

Roblin works in a CAT($-1$) space, let us sketch his method in the particular case of the hyperbolic plane and for a cocompact, torsion free, discrete subgroup $\Gamma < \mathrm{SL}(2,\R)$.

Using Patterson-Sullivan theory, he constructs the Bowen-Margulis-Sullivan (BMS) measure $m_{BMS}$ on $(\partial \mathbb{H}^2 \times \partial \mathbb{H}^2 \setminus \Delta) \times \R$ where $\Delta$ is the diagonal.
By Hopf coordinates, $m_{BMS}$ corresponds to a $\Gamma$-invariant and geodesic flow invariant measure on $T^1 \mathbb{H}^2$. 

Roblin then relies on mixing of the geodesic flow for the BMS measure in \cite[Chapter 4]{roblin} to deduce an equidistribution formula of orbit points $(\gamma x, \gamma^{-1}x)_{\gamma \in \Gamma}$.
In average, these points equidistribute towards a product of conformal Patterson-Sullivan densities of the geometric boundary of $\mathbb{H}^2$.
That horofunction and geometric compactification coincide in this case is one of the key reasons why this convergence holds.

To get from orbit points to periodic orbits, he then relies on a geometric configuration between $x, \gamma x, \gamma^{-1}x$ that implies that $\gamma$ is hyperbolic. 
For any point $x\in \mathbb{H}^2$ and an isometry $\gamma$ such that $x \neq \gamma x$, denote by $\gamma_x^+ \in \partial \mathbb{H}^2$ (resp. $\gamma_x^-$) the endpoint of the geodesic starting at $x$ going through $\gamma x$ (resp. $\gamma^{-1}x$).
Namely, if the geodesic of endpoints $\gamma_x^+$ and $\gamma_x^-$ passes close enough to $x$ and $d_{\mathbb{H}^2}(x,\gamma x)$ is large enough, then $\gamma$ is hyperbolic, its translation axis passes close to $x$ and the attracting (resp. repelling) endpoint $\gamma^+ \in \partial \mathbb{H}^2$ (resp. $\gamma^-$) is close to $\gamma_x^+$ (resp. $\gamma_x^-$).
Under restriction to suitable small sets called "corridors", this geometric configuration allows to remove finitely many terms in the sum of Dirac masses, the rest corresponding to translation axis of hyperbolic elements. 
Using a partition of unity, one then deduces the equidistribution in the quotient $T^1 (\Gamma \backslash \mathbb{H}^2)$.

\paragraph{Higher rank situation}

Horofunction and geometric compactification of higher rank symmetric spaces are no longer the same.  
However, the space of Weyl chamber $G/M$ of the higher rank symmetric space $X$ admits Hopf coordinates $\mathcal{F}^{(2)} \times \mathfrak{a}$ (Cf. \cite[Chapter 8, §8.G.2]{thirion} or §2 below), where $\mathcal{F}$ is the Furstenberg boundary and $\mathcal{F}^{(2)}$ is the subset of transverse pairs in $\calF \times \calF$.

Each delta mass $L_{ F_{[\gamma]} }$ in the sum \eqref{reorganising-sum} is the quotient of the measure of $G/M$ whose disintegration along Hopf coordinates is $D_{\gamma^+} \otimes D_{\gamma^-} \otimes Leb_{\mathfrak{a}}$, where $\gamma^+$ (resp. $\gamma^-$) are the attracting (resp. repelling) fixed points for the left action of $\gamma$ on $\calF$. 

The Haar measure on $G/M$ can be disintegrated along the Hopf coordinates and we write it as a higher rank BMS measure.
Haar densities on $\calF$ satisfy identities reminiscent of Patterson-Sullivan theory (Cf. §2.2 or \cite{helgasonGroupsGeometricAnalysis2000},\cite{albuquerque99}, \cite{quintMesuresPattersonSullivanRang2002}). 

We do not look for an equidistribution formula of orbit points in a suitable compactification of the higher rank symmetric space. 
Instead of looking at geodesic half-lines and their endpoints in the geometric boundary, 
we use the identification of $G/M$ with geometric Weyl chambers i.e. isometric embeddings in $X$ of the closed positive Weyl chamber $\frak a^+$ which in turn can be parameterized by $X \times \cal F$, the data of the base point and the asymptotic Weyl chamber which identifies with the Furstenberg boundary.
For every $\gamma$ and $x \in X$, provided $\gamma x$ is in the interior of a geometric Weyl chamber based at $x$ i.e. $\adis(x,\gamma x) \in \frak a^{++}$, (Cf. \eqref{defin-dist-a} for a formal definition) one can define the asymptotic directions of angular points in the Furstenberg boundary $\gamma_x^+$, $\gamma_x^-$. 
Gorodnik-Nevo \cite{gorodnikCountingLatticePoints2012} prove an equidistribution formula of "angular" points for irreducible lattices towards Haar densities of $\calF$.

For Lipschitz test functions $\psi$ and any $x \in X$, for all $t \gg \dis_X(o,x)$,  
$$ \frac{1}{ \vol(D_t)} \underset{ \adis(x,\gamma x) \in B_{\frak a}^{++}(0,T) }{\sum_{\gamma \in \Gamma}} \psi(\gamma_x^+,\gamma_x^-) = \frac{1}{ \vol(G \backslash \Gamma)} \int \psi \dd \mu_x \otimes \mu_x+ E(t,\psi,x)$$
where $\mu_x$ is the $Stab_G(x)$ invariant Haar density on $\calF$ and the error $E(t,\psi, x)= O(Lip(\psi) C_x  \vol(D_t)^{-\kappa})$ with $\kappa >0$ and $\log C_x \gg \dis_X(o,x)$.
Their formula provides an error term that comes from the spectral properties of averaging operators on the Borel probability spaces $L^p(\Gamma \backslash G)$.

Then we get from asymptotic angular points to attracting or repelling fixed point of loxodromic elements by adapting the geometric argument to the higher rank case.

\subsection*{Organization of the paper}

In Section 2, we gather the basic facts and preliminaries about semisimple real Lie groups, the Furstenberg boundary, Hopf coordinates, higher rank Patterson-Sullivan measure, volume estimates and the angular distribution of lattice points.

In Section 3, we prove a lemma comparing the angular part of an element in $G$ with its contracting and repelling fixed points in the Furstenberg boundary. In Section 4, we relate loxodromic elements and periodic tori.

In Section 5, we prove Theorem \ref{thm-introequid} for cocompact lattices. 

In Section 6, we prove Theorem \ref{corol_conjugacy_count}.

In Appendix, we follow the works of Gorodnik-Nevo \cite{gorodnikCountingLatticePoints2012} \cite{gorodnikLiftingRestrictingSifting2012} and explain why their results work in our setting.

\paragraph{Notation.} In the paper, given two real functions $f$ and $g$, we write $f \ll g$ or $f=O(g)$ if there exists a
constant $C > 0$ only depending on $G, \Gamma$ such that $f \leq Cg$. We write $f \asymp g$ if $f \ll g$ and $g\ll f$.

\section*{Acknowledgement}
We would like to thank Alex Gorodnik for explaining his results with Amos Nevo and Viet Dang for his inspirational habilitation defence sparking the collaboration. The first author would also like to thank Mark Pollicott for telling her to look at the counting problem in Ralf Spatzier's thesis. We would like to thank Jean Lécureux for pointing out the reference and including the proof of Proposition \ref{prop-bruhat-dyn} in Appendix.
Part of this work was done while two authors were at the \textit{Hyperbolic dynamical systems and resonances} conference in Porquerolles and the \textit{Anosov}$^3$ mini-workshop in Oberwolfach; we would like to thank the organizers Colin Guillarmou, Benjamin Delarue (formerly K\"{u}ster), Maria Beatrice Pozzetti, Tobias Weich, and the hospitality of the centres. We would also like to thank the hospitality of Institut für Mathematik Universität Zürich and Fakultät für Mathematik und Informatik Universität Heidelberg for each time the authors visit each other.
The first author acknowledges funding by the Deutsche Forschungsgemeinschaft (DFG, German Research Foundation) – 281869850 (RTG 2229) and by ANR PIA funding: ANR-20-IDEES-0002. 
The second author acknowledges the funding by Alex Gorodnik's SNF grant 200021--182089.

\section{Background}

\begin{center}
\fbox{
\begin{minipage}{.7 \textwidth}
In the whole article, $G$ is a semisimple, connected, real linear Lie group, without compact factor.
\end{minipage}
}
\end{center}

Classical references for this section are \cite{thirion}, \cite{guivarc2012compactifications}, \cite{helgason1978differential}. One also may refer to the exposition in \cite{dg20}.

Let $K$ be a maximal compact subgroup of $G$. Then $X= G/K$ is a globally symmetric space of non-compact type and $G= \mathrm{Isom}_0(X)$. 
We fix a base point $o \in X$ such that $K= \mathrm{Stab}_G(o)$. 
For every $x\in X$, we denote by $K_x:= \mathrm{Stab}_G(x)$. Note that for any $h_x \in G$ such that $h_x o=x$, then $K_x=h_xK h_x^{-1}$, independently of the choice of $h_x$.

\paragraph{Geometric Weyl chambers}
Denote by $\frak{g}$ (resp. $\frak{k}$) the Lie algebra of $G$ (resp. $K$) and consider the Cartan decomposition in the Lie algebra $\frak{g}= \frak{k} \oplus \frak{p}$.
Let $\frak{a} \subset \frak{p}$ be a \emph{Cartan subspace} of $\frak{g}$.
Then $A:= \exp(\frak{a})$ is a maximal $\mathbb{R}$-split torus of $G$.
Denote by $M:= Z_K(A)$ the centralizer of $A$ in $K$.
The \emph{real rank} of $G$, denoted by $r_G$, is equal to $\dim _{\mathbb{R}} \frak{a}$.
We say that $G$ is of \emph{higher rank} when $r_G \geq 2$.

For any linear form $\alpha$ on $\frak{a}$, set $\frak{g}_\alpha := \lbrace v \in \frak{g} \; \vert \; \forall u \in \frak{a}, \; [u,v]=\alpha(u) v \rbrace.$
The \emph{set of restricted roots} is denoted by $\Sigma:= \lbrace \alpha \in \frak{a}^*\setminus \lbrace 0\rbrace \; \vert \; \frak{g}_{\alpha} \neq \lbrace 0\rbrace \rbrace.$
The kernel of each restricted root is a hyperplane of $\frak{a}$. The \emph{Weyl chambers} of $\frak{a}$ are the connected components of $\frak{a} \setminus \cup_{\alpha\in \Sigma} \ker(\alpha)$. 
We choose a \emph{positive Weyl chamber} by fixing such a connected component and denote it (resp. its closure) by $\frak{a}^{++}$ (resp. $\frak{a}^+$).
In the Lie group, we denote by $A^{++}:=\exp(\frak{a}^{++})$ (resp. $A^+:= \exp(\frak{a}^+)$). Denote by $N_K(A)$ the normalizer of $A$ in $K$. The group $N_K(A)/M$ is the \emph{Weyl group}, denoted by $\calW$. The Weyl group also acts on the Lie algebra $\frak a$ by the adjoint action, which acts transitively on the set of connected components of $\frak{a} \setminus \cup_{\alpha\in \Sigma} \ker(\alpha)$.

A \emph{geometric Weyl chamber} is a subset of $X$ of the form $g. (A^+ o)$, where $g \in G$.
The \emph{base point} of the geometric Weyl chamber $g A^+o$ is the point $go \in X$.
In \cite[§2]{dg20}, we obtained the following identifications between the space of Weyl chambers and the set of geometric Weyl chambers of $X$,
\begin{equation}\label{eq-geomWC}
 G/M \simeq G. (A^+ o).
\end{equation}

\paragraph{Cartan projection}
\begin{defin}\label{defin-cartan-proj}
For any $g \in G$, we define, by Cartan decomposition, a unique element $\underline{a}(g) \in \frak{a}^+$ such that $g \in K \exp(\underline{a}(g)) K$.
The map $\underline{a}:G \rightarrow \frak{a}^+$ is called the \emph{Cartan projection}.
\end{defin}
Recall $\Vert . \Vert$ is the associated norm on $\frak{a}$ coming from the Killing form. 
The Cartan projection allows to define an $\frak{a}^+$-valued function on $X \times X$, denoted by $\adis$. 
For every $x,y\in X$, any choice $h_x,h_y \in G$ such that $h_x o =x$ and $h_yo=y$, we set
\begin{equation}\label{defin-dist-a}
    \adis (x,y) := \underline{a}(h_x^{-1} h_y).
\end{equation}
This function does not depend on the choice of $h_x$ and $h_y$ up to right multiplication by $K$. We define the $G$-invariant riemannian distance on $X$
$$ \dis_X(x,y):= \Vert \adis(x,y) \Vert .$$

The following fact is standard for symmetric spaces of non-compact type.
\begin{fait}\label{fait-a-dist}
For every $x,y\in X$, there is a geometric Weyl chamber based on $x$ containing $y$.
If furthermore, $\adis(x,y)\in \frak{a}^{++}$, such a geometric Weyl chamber is defined by a unique element $h_{xy}M \in G/M$ such that $h_{xy}o=x$ and $h_{xy}  e^{ \adis(x,y) } o = y$.
\end{fait}

\paragraph{Jordan projection}
Denote by $\Sigma^+$ the subset of roots which take positive values in the positive Weyl chamber. 
It allows to define the following nilpotent subalgebras $\frak{n}:= \oplus_{\alpha \in \Sigma^+} \frak{g}_{\alpha}$ and $\frak{n}^-=\oplus_{\alpha \in \Sigma^+} \frak{g}_{-\alpha}$.
Denote by $N:= \exp (\frak{n})$ and $N^-:=\exp(\frak{n}^-)$ two maximal unipotent subgroups of $G$.

By Jordan decomposition, every element $g \in G$ admits a unique decomposition $g=g_e g_h g_u$ where $g_e, g_h$ and $g_u$ commute and such that $g_e$ (resp. $g_h, g_u$) is conjugated to an element in $K$ (resp. $A^+$, $N$). 
The element $g_e$ (resp. $g_h$, $g_u$) is called the \emph{elliptic part} (resp. \emph{hyperbolic part}, \emph{unipotent part}) of $g$.

\begin{defin}\label{def-Jordan-proj}
For any element $g \in G$, there is a unique element $\lambda(g) \in \frak{a}^+$ such that the hyperbolic part $g_h$ is conjugated to $\exp (\lambda(g)) \in A^+$.
The map $\lambda : G \rightarrow \frak{a}^+$ is called the \emph{Jordan projection}.

\noindent Any element $g \in G$ such that $\lambda(g) \in \frak{a}^{++}$ is called \emph{loxodromic.} Non loxodromic elements are called \emph{singular}.

\noindent Denote by $G^{lox}$ (resp. $G^{sing}$) the set of loxodromic (resp. singular) elements of $G$ and for any subset $S\subset G$, denote by $S^{lox}:= S \cap G^{lox}$ (resp. $S^{sing}:=S\cap G^{sing}$).
\end{defin}
Equivalently (Cf. §4 \cite{dang20}), loxodromic elements are conjugated in $G$ to elements in $A^{++}M$.

\paragraph{Asymptotic Weyl chambers}
Denote by $P:=MAN$ and by $\mathcal{F}:=G/P$ the \emph{Furstenberg boundary}.
We recall the interpretation of $\mathcal{F}$ in terms of asymptotic Weyl chambers.

Following the exposition in \cite{dg20}, we introduce the following equivalence relation between geometric Weyl chambers:
$$ g_1 A^+ o \sim g_2 A^+ o \Longleftrightarrow \sup_{a \in A^{++}} \dis_X( g_1 a o, g_2 a o ) < + \infty.$$
Equivalence classes for this relation are called \emph{asymptotic Weyl chambers}.
Denote by $\eta_0$ (resp. $\zeta_0$) the asymptotic Weyl chamber of $A^+o$ (resp. $(A^+)^{-1} o$).
The set of asymptotic Weyl chambers identifies with the Furstenberg boundary (see for instance \cite[Fact 2.5]{dg20} for a proof),  
\begin{equation} \label{eq-asymptotic-furstenberg}
 \mathcal{F} \simeq \big( G. (A^+o)/\sim  \big) \simeq K/M \simeq K. \eta_0 .
\end{equation}
Since $MAN^-$ is also a minimal parabolic subgroup of $G$, it is conjugated to $P=MAN$.
Choose $k_-\in K$ such that $MAN^-:=k_-(P)k_-^{-1}$ and set $\zeta_0:=k_- \eta_0$.
By definition, $\mathrm{Stab}_G(\eta_0)=P$ and $\mathrm{Stab}_G(\zeta_0) = MAN^-$.

\begin{rem}
Note that one may choose in this particular case for $k_-$ an element in $N_K(A)$ such that $k_- A^+ k_-^{-1} = (A^+)^{-1}$.

\noindent The more general construction is detailed in the following.
Let $\Theta\subset \Pi$ be a subset of simple roots and let $\calW_\Theta$ be the Weyl subgroup generated by reflections $s_\alpha$ for $\alpha\in \Theta^\complement$.
Then the standard parabolic subgroups of $G$ may be parameterized by $P_\Theta:= P \calW_\Theta P$ where $P_\Pi$ is the Borel subgroup.
Here we take the reverse of the Bourbaki convention \cite{bourbaki_lie_2004}.

\noindent Denote by $\tau$ the Cartan involution of $G$ (Cf. \cite{helgason1978differential}): it is an automorphism of $G$ that acts on $\frak{a}$ by $-id$ and on $A$ by $a\mapsto a^{-1}$.
For $\mathrm{SL}(d,\mathbb{R})$, the Cartan involution is the automorphism $g\mapsto {}^t(g^{-1})$.
The involution $\tau$ induces an involution of the set of simple roots $\iota:\Pi\to\Pi$, such that for all subset $\Theta \subset \Pi$, the parabolic subgroup $\tau(P_\Theta)$ is conjugated to $P_{\iota(\Theta)}$. 
Denote by $P_\Theta^-:=\tau(P_\Theta)$.  
In particular, for the Borel subgroup, the parabolic $P^-=MAN^-$ is a conjugate of $P$. 
\end{rem}

In the remainder of the article, we identify $G. (A^+ o ) / \sim$ with $\mathcal{F}$ and $G . (A^+ o)$ with $G/M$.
We recall that a geometric Weyl chamber is uniquely determined by its base point in $X$ and the asymptotic Weyl chamber it represents.
\begin{fait}\label{fait-asymptotic-point}
The following $G$-equivariant map is a diffeomorphism:
\begin{align*}
    G/M & \overset{\sim}{\longrightarrow} X \times \mathcal{F} \\
    gM & \longmapsto (go, g \eta_0).
\end{align*}
\end{fait}
\noindent For every $(x,\xi) \in X \times \mathcal{F}$, we denote by $g_{x,\xi} M \in G/M$ the geometric Weyl chamber of base point $x$ asymptotic to $\xi$.

\paragraph{Busemann and Iwasawa cocycle}
For every $\xi \in \mathcal{F}$ and $g \in G$, consider, by Iwasawa decomposition $KAN$, the unique element $\sigma(g,\xi) \in \frak{a}$, called the \emph{Iwasawa cocycle}, such that if $k_\xi \in K$ satisfies $k_\xi \eta_0 = \xi$, then
\begin{equation}\label{eq-iwasawa-cocycle}
    g k_\xi \in K \exp( \sigma (g, \xi) ) N .
\end{equation}
The \emph{cocycle relation} holds (Cf. \cite[Lemma 5.29]{benoist-quint}) i.e. for all $g_1, g_2 \in G$ and $\xi \in \mathcal{F}$, then
\begin{equation}\label{eq-relation-Icocycle}
\sigma(g_1 g_2,\xi) = \sigma(g_1, g_2 \xi) + \sigma(g_2 , \xi).    
\end{equation}
Note that restricted to $K \times \mathcal{F}$, the Iwasawa cocycle is the zero function, i.e. for every $k \in K$ and $\xi \in \mathcal{F}$, then $\sigma(k,\xi)=0$.
This motivates the following definition of the Busemann cocycle for two points of $X$ and an asymptotic Weyl chamber. 
\begin{defin}\label{defin-buseman}
For every $x,y \in X$ and $\xi \in \mathcal{F}$, we define the \emph{Busemann cocycle} by 
$$ \beta_{\xi} (x,y):= \sigma (h_x^{-1} h_y, h_y^{-1} \xi ) $$
independently of the choice of $h_x, h_y \in G$ such that $h_x o = x$ and $h_y o = y$.
\end{defin}
Remark that for every $x,y \in X$ and $\xi \in \mathcal{F}$, for all $g \in G$ and all $z \in X$,
\begin{align}\label{eq-relationBcocycle}
    \beta_{ \xi}(x, y )  &= \beta_{g\xi}(gx,gy) \\
    \beta_\xi (x,y) &= \beta_\xi(x,z) + \beta_\xi(z,y).
\end{align}
The first equation is the $G$-invariance of the formula, whereas the second is due to the cocycle relation of the Iwasawa cocycle.

\paragraph{Transverse points in $\cal{F}$}
The subset of \emph{ordered transverse pairs} of $\cal{F} \times \cal{F}$ is defined by the $G$-orbit
\begin{equation} \label{eq-defin-transverse-pts}
    \cal{F}^{(2)}:= \lbrace (g \eta_0,g\zeta_0) \; \vert \; g \in G \rbrace. 
\end{equation}
We say that $\xi, \eta \in \cal{F}$ are \emph{opposite} or \emph{transverse} if $(\xi,\eta) \in \cal{F}^{(2)}$.

In terms of asymptotic Weyl chambers, $\xi, \eta \in \cal{F}$ are opposite when there exists a geometric Weyl chamber $g. (A^+ o)$ asymptotic to $\xi$ such that $g. ( (A^+)^{-1}o )$ is asymptotic to $\eta$.
Note that (Cf. §3.2 \cite{thirion}) we have the following identifications
$$ \cal{F}^{(2)} \simeq G/AM .$$
\begin{defin}\label{defin-max-flat}
For every $(\xi,\eta)\in \cal{F}^{(2)}$, for any choice $g_{\xi,\eta} \in G$ such that $g_{\xi,\eta} (\eta_0,\zeta_0)= (\xi,\eta)$, we denote by 
$$(\xi \eta)_X:= g_{\xi,\eta} . (Ao)$$
the associated maximal flat in the symmetric space $X$.

For every $(x,\xi) \in X \times \cal{F}$,  we denote by 
$\xi_x^\perp \in \cal{F}$ the unique opposite point to $\xi$ such that $x \in (\xi \xi_x^\perp )_X$.
Equivalently, $\xi_x^\perp:= g_{x,\xi}\zeta_0$, where $g_{x,\xi}M\in G/M$ corresponds (Cf. Fact \ref{fait-asymptotic-point}) to the geometric Weyl chamber of base point $x$ asymptotic to $\xi$.
\end{defin}
\begin{rem}\label{rem-zeta-eta}
Note that $(\zeta_0)_o^\perp = \eta_0$ and vice-versa.
\end{rem}

\paragraph{Hopf coordinates}
Let $\cal{H}$ be the Hopf coordinate map of $G/M$ (Cf. \cite[Chapter 8, §8.G.2]{thirion} or \cite{dg20})
\begin{align*}
\cal{H}: G/M &\rightarrow \calF^{(2)}\times \frak a\\
gM &\mapsto (g\eta_0,g\zeta_0,\sigma(g,\eta_0)).
\end{align*}
Hopf coordinates are left $G$-equivariant and right $A$-equivariant in the following sense: 
\begin{itemize}
    \item[(i)] the left action of $G$ on $G/M$ reads in those coordinates equivariantly on $\cal{F}^{(2)}$ and using the Iwasawa cocycle as follows. For all $h\in G$ and $(\xi,\eta,Y)\in\cal{F}^{(2)}\times \frak a$,
\begin{equation}\label{equ-hxi}
h(\xi,\eta,Y)=(h\xi,h\eta, Y+\sigma(h,\xi)).
\end{equation}
    \item[(ii)] the right action of $A$ on $G/M$ reads for all $(\xi,\eta,Y) \in \cal{F}^{(2)} \times \frak{a}$ and $a \in A$ by keeping the first two coordinates constant and translating the last one by $\log(a)$ 
    $$\cal{H}( \cal{H}^{-1}(\xi,\eta,Y) a  ) = (\xi,\eta, Y + \log(a)) .$$
\end{itemize}

Using the geometric Weyl chamber interpretation and the Busemann cocycle notations,
the Hopf map reads the same as in Roblin's work \cite{roblin}:
\begin{equation}\label{equ_hopfmap}
\begin{split}
    X \times \calF & \longrightarrow \calF^{(2)} \times \frak{a} \\
    (x,\xi) & \longmapsto (\xi, \xi_x^\perp , \beta_{\xi}(o,x) ).
\end{split}
\end{equation}
This translated map is also left $G$-equivariant in the sense that for every $g\in G$ and every $(x,\xi) \in X \times \calF$, using the cocycle relation \eqref{eq-relationBcocycle}, the element $(gx,g\xi)$ has Hopf coordinates
$$(g \xi, g \xi_x^\perp , \beta_{g \xi}(o,go) + \beta_{\xi}(o,x) ).$$
Note that $\beta_{g \xi} (o,go) = \sigma(g,\xi)$, therefore the notations are consistent.

\subsection{Disintegration of the Haar measure}
Patterson--Sullivan measures were generalized to the higher rank setting in \cite{albuquerque99}, \cite{quintMesuresPattersonSullivanRang2002}. 
We follow Thirion's \cite[Chapter 9 §9.e]{thirion} construction of higher rank Patterson--Sullivan measures.
He dealt with the space of Weyl chambers of $\mathrm{SL}(d,\mathbb{R})$, it turns out that his method is more general.

We start by the so-called Patterson densities.
For $x\in X$, let $K_x$ be the stabilizer group of $x$ in $G$. Let $\mu_x$ be the unique $K_x$ invariant probability measure on the Furstenberg boundary $\calF$. Then we have for $g\in G$ and $x\in X$
\begin{equation}\label{equ.gmux}
    g_*\mu_x=\mu_{gx}, 
\end{equation}
where $g_*\mu_x$ is the pushforward of $\mu_x$ under the $g$ action. This relation holds because the stabilizer of $g_*\mu_x$ is given by $gK_xg^{-1}=K_{gx}$.
Let $\rho=\frac{1}{2}\sum_{\alpha\in\Sigma}\alpha$ be the half of the sum of positive roots with multiplicities. By \cite[Lemma 6.3]{quintMesuresPattersonSullivanRang2002} or \cite[I 5.1]{helgasonGroupsGeometricAnalysis2000}, for $g$ in $G$ and every $\xi \in \calF$ we have
\begin{equation}\label{equ.gmu0}
    \frac{ \dd g_*\mu_o}{\dd\mu_o}(\xi)=e^{-\rho \sigma(g^{-1},\xi)}= e^{-\rho  \beta_\xi(go,o)},
\end{equation}
which is a $G$ quasi-invariant measure. 
Note that in the rank one notations, we replaced the critical exponent with the linear form $\rho$ and apply it with the higher rank Busemann function.
Then we will introduce the Gromov product to obtain a $G$-invariant measure on $\calF^{(2)}$.
\begin{defin}\label{defin_gromov}
For a pair $(\xi,\eta)\in\calF^{(2)}$, we associate it with the unique element in the Lie algebra $\frak a$ such that for all weights $\chi^\alpha$
\[\chi^\alpha(\xi|\eta)_o:
=  -\log\frac{|\varphi(v)|}{\|\varphi\|\|v\|}, \]
where $v\in V^\alpha-\{0\}$ is a representative of $x^\alpha(\xi)$ and $\varphi$ is a non zero linear form such that $\ker \varphi = x^{\alpha}(\eta_o^\perp)^\perp$.
\end{defin}

This definition already appears in \cite[Section 8.10]{bochiAnosovRepresentationsDominated2019}, \cite[Section 4]{sambarinoOrbitalCountingProblem2015} for semisimple Lie groups and \cite{thirion} for $\slr$. 
Our definition of $\delta$ seems different from the one in \cite{bochiAnosovRepresentationsDominated2019}, \cite{sambarinoOrbitalCountingProblem2015}. By using the correspondence between linear forms and hyperplanes for Euclidean spaces, we can verify that they are the same.
An important property is that \cite[Lemma 4.12]{sambarinoOrbitalCountingProblem2015}: for all $g\in G$ and $(\xi,\eta)\in\calF^{(2)}$, we have 
\begin{equation}\label{equ.gromov}
    (g\xi|g\eta)_o-(\xi|\eta)_o=\iota\sigma(g,\xi)+\sigma(g,\eta),
\end{equation}
where $\iota$ is the inverse involution on $\frak a$.
We also define the Gromov product at other points $x$ in $X$ by $G$-invariance, by setting
\[(\xi|\eta)_x=(h_x^{-1}\xi|h_x^{-1}\eta)_o, \]
where $h_x$ is some element such that $h_xo=x$.
Since by \eqref{equ.gromov}, the Gromov product at $o$ is left $K$-invariant, this definition is independent of the choice of $h_x$. 
In \cite{bochiAnosovRepresentationsDominated2019}, the authors proved that the Gromov product $(\xi|\eta)_o$ in norm is almost the same as the distance between $o$ and the maximal flat $(\xi\eta)_X \subset X$.
\begin{lem}\label{lem-xietao}
\cite[Proposition 8.12]{bochiAnosovRepresentationsDominated2019}
There exist $C_3>1,C'>0$ such that for any $(\xi,\eta)\in \calF^{(2)}$, we have
\begin{equation*}
    \frac{1}{C_3}\|(\xi|\eta)_o\|\leq \dis_X(o,(\xi\eta)_X)\leq C_3\|(\xi|\eta)_o\|+C'.
\end{equation*}
\end{lem}
\noindent By $G$-invariance, we deduce that for every $x \in X$ and $(\xi,\eta) \in \calF^{(2)}$
$$ \frac{1}{C_3} \Vert (\xi | \eta)_x \Vert \leq \dis_X(x, (\xi \eta)_X) \leq 
C_3 \Vert (\xi | \eta)_x \Vert + C' .$$

For all $x\in X$ and $(\xi,\eta)\in\calF^{(2)}$, we define the $(0,1]$-valued function
\[f_x(\xi,\eta)=\exp(-\rho(\xi|\eta)_x). \]
We define measures $\nu_x$ on $\calF^{(2)}$ by
\begin{equation}\label{equ.nu}
    \dd\nu_x(\xi,\eta)=\frac{\dd\mu_x(\xi)\dd\mu_x(\eta)}{f_x(\xi,\eta)}.
\end{equation}

\begin{prop}\label{prop-PS-BMS}
For all $x\in X$, the measure $\nu_x$ is $G$-invariant and equal to $\nu_o$. We denote it by $\nu$.
\end{prop}
\noindent In the hyperbolic case, the measures $\mu_x$ are called Patterson-Sullivan and $\nu \otimes Leb_{\mathbb{R}}$ is the Bowen-Margulis-Sullivan measure.
In the $\slr$ case, Thirion \cite{thirion} gave a construction of this measure and proved those properties. 
We include a proof for completeness.
\begin{proof}
By \eqref{equ.gromov}, for all $x\in X$, all $(\xi,\eta)\in \calF^{(2)}$ and every $h_x \in G$ such that $h_x o =x$
\begin{align*}
  f_x(\xi,\eta)=f_o(h_x^{-1}\xi,h_x^{-1}\eta) &= f_o(\xi,\eta)\exp(-\rho(\iota\sigma(h_x^{-1},\xi)+\sigma(h_x^{-1},\eta))) \\
  &= f_o(\xi,\eta)\exp(-\rho(\iota\beta_\xi(x,o)+\beta_\eta(x,o)))
\end{align*}
On the other hand,
\[\frac{\dd\mu_x}{\dd\mu_o}(\xi)=\frac{\dd(h_x)_*\mu_o}{\dd\mu_o}(\xi)=e^{-\rho\sigma(h_x^{-1},\xi)}. \]
We obtain the same formula for $\eta$.
Combining the above two equations and using that $\rho\iota\sigma(h_x^{-1},\xi)=\rho\sigma(h_x^{-1},\xi)$, we obtain that 
\[\nu_x=\nu_o. \]
By definition of the Gromov product, we have for all $g \in G$
\[f_x(g\xi,g\eta)=f_{g^{-1}x}(\xi,\eta). \]
By equation \eqref{equ.gmux} and using that $\nu_{g^{-1}x}=\nu_x$,
\begin{align*}
    \dd\nu_x(g\xi,g\eta)=\frac{\dd\mu_x(g\xi)\dd\mu_x(g\eta)}{f_x(g\xi,g\eta)}=\frac{\dd\mu_{g^{-1}x}(\xi)\dd\mu_{g^{-1}x}(\eta)}{f_{g^{-1}x}(\xi,\eta)}=\dd\nu_{g^{-1}x}(\xi,\eta)=\dd\nu_x(\xi,\eta).
\end{align*}
Hence $\nu_x$ is $G$-invariant.
\end{proof}

With this $G$-invariant measure $\nu$ on $\calF^{(2)}$, we deduce the disintegration of the Haar measure on $G/M$ along Hopf coordinates.

\begin{prop}\label{prop-disintegration}
The product measure $\nu\otimes Leb$ on $\calF^{(2)}\times\frak a$ is a disintegration in Hopf coordinates of a Haar measure on $G/M$.
\end{prop}

\begin{proof}
The product measure $\nu\otimes Leb$ is $G$-invariant by Proposition \ref{prop-PS-BMS} and by Hopf coordinates, it is a measure on $G/M$. 
So it is a Haar measure on $G/M$.
\end{proof}

\section{An effective angular equidistribution}
We present the equidistribution result of Gorodnik--Nevo \cite{gorodnikCountingLatticePoints2012}.
First, for elements whose Cartan projection is in the interior of the Weyl chamber, we build a pair of angular points in $\calF$. 
In rank one, these points are the endpoints of the half-geodesics based at the origin point and going through the image and the inverse image of the origin.

Then we state Gorodnik--Nevo's equidistribution theorem, where they only require the lattice to be irreducible.
Since we apply their result to a ball shaped domain and to a parallelotope domain.
Finally, we give an equivalent of the  

\subsection{Cartan regular isometries}
Recall that by Cartan decomposition, for every element $g \in G$ there exist $k,l\in K$ and a unique element $\underline{a}(g) \in \frak{a}^+$ such that $g=k \exp (\underline{a}(g)) l^{-1}$. 
Note that $k$ and $l$ are defined up to right multiplication by elements in $Z_K( \exp(\underline{a}(g)))$.

\begin{defin}\label{defin-cartan-regular}
For all $x \in X$, we denote by $\underline{a}_x: G \rightarrow \mathfrak{a}^+$
the map that assigns to every $g \in G$ the $\frak{a}^+$-distance between $x$ and $gx$, i.e. $\underline{a}_x(g):= \adis(x,gx).$ We say that $g$ is \emph{$x$-cartan regular} if $\underline{a}_x(g) \in \frak{a}^{++}$.

Let $g$ be an $x$-cartan regular element, consider $h,h'\in G$ such that $ho=h'o=x$ with $he^{\underline{a}_x(g)}o=gx$ and $h'e^{ \underline{a}_x(g^{-1})}o=g^{-1}x$.
We set $g_x^+:=h \eta_0$ and $g_x^-:=h'\eta_0$. In particular, when $x=o$, we can take $h=k$ and $h'=lk_\iota$.
\end{defin}
Note that for every $g\in G$ we have $\underline{a}_x(g)= \underline{a}( h_x^{-1} g h_x)$, independently of the choice of $h_x \in G$ such that $h_xo=x$.
Furthermore, provided that $g$ is $x$-cartan regular, 
\begin{equation}\label{equ.gammax}
    g_x^{\pm}=h_x(h_x^{-1} g h_x)_o^{\pm}.
\end{equation}
Remark that $(x,g_x^+) \in X \times \cal{F}$ (resp. $(x,g_x^-)$) is the unique geometric Weyl chamber based on $x$ containing $gx$ (resp. $g^{-1}x$).
In the $\mathrm{PSL}(2,\mathbb{R})$ case, an element $g$ is $x$-cartan regular when $gx\neq x$, then $g_x^+ \in \partial \mathbb{H}^2$ (resp. $g_x^-$) is the asymptotic endpoint of the half geodesic based on $x$ going through $gx$ (resp. $g^{-1}x$).
Recall a lemma about comparing Cartan projections.
\begin{lem}[\cite{kasselCorank08} Lemma 2.3]\label{lem-cartan-diff}
For all $h,h'\in G$ we have the following inequalities,
$$ \Vert \underline{a}( h h') - \underline{a}(h) \Vert \leq  \Vert \underline{a}(h') \Vert
\text{ and }\Vert \underline{a}( h' h) - \underline{a}(h) \Vert \leq  \Vert \underline{a}(h') \Vert. $$
\end{lem}
We translate it using our notations.
\begin{lem}\label{lem-cartan-points}
For all $g\in G$, every $x,y\in X$, the following bound holds:
$$ \Vert \underline{a}_x(g)-\underline{a}_y(g) \Vert \leq 2 \dis_X(x,y) .$$ 
\end{lem}
\begin{proof} 
Let $x,y\in X$ and choose $h_x,h_y \in G$ such that $h_x o =x$ and $h_y o =y$.
We compare the Cartan projection of $h= h_y^{-1} g h_y$ to the Cartan projection of its conjugate by $h'=h_y^{-1} h_x$. Using that $\Vert \underline{a}(h') \Vert =\Vert \underline{a}(h'^{-1}) \Vert $ and Lemma \ref{lem-cartan-diff} we get
$$ \Vert \underline{a}_x(g) - \underline{a}_y(g)  \Vert \leq 2 \Vert \underline{a}(h_y^{-1} h_x) \Vert .$$
Since $\Vert \underline{a}(h_y^{-1} h_x) \Vert= \dis_X(x,y)$, we deduce the Lemma.
\end{proof}

\subsection{Angular distribution of Lattice points}
We introduce here some subsets of $G$. They will be used to obtain the main term and the exponentially decaying error term in our main Theorems \ref{corol-counting}, \ref{thm-introequid}. 

\noindent For $t>1$, let $D_t:= K \exp(\frakD_t)K$ and where $\frakD_t$ may be one of the two following types of domains
\begin{align*}
     \text{Ball domain } & B_{\frak{a}}(0,t)\cap \frak a^+,\\
     \text{Parallelotope domain } & t  \calP \text{ where }
     \calP= \prod_{\beta \in \Pi} [0,a_\beta] \text{ with } 0 < a_\beta \;,\; \forall \beta \in \Pi.
\end{align*}

Denote the subset of Cartan-regular elements by
$$D_t^{reg}:= D_t\cap (K \exp(\frak a^{++} )K).$$
For all $x\in X$, we consider similar sets
$$D_t(x):= h_x D_t h_x^{-1},$$
$$D_t^{reg}(x):= h_x D_t^{reg} h_x^{-1},$$
where $h_x\in G$ is any element such that $h_x o =x$. These sets do not depend on the choice of $h_x$.
\begin{theorem}\label{theo-gorodnik-nevo}
\hypG
Let $\Gamma <G$ be an irreducible lattice. There exist $\kappa>0$ and $C_4>0$. 
Let $x \in X$.
Then for all Lipschitz test functions $\psi \in Lip(\cal{F}\times\calF)$, there exists $E(t,\psi, x)=O(Lip(\psi) C_x  \vol(D_t)^{-\kappa})$ when $t > C_4 \dis_X(o,x)$ such that

$$ \frac{1}{ \vol(D_t)} \sum_{\gamma \in D^{reg}_t(x)\cap \Gamma} \psi(\gamma_x^+,\gamma_x^-) = \frac{1}{ \vol(\Gamma \backslash G)}   \int_{ \cal{F}\times\calF  } \psi \; \dd \mu_x\otimes\mu_x + E(t,\psi, x) .$$
\end{theorem}

This is due to Gorodnik-Nevo in \cite{gorodnikCountingLatticePoints2012}. We include the proof of this version for Lipschitz functions in the appendix. The main term is due to Gorodnik-Oh in \cite{gorodnik_orbits_2007}.
\begin{rem}
    The above formula should also work for more general Parallelotope domains $\calP'= \prod_{\beta \in \Pi} [a_\beta,b_\beta]$ with $0 \leq a_\beta<b_\beta \;,\; \forall \beta \in \Pi$.
    They are unions of $\vert\Pi\vert+1$ parallelotopes defined as above where only one of them has dominant volume growth, the other volumes grow exponentially slower. 
\end{rem}

\subsection{Volume growth of the ball domain}\label{sec:volume-ball}

Applying the Harish-Chandra formula for $D_t$ (see \cite[Chapter I Theorem 5.8]{helgasonGroupsGeometricAnalysis2000}) yields
\begin{equation}\label{equ:borel-harish-chandra}
\vol (D_t) = \int_{\frakD_t} \Pi_{\alpha \in \Sigma^+ }  \sinh(\alpha (Y) )^{m_\alpha } \dd Leb (Y),
\end{equation}
where $m_\alpha\in\N$ is the multiplicity of the positive root $\alpha$.
Now denote by $\rho$ the half of the sum of positive roots with multiplicities and set $\delta_0:=2\max_{Y\in B_{\frak a}(0,1)}\rho(Y)$.
Such choice will allow a uniform volume estimate (Cf. Lemma \ref{lem_volrest_proof} in the Appendix).

\begin{equation}\label{volume_D_t}
\vol(D_t) \sim t^{\frac{ \dim  A-1}{2}} e^{\delta_0 t}.
\end{equation}

\subsection{Volume growth of the parallelotope domain}\label{sec:volume-paral}
In this subsection, we denote by $D_t:=K \exp({t \mathcal{P}}) K$ where $\mathcal{P}$ is a parallelotope domain defined as above.
\begin{lem}\label{lem-volume-paral}
Let $\mathcal{P}:= \prod_{\beta \in \Pi} [0,a_\beta]$ where $a_\beta>0$ for all $\beta\in \Pi$ and $\delta_\calP:= \sup_{\calP} 2 \rho $.
There exits a positive $\delta^-<\delta_\calP$ and $C_G>0$ such that, as $t$ goes to $+ \infty$,
\begin{equation}\label{volume_tP}
\vol(K \exp{(t \mathcal{P})} K) = C_G e^{\delta_\calP t} + O(e^{\delta^- t})  .
\end{equation}
\end{lem}
\begin{proof}
Starting with the Harish-Chandra density, we first develop the hyperbolic sine and factor by $1/2$, then we further develop the products of sums of exponentials using that $2\rho = \sum_{\alpha\in \Sigma^+} m_\alpha \alpha$.  
$$\prod_{\alpha \in \Sigma^+} \sinh(\alpha(Y))^{m_\alpha} 
= 2^{- \sum_{\alpha\in \Sigma^+}m_\alpha} \prod_{\alpha \in \Sigma^+} \Big( e^{\alpha(Y)} - e^{-\alpha(Y)}\Big)^{m_\alpha} = 2^{- \sum_{\alpha\in \Sigma^+}m_\alpha} \bigg(
e^{ 2 \rho(Y)} + \sum_{\omega \in R} p_\omega e^{\omega(Y)} \bigg)$$
where $R = \Big\lbrace 2 \rho - \sum_{\alpha\in \Sigma^+} 2k_\alpha \alpha(Y) \; \vert \; 0 \leq k_\alpha \leq m_\alpha \text{ and } k_\alpha \in \mathbb{Z}_+ \Big\rbrace 
\setminus \lbrace 2 \rho \rbrace$  is a subset of linear forms of $\mathfrak{a}$ and the coefficients $p_\omega \in \mathbb{Z}$ are integers. 
Hence, by the Harish-Chandra formula 
$$
    \vol(K \exp{(t \mathcal{P})} K) = 2^{- \sum_{\alpha\in \Sigma^+}m_\alpha} \bigg( \int_{t\calP}  e^{ 2 \rho(Y)} \dd Leb(Y) + \sum_{\omega \in R} p_\omega \int_{t\calP} e^{\omega(Y)} \dd Leb(Y) \bigg)
$$
Recall that the set of simple roots $\Pi$ form a basis of $\mathfrak{a}^*$ and that any positive root is a vector of non-negative integers in this basis.
For any element $\omega \in R\cup \lbrace 2\rho \rbrace$, denote by $\big( n_\beta(\omega) \big)_{\beta\in \Pi}$ its integer coefficients.
Denote by $J_\Pi$ the Jacobian between the dual basis of $\Pi$ and the canonical basis i.e. such that $\dd Leb(Y) = J_\Pi \prod_{\beta \in \Pi} \dd y_\beta$.
Then by a change of variables, using that $\calP=\prod_{\beta \in \Pi} [0,a_\beta]$ and splitting the exponential in the dual basis we deduce that 
\begin{align*}
    \vol(K \exp{(t \mathcal{P})} K) 
    &= \frac{J_\Pi}{2^{\sum_{\alpha\in \Sigma^+}m_\alpha} } \bigg( \prod_{\beta \in \Pi} \int_0^{t a_\beta} e^{n_\beta(2 \rho) y_\beta} \dd y_\beta 
    + \sum_{\omega\in R} p_\omega  \prod_{\beta \in \Pi} \int_0^{t a_\beta} e^{n_\beta(\omega) y_\beta} \dd y_\beta \bigg) \\
    &= \frac{J_\Pi}{2^{\sum_{\alpha\in \Sigma^+}m_\alpha}} \bigg( 
        \prod_{\beta \in \Pi} \frac{ e^{t n_\beta(2\rho) a_\beta}-1}{ n_\beta(2\rho) }
        + \sum_{\omega\in R} p_\omega \prod_{\beta \in \Pi} \frac{ e^{t n_\beta(\omega) a_\beta} -1}{n_\beta(\omega)}
        \bigg). 
\end{align*}
Finally, set $C_G:=\frac{J_\Pi}{2^{\sum_{\alpha\in \Sigma^+}m_\alpha} \prod_{\beta\in \Pi}n_\beta(2\rho)}$ and 
$\delta_\calP:= \sum_{\beta\in \Pi}n_\beta(2\rho) a_\beta$.
Note that $\delta_\calP = \sup_{\calP} 2\rho$.
By developing the products into sums, the main term is $C_G e^{\delta_\calP t}$ and the remaining terms are $O(e^{\delta^-t})$ for some $\delta^- \in (0,\delta_\calP)$ determined by the simple roots and $(a_\beta)_{\beta\in \Pi}$.
\end{proof}

\subsection{Regularity of volume growths}

\begin{lem}\label{lem-vollip}
The function $t\mapsto \log  \vol(D_t)$ is uniformly locally Lipschitz for $t>1$. 
\end{lem}
The proof is given in Lemma \ref{lem_vollip_proof} in the Appendix. This means that there exists $C>0$ such that for all $0<\epsilon<1$ and $t>1$, we have
\[ \vol(D_{t+\epsilon})\leq e^{C\epsilon} \vol(D_t). \]

\section{Counting almost singular lattice elements}

For $0<s<t$, let 
$$D_t^{s} := \lbrace g \in D_t \; \vert \; \underline{a}(g) \in \overline{B_{\frak{a}}( \partial \frak{a}^+ ,s )  } \rbrace $$
be the set of elements in $D_t$ whose Cartan projection have distance at most $s$ to the boundary of the Weyl chamber.

For all $x\in X$, we define a similar set (independently from the choice of $h_x$ such that $h_x o = x$)
$$D_t^{s}(x):= h_x D_t^{s} h_x^{-1}. $$

We want to obtain estimates for singular elements. 
\begin{lem}\label{volume_reste}
There exists $\epsilon_D>0$ which only depends on $D_1$ such that for every $0<\epsilon<\epsilon_D$, there exists $\kappa(\epsilon)>0$ such that for $t>1$ 
\begin{equation}
\frac{  \vol(D_t^{\epsilon t})}{\vol (D_t) } = O( \vol(D_t)^{-\kappa(\epsilon)} ).
\end{equation}
\end{lem}

\begin{proof}
The proof is similar to Lemma 9.2 and 9.4 in \cite{gorodnikDistributionLatticeOrbits2007}. Let $\frakD_t^s=\{Y\in\frakD_t,\, d(Y,\partial \frak a^+)\leq s \}$. Recall $\frakD_t=\frak a^+\cap B_{\frak a}(0,t)$ for ball domain and $\frakD_t=\calP(c_\alpha)$ for parallelotope domain. Then by \eqref{equ:borel-harish-chandra}, we have
\begin{equation}\label{equ:vol uper}
 \vol(D_t^s)\leq \int _{\frakD_t^s}e^{2\rho(Y)}\dd Y. 
\end{equation}
By Lemma 9.2 in \cite{gorodnikDistributionLatticeOrbits2007}, if $\epsilon$ smaller than some constant $\epsilon_D$, then by the strict convexity of $B_{\frak a}(0,1)$ and $\calP(c_\alpha)$, there exists $\kappa'(\epsilon)>0$ such that
\[
\max_{Y\in \frakD_1^\epsilon}2\rho(Y)\leq \delta_0-\kappa'(\epsilon). \]
So by \eqref{equ:vol uper}, we have $ \vol(D_t^{\epsilon t})\ll t^{\dim  A}e^{(\delta_0-\kappa'(\epsilon))t}$. Due to the asymptotic of $ \vol(D_t)$ \eqref{volume_D_t}, the proof is complete.
\end{proof}

\begin{lem}\label{count-reste}
Let $\Gamma$ be a lattice in $G$, then
for all $t>1$ and $\epsilon<\epsilon_D$,
\begin{equation*}
    \frac{|\Gamma\cap D_t^{\epsilon t}|}{ \vol(D_t)}=O( \vol(D_t)^{-\kappa(\epsilon)}).
\end{equation*}
\end{lem}
\begin{proof}
Let $\epsilon'>0$ be a small constant such that the ball centered at $e$ with radius $\epsilon'$ satisfies $B(e,\epsilon')^2\cap\Gamma=\{e\}$. Then we have
\begin{equation*}
    |\Gamma\cap D_t^{\epsilon t}|\leq \frac{ \vol( B(e,\epsilon')D_t^{\epsilon t})}{ \vol(B(e,\epsilon'))}.
\end{equation*}
By Lemma \ref{lem-cartan-diff}, we have for $h'\in B(e,\epsilon')$ and $h\in D_t^{\epsilon t}$,
\[ \Vert \underline{a}( h' h) - \underline{a}(h) \Vert \leq  \Vert \underline{a}(h') \Vert\leq \ell\epsilon', \]
for some $\ell>0$. Therefore the product set
\[ B(e,\epsilon')D_t^{\epsilon t}\subset D_{t+\ell\epsilon'}^{\epsilon t+\ell\epsilon'}.  \]
Hence we have
\[|\Gamma\cap D_t^{\epsilon t}|\leq  \frac{ \vol(D_{t+\ell\epsilon'}^{\epsilon t+\ell\epsilon'})}{ \vol(B(e,\epsilon'))}, \]
which is $O( \vol(D_t)^{1-\kappa(\epsilon)})$ due to Lemma \ref{lem-vollip} and \eqref{volume_reste}.
\end{proof}

As a corollary, we have
\begin{lem}\label{lem-dtx}
For $0<\epsilon<\epsilon_D/2$, $t>1$ and $x\in X$ with $\dis_X(o,x)<\min\{\frac{\epsilon}{2(1-2\epsilon)},\frac{\kappa(2\epsilon)}{4(1-\kappa(2\epsilon))} \} t$, we have
\begin{equation*}
    \frac{|\Gamma\cap D_t^{\epsilon t}(x)|}{ \vol(D_t)}=O( \vol(D_t)^{-\kappa(2\epsilon)/2}).
\end{equation*}
\end{lem}

\begin{proof}
By Lemma \ref{lem-cartan-points}, we have
\[\| a_o(\gamma)-a_x(\gamma )\|\leq 2 \dis_X(x,o). \]
Therefore by Lemma \ref{count-reste} with $2\epsilon$ we obtain
\begin{align*}
|\Gamma\cap D_t^{\epsilon t}(x)|&\leq |\Gamma\cap D_{t+2 \dis_X(x,o)}^{\epsilon t+2 \dis_X(x,o)}|\ll  \vol(D_{t+2 \dis_X(x,o)})^{1-\kappa(2\epsilon)},
\end{align*}
where we use the hypothesis that $\epsilon t+2 \dis_X(x,o)\leq 2\epsilon( t+2 \dis_X(x,o))$.

By hypothesis, we have
\[ (1-\kappa(2\epsilon))(t+2\dis_X(o,x))\leq (1-\kappa(2\epsilon)/2)t. \]
Then by \eqref{volume_D_t}, we have
\begin{align*}
 \vol(D_{t+2 \dis_X(x,o)})^{1-\kappa(2\epsilon)}=O( \vol(D_t)^{1-\kappa(2\epsilon)/2}).
\end{align*}
The proof is complete.
\end{proof}

As a corollary, combined with Lemma \ref{lem-dtx} we have

\begin{lem}\label{lem-dtxt}
There exist $C_5>0$ and $C>0$ such that if $t> C_5\dis_X(o,x)$, then 
\[ \frac{|\Gamma\cap D_t(x)|}{ \vol(D_t)}\leq C. \]
\end{lem}
\begin{proof}
Due to the definition $C_x=C_1e^{C_0\dis_X(o,x)}$, we know that if $t\gg \dis_X(o,x)$, then by taking $\psi=1$ Theorem \ref{theo-gorodnik-nevo} implies that
\[ |\Gamma\cap D_t^{reg}(x)|\ll  \vol(D_t). \]
For the part $|\Gamma\cap(D_t(x)-D_t^{reg}(x))|$, if $t\gg \dis_X(o,x)$, then we can use Lemma \ref{lem-dtx} to bound it. Combining these two parts, we obtain the lemma.
\end{proof}

\section{A configuration for being loxodromic}

Recall Definition \ref{def-Jordan-proj} that the elements in $G$ of Jordan projection in $\frak{a}^{++}$ are called loxodromic.
Equivalently, loxodromic elements are conjugated to elements in $A^{++}M$.
Let $g\in G^{lox}$ be a loxodromic element, choose $h_g \in G$ such that $h_g^{-1} gh_g \in \exp(\lambda(g)) M$.
Note that $g h_gM= h_g e^{\lambda(g)}M$. 
Denote by $g^+:= h_g \eta_0$ (resp. $g^-:= h_g \zeta_0$) the attracting (resp. repelling) fixed points in $\calF$ for the action of $g$. They are independent of the choice of $h_g$.
Hence for every $Y \in \frak{a}$, in Hopf coordinates 
\begin{equation}\label{eq_gg+}
    g(g^+, g^- , Y)=(g^+,g^-, Y+ \lambda(g)).
\end{equation}

\subsection{The Furstenberg boundary}
\paragraph{Representations of a semisimple Lie group}
Let us first recall a few facts about representations of a semisimple Lie group.
Let $(V,\rho)$ be a representation of $G$ into a real vector space of finite dimension.
For every real character $\chi : \frak a \rightarrow \mathbb{R}$, we denote by
$$ V_\chi := \lbrace v \in V \; \vert \; \rho(u)v=\chi(u)v, \; \forall u \in \frak a \rbrace $$
the associated vector space. 
The set of \emph{restricted weights} is the subset 
$$ \Sigma(\rho):= \lbrace \chi \; \vert \; V_\chi \neq \{ 0\} \rbrace. $$
They are partially ordered using the positive Weyl chamber as follows.
$$ (\chi_1 \leq \chi_2) \Leftrightarrow ( \chi_1 (u) \leq \chi_2(u), \; \forall u \in \frak a^+  ). $$
When the representation $\rho$ is irreducible, the set of restricted weights admits a maximum, called the \emph{maximal restricted weight}.
The irreducible representation $\rho$ is \emph{proximal} when the associated vector space of the maximal restricted weight is a line.

\paragraph{Restricted weights of the fundamental representations}
For the adjoint representation, the set of restricted weights coincides with the set of restricted roots $\Sigma$. Denote by $\Sigma^+$ the set of positive restricted roots and by $\Pi \subset \Sigma^+$ the set of simple roots. 
Tits (\cite[Lemma 6.32]{benoist-quint}) proved that for every $\alpha\in \Pi$, there exists an irreducible and proximal representation $(\rho_\alpha,V^\alpha)$ of $G$ such that the restricted weights are in
\begin{equation}\label{eq-restricted-weights}
\bigg\lbrace \chi^\alpha,\; \chi^\alpha - \alpha ,\; \chi^\alpha -\alpha - \sum_{\beta \in \Pi } n_\beta \beta \; \bigg\vert \; n_\beta \in \mathbb{Z}_+ \bigg\rbrace.
\end{equation}
Furthermore, the maximal weights $(\chi^\alpha)_{\alpha\in \Pi}$ of these representations form a basis of $\frak{a}^*$.

\paragraph{Distances in the projective space}
For every $\alpha\in \Pi$, we choose a Euclidean norm $\Vert . \Vert$ on $V^\alpha$ such that the elements in $\rho_\alpha(A)$ (resp. $\rho_\alpha(K)$) are symmetric (resp. unitary).
Note that $\Vert \rho_\alpha(a) \Vert=\exp( \chi^\alpha(\log a))$  for all $a\in A^+$.
Abusing notation, we denote by $\Vert . \Vert$ the induced  Euclidean norm on $V^\alpha \wedge V^\alpha$.
Remark that for all $a\in A^+$,
\begin{equation}\label{norm-wedge}
     \Vert \wedge_2 \rho_\alpha(a) \Vert =\exp( (2\chi^\alpha- \alpha)\log a) .
\end{equation}
We define the distance in the projective space for all $x,y\in\mathbb{P}(V^\alpha)$ as follows,
\begin{equation}\label{defin-dist-proj}
\dis_\alpha (x,y):= \frac{\Vert v_x \wedge v_y \Vert}{\Vert v_x \Vert. \Vert v_y\Vert}
\end{equation}
independently of the choice of $v_x, v_y \in V$ such that $x=\mathbb{R}v_x$ and $y=\mathbb{R}v_y$.
Note that this distance is equivalent to the induced Riemannian distance on $\mathbb{P}(V^\alpha)$, since we are taking the sine of the angle in $[0,\pi/2]$ between two lines.
For all $x \in \mathbb{P}(V^\alpha)$ and $\varepsilon \in (0,1]$, denote by $B(x,\varepsilon)$ the ball centered at $x$ of radius $\varepsilon$ for this distance.

Denote by $x_+^\alpha \in \mathbb{P}(V^\alpha)$ the projective point corresponding to the eigenspace for the maximal restricted weight $\chi^\alpha$. 
Since $\rho_\alpha(A)$ are symmetric endomorphisms for the Euclidean norm on $V^\alpha$, the orthogonal hyperplane to $x_+^\alpha$ is $\rho_\alpha(A)$-invariant and abusing notations we write 
$$(x_+^\alpha)^\perp = \oplus_{\chi \in \Sigma(\rho_\alpha) \setminus \lbrace \chi^\alpha \rbrace } V^\alpha_{\chi}.$$
For all projective point $y\in \mathbb{P}(V^\alpha)$, we denote by $y^\perp \subset V^\alpha$ the orthogonal hyperplane and by $\varphi_y \in (V^\alpha)^*$ a linear form such that $\ker \varphi_y =y^\perp$. 
For all $x,y \in \mathbb{P}(V^\alpha)$, we define (independently of the choice of non-zero $v_x \in x$)
\begin{equation}\label{defin-delta-proj}
\delta_\alpha(y,x):= \frac{ \vert \varphi_y(v_x) \vert }{ \Vert \varphi_y \Vert. \Vert v_x \Vert}.
\end{equation}
By properties of the norms and distances on the projective space, the previous function is symmetric and for all $x,y\in \mathbb{P}(V^\alpha)$,
\begin{equation}\label{defin-delta-dist}
\delta_\alpha(y,x)= \delta_\alpha(x,y)= \dis_\alpha (y^\perp, x) = \dis_\alpha (y, x^\perp).
\end{equation}
Hence $\dis_\alpha (x_+^\alpha, (x_+^\alpha)^\perp)=1.$
For all $\varepsilon>0$, denote by 
$\cal{V}_\varepsilon ( (x_+^\alpha)^\perp  )^\complement := \lbrace y \in \mathbb{P}(V^\alpha) \;\vert \; \delta_\alpha(y, x_+^\alpha) \geq \varepsilon \rbrace.$ 
We give a proof of the following classical dynamical lemma for completeness.
\begin{lem}\label{lem_rhoa-action}
Let $\varepsilon >0$ and $a \in A^{+}$. Assume there exists $\alpha\in \Pi$ such that $\alpha (\log a) \geq -2\log (\varepsilon)$.
Then $\rho_\alpha (a) \cal{V}_\varepsilon ((x_+^\alpha)^\perp )^\complement  \subset B(x_+^\alpha,\varepsilon)$.  
\end{lem}
\begin{proof}
We use the notations in §14.1 \cite{benoist-quint}. 
Let $\alpha \in \Pi$ such that $\alpha(\log a) \geq -2\log(\varepsilon)$.
Recall \eqref{norm-wedge} that $\Vert \wedge_2 \rho_\alpha(a) \Vert =\exp( (2\chi^\alpha- \alpha)\log a)$ and $\Vert \rho_\alpha(a) \Vert=\exp(\chi^\alpha(\log a))$.
We compute the gap between the first and second eigenvalues of $\rho_\alpha(a)$,
\[\gamma_{1,2}(\rho_\alpha(a)):=\frac{\|\wedge_2 \rho_\alpha(a)\|}{\|\rho_\alpha(a) \|^2}=e^{-\alpha(\log a)}. \]
By assumption, $e^{-\alpha(\log a)} < \varepsilon^2$, hence $\gamma_{1,2}(\rho_\alpha(a)) < \varepsilon^2$.
Then we apply Lemma 14.2 (iii) in \cite{benoist-quint}, for every $y \in \cal{V}_\varepsilon ( (x_+^\alpha)^\perp )^\complement$,
\[ \dd_\alpha( \rho_\alpha(a) y ,x_+^\alpha)\delta_\delta(x_+^\alpha, y)<\gamma_{1,2}(\rho_\alpha(a)). \]
By definition $\delta_\alpha (y,x_+^\alpha) \geq \varepsilon$, hence $\dd_\alpha(\rho_\alpha(a)y,x_+^\alpha) <\varepsilon$ and we deduce that $\rho_\alpha(a) \cal{V}_\varepsilon ( (x_+^\alpha)^\perp)^\complement \subset B(x_+^\alpha,\varepsilon)$.
\end{proof}

\paragraph{Distances and balls in $\cal{F}$}
Using the fundamental representations $(\rho_\alpha)_{\alpha\in \Pi}$, Tits (Cf. \cite[Lemma 6.32]{benoist-quint}) also proved that the following map is an embedding:
\begin{align*}
\cal{F} &\longrightarrow \prod_{\alpha\in \Pi} \mathbb{P}(V^\alpha) \\
\xi=k \eta_0 & \longmapsto (x^\alpha(\xi))_{\alpha\in \Pi}:= (\rho_\alpha(k) x_+^\alpha)_{\alpha \in \Pi}.
\end{align*}

Denote by $\dis_P$ Riemannian distance on the product space $\Pi_{\alpha\in\Pi}\P(V^\alpha)$. 
Recall that on any product space $X\times Y$ where $(X,g_1)$ and $(Y,g_2)$ are endowed with Riemannian metrics $g_1$ and $g_2$, the product Riemannian metric $g$ is given for all $(x,y ; v,w) \in T_{(x,y)} X \times Y$
where $(x,v)\in T_x X$ and $(y,w) \in T_y Y$, by
\[g(x,y;v,w)=g_1(x,v)+g_2(y,w). \]
The Riemannian distance $\dis_R$ associated to this product Riemannian metric satisfies
\[ \max\{\dis_1,\dis_2\}\leq \dis_R \leq \dis_1+\dis_2. \]
Since for every $\alpha \in \Pi$, the distance $\dis_\alpha$ is equivalent to the Riemannian distance on the projective space, we deduce that $\dis_P$ is equivalent to the maximal metric i.e. $\dis_P \asymp \dis:=\sup_{\alpha \in \Pi} \dis_\alpha $.
Using Tits' embedding of $\calF$ in to the product space $\Pi_{\alpha\in\Pi}\P(V^\alpha)$, we deduce that the induced metric is non-degenerate on $\calF$. 
We thus define the following distance on $\cal{F}$ for all $\xi,\eta \in \cal{F}$
\begin{equation}\label{defin-dist}
\dis (\xi,\eta):= \sup_{\alpha \in \Pi} \dis_\alpha (x^\alpha(\xi),x^\alpha(\eta)).
\end{equation}
It is equivalent to the Riemannian distance on $\calF$ induced by the embedding on the product space $\Pi_{\alpha\in\Pi}\P(V^\alpha)$. 
For all $\xi \in \calF$ and $\varepsilon\in (0,1)$, we denote the balls for this distance by
\begin{equation}\label{defin-ballFd}
    B(\xi,\varepsilon):= \lbrace \eta\in \calF \; \vert \; \dis(\xi,\eta)<\varepsilon \rbrace.
\end{equation}
\noindent Similarly, noting that $(\zeta_0)_o^\perp = \eta_0$, we set
\begin{equation}\label{defin-delta}
\delta (\xi,\eta):= \inf_{\alpha \in \Pi} \delta_\alpha (x^\alpha(\xi),x^\alpha(\eta_o^\perp)) = \inf_{\alpha \in \Pi} \dis_\alpha(x^\alpha(\xi), x^\alpha(\eta_o^\perp)^\perp ).
\end{equation}
For all $\xi \in \mathcal{F}$ and $\varepsilon\in (0,1)$, we denote the balls for $\delta$ by
\begin{equation}\label{defin-ball-Fdel}
    \cal{V}_\varepsilon (\xi) := \lbrace \eta \in \calF \; \vert \; \delta(\eta,\xi) <\varepsilon \rbrace.
\end{equation}
Using the above notations given for the balls in $\calF$ for $\delta$ and $\dis$ and their $K$-invariance, we upgrade the dynamical Lemma \ref{lem_rhoa-action} to elements in $G$ whose Cartan projection is far from the walls of the Weyl chambers.
\begin{lem}\label{lem-action-aplus}
For all $g \in G$, choose $k,l\in K$ by Cartan decomposition such that $g=k \exp(\underline{a}(g)) l^{-1}$.
Let $\varepsilon>0$ and assume that $\dis ( \underline{a}(g) , \partial \frak{a}^+) \gg -2 \log \varepsilon$,
then $g \cal{V}_\varepsilon(l \zeta_0 )^\complement  \subset B( k \eta_0,\varepsilon)$. 
\end{lem}
\begin{proof}
Note that $\alpha (v) \asymp \dis(v, \ker \alpha)$ for all $v \in \frak{a}^+$.
Hence by taking the infimum over $\alpha \in \Pi$, then using that $\inf_{\alpha \in \Pi} \dis (v, \ker \alpha) = \dis (v, \cup_{\alpha \in \Pi} \ker \alpha )$ and finally, because $\frak a^+$ is a salient cone, $\partial \frak{a}^+ = \frak{a}^+ \cap \big( \cup_{\alpha \in \Pi} \ker \alpha \big)$, we deduce that for all $v \in \frak{a}^+$,
$$ \dis(v, \partial \frak{a}^+) \asymp \inf_{\alpha \in \Pi} \alpha (v).$$
Now using the underlying constant in $\asymp$, we may assume that, $\inf_{\alpha \in \Pi} \alpha(\underline{a}(g)) \geq - 2 \log \varepsilon$.
Applying the dynamical Lemma \ref{lem_rhoa-action} simultaneously for all $\alpha \in \Pi$, using Remark \ref{rem-zeta-eta} that $(\zeta_0)_o^\perp= \eta_0$, we deduce that $e^{\underline{a}(g)} \cal{V}_\varepsilon (\zeta_0)^\complement \subset B(\eta_0,\varepsilon)$.
Finally, we deduce the lemma by invariance of left $K$-action on both $\dis$ and $\delta$.
\end{proof}

\paragraph{Action of $G$ on $\cal{F}$}
We want to understand how the left action of $G$ on $\calF$ distorts the balls for $\delta$ and $\dis$.
Let $C_\frak{a}>1$ be a positive constant such that for all $v \in \frak{a}$,
 $$\frac{1}{\sqrt{C_{\frak{a}} }} \Vert v \Vert \leq  \sup_{\alpha \in \Pi}  \vert \chi^\alpha ( v ) \vert \leq  \sqrt{C_\frak{a}} \Vert v \Vert  .$$
This constant gives the comparison of the sup-norm induced by the dual basis $(\chi^\alpha)_{\alpha \in \Pi}$ with the Euclidean norm $\Vert \; \Vert$ on $\mathfrak{a}$.

\begin{lem}\label{lem-actiong}
The distances $\dis$ and $\delta$ are left $K$-invariant.
There exist $C_0, C_1>1$ such that for all $g$ in $G$ and $\xi,\eta$ in $\calF$, we have the following inequalities:
\begin{itemize}
    \item[(i)] $\dis (g \xi, g \eta) \leq C_1 e^{ C_0 \dis_X(o,go)   } \dis(\xi,\eta),$
    \item[(ii)] $\delta (g \xi, g \eta) \leq C_1 e^{ C_0 \dis_X(o,go)   } \delta (\xi,\eta),$
    \item[(iii)] $ \Vert \sigma (g, \xi) - \sigma(g,\eta) \Vert  \leq C_1 e^{ C_0 \dis_X(o,go)   } \dis(\xi,\eta),$
    \item[(iv)] $ \Vert \sigma(g,\xi) \Vert \leq C_\frak{a} \dis_X(o,go) .$
\end{itemize}
Furthermore, for every $x,y \in X$ and $\xi \in \calF$, (iv) is the same as
\begin{itemize}
        \item[(iv')]$ \Vert \beta_{\xi}(x,y) \Vert \leq C_{\frak{a}} \dis_X(x,y) .$
\end{itemize}
\end{lem}
\noindent In particular, for all $x\in X$ we set $C_x:= C_1 e^{C_0 \dis_X(o,x)}$. 
Then for all $h_x \in G$ such that $h_xo=x$ and all $\xi \in \calF$ and every $r \in (0, C_x^{-1}) $, the inequalities given by (i) and (ii) imply 
\begin{itemize}
    \item[(i')]  $B(h_x \xi, C_x^{-1} r) \subset h_x B(\xi,r) \subset B(h_x \xi, C_x r)$,
    \item[(ii')] $ \cal{V}_{C_x r}(h_x \xi) \subset h_x \cal{V}_r(\xi) \subset  \cal{V}_{C_x^{-1} r}(h_x\xi)$.
\end{itemize}

\begin{proof}
For each $V^\alpha$, by (13.1) in \cite{benoist-quint}, we have
\[\dis_\alpha(x^\alpha(g\xi),x^\alpha(g\eta))\leq \|\rho_\alpha(g)\|^2\|\rho_\alpha(g^{-1})\|^2 \dis_\alpha (x^\alpha(\xi),x^\alpha(\eta)). \]
By \eqref{defin-dist} and $\|\rho_\alpha(g)\|=\|\rho_\alpha \exp(\underline{a}(g)) \|=\exp(\chi^\alpha(\underline{a}(g)))$, we obtain the first inequality for $C_0=4 C_{\frak{a}}$.

For (ii), we first prove that $(x^\alpha((g\eta)_o^\perp))^\perp=\rho_\alpha(g)x^\alpha(\eta_o^\perp)^\perp$. There exist $k_1,k\in K$ such that $\eta=k_1\eta_0$ and $gk_1=kan\in KAN$. Then due to $k$ preserving $o$ and the Euclidean metric on $V^\alpha$, we obtain
\[(x^\alpha((g\eta)_o^\perp))^\perp=(x^\alpha((k\eta_0)_o^\perp))^\perp=\rho_\alpha(k)(x^\alpha((\eta_0)_o^\perp))^\perp. \]
Due to $AN$ preserving $(x^\alpha((\eta_0)_o^\perp))^\perp=(x^\alpha(\zeta_0))^\perp$, we deduce that $\rho_\alpha(k)(x^\alpha((\eta_0)_o^\perp))^\perp=\rho_\alpha(gk_1)(x^\alpha((\eta_0)_o^\perp))^\perp$. Therefore, we obtain $(x^\alpha((g\eta)_o^\perp))^\perp=\rho_\alpha(g)(x^\alpha(\eta_o^\perp))^\perp$.
Then for all $\xi,\eta \in \calF,$
\begin{align*}
    \delta_\alpha(x^\alpha(g\xi),x^\alpha((g\eta)^\perp_o)) &=\dis_\alpha(x^\alpha(g\xi),x^\alpha((g\eta)_o^\perp)^\perp)=\dis_\alpha(\rho_\alpha(g)x^\alpha(\xi)^\perp,\rho_\alpha(g)x^\alpha(\eta_o^\perp)^\perp) \\
    &\leq \Vert \rho_\alpha(g) \Vert^2 \|\rho_\alpha(g)^{-1} \|^2  \;  \dis_\alpha(x^\alpha(\xi),x^\alpha(\eta_o^\perp)^\perp) .
\end{align*}
Therefore, since $\Vert \rho_\alpha(g) \Vert \|\rho_\alpha(g)^{-1} \| \leq \exp ( 2 \sup ( 
 \chi^\alpha ( \underline{a}(g)), \chi^\alpha( \iota \underline{a}(g)) ) )$ and $C_0=4 C_{\frak{a}}$, we deduce that
\[ \delta(g\xi,g\eta)=\inf_{\alpha\in\Pi}\delta_\alpha(x^\alpha(g\xi),x^\alpha((g\eta)_o^\perp))\leq C_1e^{C_0\|\underline{a}(g)\|} \delta(\xi,\eta). \]

(iii) is given in \cite[Lemma 13.1]{benoist-quint}.

(iv), see \cite[Lemma 3.12]{dg20} for a similar statement, and it is also a direct consequence of \cite[Lemma 6.33 (ii), Corollary 8.20]{benoist-quint}. 

Finally (iv') is a consequence of the formulas $\beta_\xi(x,y)=\sigma( h_x^{-1} h_y, h_y^{-1} \xi)$ and $\dis_X(x,y)=\Vert \underline{a}(h_x^{-1} h_y) \Vert$ independently of the choice of $h_x, h_y \in G$ such that $h_xo=x$ and $h_yo=y$.
\end{proof}

\subsection{Distances on $G/M$ }\label{sec-dist}

On one hand, denote by $\dis_1$ the left $G$-invariant and right $K$-invariant Riemannian distance on $G/M$.
It is the higher rank analogue of the distance on the unit tangent bundle of $\mathbb{H}^2$. 
The map $(G/M, \dis_1) \rightarrow (X,\dis_X)$ is continuous and equivariant for the left $G$-action.

On the other hand, using Hopf coordinates, we consider $\dis_2$, a distance equivalent to the Riemannian product distance induced by the embedding $\calF^{(2)}\times \mathfrak{a} \hookrightarrow \calF \times \calF \times \mathfrak{a}$.
These distances are locally equivalent, however, since $\dis_2$ is not left $G$-invariant they are not globally equivalent.

We compute an expanding estimate for the action of $G$ on $G/M$ for $\dis_2$.
Then we deduce the constants underlying the local equivalence of $\dis_1$ and $\dis_2$. 

\paragraph{Distance induced by Hopf coordinates}

For every pair $(\xi^+,\xi^-, v), (\eta^+,\eta^- ,w) \in \calF^{(2)} \times \frak{a}$, we define 
\begin{equation}\label{defin-distHopf}
    \dis_2 \big( (\xi^+,\xi^-,v), (\eta^+,\eta^-,w) \big) := \sup ( \dd(\xi^+,\eta^+), \dd(\xi^-,\eta^-), \Vert v-w \Vert ).
\end{equation}
Due to the definitions \eqref{defin-dist} of the distances on $\calF$, the distance $\dis_2$ is not left $G$-invariant even though it is left $K$-invariant.
Since $\dis$ is equivalent to the Riemannian distance on $\calF$ and the maximal metric is equivalent to the Riemannian metric on the product space $\calF \times \calF \times \mathfrak{a}$, the distance $\dis_2$ is equivalent to the product distance.
It is non-degenerate because of the embedding $\calF^{(2)}\times \mathfrak{a} \hookrightarrow \calF \times \calF \times \mathfrak{a}$.

Abusing notations, for every $z_1, z_2 \in G$, we also denote by  
$\dis_2( z_1M , z_2 M ) := \dis_2\big( \cal{H}(z_1 M) , \cal{H}(z_2M)  \big).$
For all $(\xi^+,\xi^-,v) \in \calF^{(2)} \times \frak{a}$, all $r\in (0,\frac{1}{2}\delta(\xi^+,\xi^-) )  $, the ball of radius $r$ for $\dis_2$ centered in that element is
$$ B(\xi^+,r) \times B(\xi^-,r) \times B_{\frak{a}}(v,r) .$$

\begin{lem}\label{lem-d2}
For $g\in G$ and $z_1,z_2$ in $G$, we have
\[\dis_2(gz_1M,gz_2M)\leq  C_1e^{C_0\|\underline{a}(g)\|}\dis_2(z_1M,z_2M). \]
\end{lem}
\begin{proof}
We write down $z_1M$ and $z_2M$ in Hopf coordinates, we denote by $(\xi_i^+,\xi_i^-,v_i):= \cal{H}(z_iM)$ for $i=1,2$.
By \eqref{equ-hxi} and Lemma \ref{lem-actiong} (i),(ii),(iii), we have
\begin{align*}
    \dis_2(gz_1M,gz_2M)&= \dis_2((g\xi_1^+,g\xi_1^-,v_1+\sigma(g,\xi_1^+)),(g\xi_2^+,g\xi_2^-,v_2+\sigma(g,\xi_2^+))\\
    &= \sup\big( \dd (g \xi_1^+, g\xi_2^+ ), \dd(g \xi_1^-,g \xi_2^-) , \Vert v_1 -v_2 + \sigma(g, \xi_1^+) - \sigma(g,\xi_2^+) \Vert   \big)\\
    &\leq \sup \big( C_1e^{C_0\|\underline{a}(g)\|} \dd(\xi_1^+,\xi_2^+) ,  C_1e^{C_0\|\underline{a}(g)\|} \dd(\xi_1^-,\xi_2^-), C_1e^{C_0\|\underline{a}(g)\|} \dd(\xi_1^+,\xi_2^+)+\|v_1-v_2\| \big) \\
     &\leq  C_1e^{C_0\|\underline{a}(g)\|} \dis_2(z_1M,z_2M).
\end{align*}
\end{proof}
\noindent As a consequence, for all $z_1\in G/M$, small $r>0$ and $g\in G$
$$ g B_2(z_1,r) \subset B_2 (gz_1, C_1e^{C_0\|\underline{a}(g)\|} r ).$$

\paragraph{Local equivalence constants}
Since $\dis_1$ and $\dis_2$ are Riemannian metrics of $G/M$, they are locally equivalent.
We fix a neighbourhood $O$ of $eM$ and a constant $C_2>0$ such that for every $z_1, z_2 \in O$,
\[\frac{1}{C_2} \dis_2(z_1,z_2)   \leq \dis_1(z_1,z_2)\leq C_2 \dis_2(z_1,z_2). \]
For any $r>0$, denote by $B_1(zM, r)\subset G/M$ the ball of radius $r$ centered on $zM$, for the distance $\dis_1$. 
Fix $\epsilon_0>0$ such that $B_1(eM,\epsilon_0) \cup B_2(eM,\epsilon_0) \subset O$.
\begin{defin}
For $x\in X$, let 
\begin{equation}\label{equ-cx}
    C_x=8C_2C_1\exp(C_0\dis_X(o,x)).
\end{equation}
\end{defin}

\begin{lem}\label{lem-dist-compare}
For $x\in X$ and $z_1,z_2\in G/M$ with $x=\pi(z_1)$, if $\dis_2(z_1,z_2)<\frac{\epsilon_0}{C_x}$ or $\dis_1(z_1,z_2)<\epsilon_0$, then
\[\dis_1(z_1,z_2)\leq \frac{C_x}{4}\dis_2(z_1,z_2). \]
\end{lem}

\begin{proof}

We take $h_x$ such that $h_x^{-1}z_1=eM$.
Then we have either 
$$\dis_2(h_x^{-1}z_1,h_x^{-1}z_2)\leq C_x\dis_2(z_1,z_2)<\epsilon_0,$$
(due to Lemma \ref{lem-d2})
or $\dis_1(h_x^{-1}z_1,h_x^{-1}z_2)=\dis_1(z_1,z_2)<\epsilon_0$. Due to the choice of $\epsilon_0$, we can apply local equivalence at $eM$ and Lemma \ref{lem-d2} to obtain
\[\dis_1(z_1,z_2)=\dis_1(h_x^{-1}z_1,h_x^{-1}z_2)\leq C_2\dis_2(h_x^{-1}z_1,h_x^{-1}z_2)\leq C_x\dis_2(z_1,z_2)/4. \]
The proof is complete.
\end{proof}

\subsection{Corridors of maximal flats}
Recall from Definition \ref{defin-max-flat}, for every point $y\in X$ and every $\xi \in \mathcal{F}$, we denote by $\xi_y^\perp \in\calF$ the opposite element such that $y \in \big(\xi \xi_y^\perp \big)_X$. 
\begin{defin}\label{defin_corridors}
Let $x \in X$ and $r>0$.
We denote by $\cal{F}^{(2)}(x,r)$ the open \emph{corridor of maximal flats} at distance $r$ of $x$  
\begin{equation}\label{eq_corridor}
\cal{F}^{(2)} (x,r) := \lbrace (\xi,\eta) \in \mathcal{F}^{(2)} \; \vert \; \dis_X (x, (\xi \eta)_X ) <r \rbrace. 
\end{equation}
We denote by $\widetilde{ \cal{F}^{(2)}} (x,r)$ the set of Weyl chambers based in $B_X(x,r)$
\begin{equation}\label{eq_based-wc}
\widetilde{ \cal{F}^{(2)}} (x,r) := \Big\lbrace \big(\xi,\xi_y^\perp , \beta_{\xi} (o,y) \big) \in \mathcal{F}^{(2)}\times \frak{a} \; \Big\vert \; y \in B_X(x,r) \Big\rbrace. 
\end{equation}
\end{defin}
By Hopf coordinates \eqref{equ_hopfmap}, we obtain
\begin{fait}\label{prop-corridor}
For all $x \in X$ and $r>0$, the set $\widetilde{ \cal{F}^{(2)}} (x,r)$ is the preimage of $B_X(x,r)$ by the projection $G/M \rightarrow G/K$. 
\end{fait}

\begin{lem}\label{lem-compare}
Let $x \in X$ and $\min\{\frac{\epsilon_0}{2},\frac{\log 2}{C_0} \}>r>0$.
Then for every $\varepsilon \in (0, C_x^{-1}r)$, all $(\xi^+,\xi^-) \in \cal{F}^{(2)}(x,r)$,
$$ B(\xi^+,\varepsilon) \times B(\xi^-,\varepsilon) \subset \cal{F}^{(2)}(x,2r).$$
\end{lem}

\begin{proof}
Let $(\xi^+,\xi^-) \in \cal{F}^{(2)}(x,r)$ as in the hypothesis.
There exists $y \in B_X(x,r)$ such that $\xi^-=(\xi^+)_y^\perp$.
Now we choose in Hopf coordinates $z:=(\xi^+,\xi^- \; , \; \beta_{\xi^+}(o,y) )$.
By properties of $\dis_1$, that $B_X(y,r) \subset B_X(x,2r)$ and comparison Lemma \ref{lem-dist-compare} between the distances, we have 
$$ B_2\Big(z, \frac{4}{C_x} r \Big) \subset B_1(z,r) \subset \widetilde{ \cal{F}^{(2)}} (x,2r).$$
Finally, the proof is completed by projecting the ball $B_2\Big(z, \frac{4}{C_x} r \Big)$ into the coordinates in $\calF^{(2)}$.
\end{proof}

\begin{lem}\label{lem-cartan-lox}
Let $g\in G$ and $x\in X$.
Assume there is a transverse pair $(\xi^+,\xi^-) \in \cal{F}^{(2)}$ of fixed points for the action of $g$ on $\cal{F}$.
Then
$$ \Vert \lambda(g) - \underline{a}_x(g) \Vert \leq 2 \dis_X (x, (\xi^+ \xi^-)_X ) .$$
\end{lem}
\begin{proof} 
For every transverse pair $(\xi^+,\xi^-)$, there exists, up to right multiplication by elements of $AM$, an $h\in G$ such that $h(\eta_0,\zeta_0)=(\xi^+,\xi^-)$. 
By assumption, $\xi^+$ and $\xi^-$ are fixed by $g$, i.e. $gh \in hAM$.
By Cartan decomposition, for every $p\in h AM o$, we have $\underline{a}_{p}(g)=\lambda(g)$.

Since $hAM o=hAo$, which is equal to the flat $(\xi^+\xi^-)_X$. 
It then follows from Lemma \ref{lem-cartan-points} that
for every $p \in (\xi^+ \xi^-)_X$
$$ \Vert \lambda(g) - \underline{a}_x(g) \Vert=\|\underline{a}_p(g)-\underline{a}_x(g) \| \leq 2 \dis_X(x,p). $$
Taking the infimum over the points in the flat $(\xi^+\xi^-)_X$ yields the upper bound.
\end{proof}

\subsection{The configuration}
Recall that for all $x\in X$, we defined the constant
$C_x=8C_2C_1 e^{C_0\dis_X(o,x)}.$
\begin{defin}\label{defin_t0}
Denote by $r_0$ the unique zero in $(0,1)$ of the real valued function $r \mapsto -\log r -\max\{C_3,2\}r$.
For all $\varepsilon>0$ and $x \in X$ we define some function
$$t_0(x,\varepsilon) \gg 2\log C_x-2 \log(\varepsilon),$$
where the constant underlying $\gg$ is the same as in Lemma \ref{lem-action-aplus}.
\end{defin}

\begin{prop}\label{prop-lemroblin}
For all $x\in X$ and $r\in (0,r_0)$ and $\varepsilon\in (0,\min\{C_x^{-1}r,\epsilon_0\} )$, 
 every $\gamma \in G$ satisfying the following conditions is loxodromic.
\begin{itemize}
\item[(i)] $\underline{a}_x(\gamma) \in \frak{a}^{++}$ and $\dd(\underline{a}_x(\gamma), \partial \frak{a}^+ ) \geq  t_0(x,\epsilon)$,
\item[(ii)] $(\gamma_x^+,\gamma_x^-) \in \mathcal{F}^{(2)}$ are transverse and $\dis_X \big(x, (\gamma_x^+ \gamma_x^-)_X  \big) <r$.
\end{itemize}
Furthermore, its attracting and repelling point satisfy $\gamma^\pm \in B(\gamma_x^\pm, \varepsilon)$.
\end{prop}

\begin{proof}
There exist $k_{\gamma_x^+}, l_{\gamma_x^-} \in h_x K$ (as $h$ and $h'k_\iota$ in Definition \ref{defin-cartan-regular}), defined up to right multiplication by elements of $M$ and independent of the choice of representative $h_x \in G$ such that
$\gamma = k_{\gamma_x^+} e^{\underline{a}_x(\gamma)} l_{\gamma_x^-}^{-1} .$
Apply Lemma \ref{lem-action-aplus}, to the element $h_x^{-1}\gamma h_x=h_x^{-1}k_{\gamma_x^+} e^{\underline{a}_x(\gamma)} (h_x^{-1}l_{\gamma_x^-})^{-1} \in K A^{++}K$,
$$ h_x^{-1}\gamma h_x\; \cal{V}_{C_x^{-1}\varepsilon}( h_x^{-1}\gamma_x^-)^\complement \subset B(h_x^{-1}\gamma_x^+,C_x^{-1}\varepsilon).$$
We multiply by $h_x$ on the left 
$ \gamma h_x \cal{V}_{C_x^{-1}\varepsilon}(h_x^{-1}\gamma_x^- )^\complement \subset h_x B(h_x^{-1}\gamma_x^+,C_x^{-1}\varepsilon).$
Using the properties of $C_x>0$ (Lemma \ref{lem-actiong}), we deduce the following inclusions 
\begin{itemize}
\item[•] $h_x B(h_x^{-1}\gamma_x^+,C_x^{-1}\varepsilon) \subset B(\gamma_x^+, \varepsilon)$,
\item[•] $\cal{V}_{\varepsilon}(\gamma_x^-)^\complement \subset h_x \cal{V}_{C_x^{-1}\varepsilon}(h_x^{-1}\gamma_x^-)^\complement$.
\end{itemize}
Hence
$ \gamma \cal{V}_{\varepsilon}(\gamma_x^-)^\complement \subset B(\gamma_x^+,\varepsilon).$
Recall that $\iota$ is the opposition involution and $k_\iota \in N_K(A)$ such that $\iota=-Ad(k_\iota)$, then
$$\gamma^{-1}=l_{\gamma_x^-}k_\iota \; e^{\iota \underline{a}_x(\gamma)}\; (k_{\gamma_x^+}k_\iota)^{-1}.$$
Since $\iota \underline{a}_x(g)$ is at distance at most $t_0$ from $\partial \frak{a}^+$ and $(\gamma^{-1})_x^\pm = \gamma_x^\mp$, we deduce that
$ \gamma^{-1} \cal{V}_\varepsilon(\gamma_x^+)^\complement \subset B(\gamma_x^-,\varepsilon).$

Due to $\dis_X(o,((h_x^{-1}\gamma_x^+)(h_x^{-1}\gamma_x^-))_X)=\dis_X(x,(\gamma_x^+\gamma_x^-)_X)<r$, by Lemma \ref{lem-xietao} and Definition \ref{defin_gromov}, we obtain
\[ \delta(h_x^{-1}\gamma_x^+,h_x^{-1}\gamma_x^-)\geq e^{-C_3r}. \]
Then by Lemma \ref{lem-actiong}, we have
\[ \delta(\gamma_x^+,\gamma_x^-)\geq C_x^{-1}\delta(h_x^{-1}\gamma_x^+,h_x^{-1}\gamma_x^-)\geq C_x^{-1}e^{-C_3r}. \]
Due to the choice of $\epsilon,r$, we have $C_x^{-1}e^{-C_3r}>2\epsilon$. Hence we have $B(\gamma_x^\pm, \varepsilon) \subset \cal{V}_\varepsilon(\gamma_x^\mp)^\complement$. Then we deduce that $\gamma$ (resp. $\gamma^{-1}$) has an attracting fixed point $\xi^+\in B(\gamma_x^+,\varepsilon)$ (resp. $\xi^-\in B(\gamma_x^-,\varepsilon)$).

Since $\gamma$ admits a fixed maximal flat $(\xi^+\xi^-)_X$, we apply Lemma \ref{lem-cartan-lox},
$$ \Vert \lambda(\gamma)-\underline{a}_x(\gamma) \Vert \leq 2 \dis_X(x, (\xi^+\xi^-)_X ).$$
By hypothesis $\varepsilon <C_x^{-1}r$, Lemma \ref{lem-compare} implies that 
$ B(\gamma_x^+,\varepsilon) \times B(\gamma_x^-,\varepsilon) \subset \cal{F}^{(2)}(x,2r)$.
Hence $\lambda(\gamma) \in B( \underline{a}_x(\gamma), 4r)$.
Using that $r<r_0$ and $\varepsilon < C_x^{-1}r$, we get a lower bound
$t_0(x,\varepsilon) >  - 2 \log r > 4r.$
We deduce that $B(\underline{a}_x(\gamma),4r) \subset \frak{a}^{++}$, therefore $\gamma$ is loxodromic.

Finally, because the bassin of attraction of $\gamma^+$ (resp. $\gamma^-$) is a dense open set of $\calF$, there are points in $B(\gamma_x^+,\varepsilon)$ (resp. $B(\gamma_x^-,\varepsilon)$) that $\gamma$ (resp. $\gamma^{-1}$) will attract to $\gamma^+$ (resp. $\gamma^-$). 
Since $\calF$ is Hausdorff for $\dd$, we deduce that $\gamma^+=\xi^+$ (resp. $\gamma^-=\xi^-$). 
\end{proof}

\section{Conjugacy classes and periodic tori}
In this section, we remove the torsion free assumption and only assume $\Gamma<G$ to be a cocompact lattice.
We denote in brackets the $\Gamma$-conjugacy classes of elements in $\Gamma$.
Denote by $[\Gamma^{lox}]$ (resp. $[\Gamma^{sing}]$) the set of $\Gamma$-conjugacy classes of loxodromic (resp.  singular) elements.

The following Lemma is due to Selberg.
\begin{lem}[ \cite{selberg}, \cite{pr} ]\label{lem-selberg}
Let $\Gamma$ be a cocompact lattice.
Let $F$ be a right $A$-orbit in $\Gamma\backslash G/M$. If $\Lambda(F)\cap\frak a^{++} \neq \emptyset $, then $F$ is a compact periodic $A$-orbit.
\end{lem}
\begin{proof}
We can write $F=\Gamma g AM$. For a non zero $Y\in \Lambda(F)\cap\frak a^{++}$, by definition $\Gamma gM=\Gamma g\exp(Y)M$.
Hence there exists an element $\gamma\in\Gamma$ such that $\gamma g=g\exp(Y)m_Y$. By Selberg's Lemma in \cite{selberg} or \cite[Lemma 1.10]{pr}, the map $\Gamma_\gamma\backslash G_\gamma \rightarrow \Gamma \backslash G$ is proper, where $G_\gamma$ (resp. $\Gamma_\gamma$) denotes the centralizer of $\gamma$ in $G$ (resp. $\Gamma$). 
Therefore, $\Gamma_\gamma\backslash G_\gamma$ is compact.

Then $G_\gamma$ is a conjugated to $AM$. 
Now $gAMg^{-1}$ commutes with $\gamma$, so $G_\gamma = gAMg^{-1}$ and $\Gamma_\gamma = \Gamma \cap G_\gamma$. 
Then $\Gamma_\gamma\backslash G_\gamma = (\Gamma\cap G_\gamma)\backslash G_\gamma$ compact implies that $\Gamma gAM=\Gamma G_\gamma g$ is compact in $\Gamma\backslash G$. So $F$ is compact in $\Gamma\backslash G/M$.
\end{proof}
In the first paragraph, we give a relation between conjugacy classes of loxodromic elements and periodic tori.
In the second paragraph, we give a proof for completeness that singular elements of a cocompact lattice do not have a unipotent part.

\subsection{The case of loxodromic elements}

For every loxodromic element $\gamma \in \Gamma^{lox}$, we denote by  
$\calL_\gamma$ the measure of $G/M$ supported on the $A$-orbit of Hopf coordinates $( \gamma^+ , \gamma^-; \frak a)$ such that its disintegration in Hopf coordinates is given by
\begin{equation}\label{defin_mes-tore}
\calL_\gamma := D_{\gamma^+} \otimes D_{\gamma^-} \otimes Leb_{\frak{a}},
\end{equation}
where $D_{\gamma^\pm}$ is the Dirac measure at $\gamma^\pm$. 

Remark that the quotient in $\Gamma\backslash G/M$ of the $A$-orbit $( \gamma^+ , \gamma^-; \frak a)$ only depends on the conjugacy class $[\gamma]$.
Denote by $F_{[\gamma]}$ the quotient of this $A$-orbit in $\Gamma\backslash G/M$. 
We claim that every point in $F_{[\gamma]}$ is periodic for the Weyl chamber flow $\Gamma \backslash G/M \curvearrowleft e^{\lambda(\gamma)}$.
Indeed, by \eqref{eq_gg+}, that is $\gamma(\gamma^+,\gamma^-,Y)=(\gamma^+,\gamma^-,Y+\lambda(\gamma))$ for every $Y\in \frak{a}$, hence $\lambda(\gamma)\in\Lambda(F_{[\gamma]})$. 
If we take $g_\gamma$ an element such that $(\gamma^+,\gamma^-,0)=g(\eta_0,\zeta_0,0)$ i.e. that Jordan diagonalizes $\gamma$, then the formula also implies $g_\gamma^{-1}\gamma g_\gamma\in \exp(\lambda(\gamma))M$. 
With this $g_\gamma$, we may express this $A$-orbit $F_{[\gamma]}=\Gamma g_\gamma AM$.

Let $\calG (A):=\{(Y,F)|\ F\in C(A),\ Y\in\Lambda(F)\cap \frak a^{++} \}$.

\begin{prop}\label{lem-gaglox}
Let $\Gamma$ be a cocompact lattice.
 	If the action of $\Gamma$ on $G/M$ is free, then the following map is well-defined and bijective.
  \begin{align*}
      \Psi : [\Gamma^{lox}] &\longrightarrow \calG(A) \\
       [\gamma] & \longmapsto (\lambda(\gamma), F_{[\gamma]}).
  \end{align*}
\end{prop}

\begin{proof}
We first prove that the following map is surjective.
\begin{align*}
    \widetilde{\Psi} : \Gamma^{lox} & \longrightarrow \calG(A) \\
    \gamma & \longmapsto (\lambda(\gamma), F_{[\gamma]}).
\end{align*}
Indeed, fix any compact periodic $A$-orbit $F$. We may denote it by $\Gamma g AM$, for some $g\in G$. 
 For every regular period in this $A$-orbit $Y\in\Lambda(F)\cap\frak a^{++}$, by definition, $xe^{Y}=x$ for all $x \in F$. 
 Now we fix any point $x\in F$ and choose any $g\in G$ such that $\Gamma gM=x$.
 We deduce that there exists an element $\gamma_Y\in\Gamma$ such that $\gamma_Yg =g\exp(Y)m_Y$ for some $m_Y \in M$.
 Hence the surjectivity of the map $\widetilde{\Psi}$.

 Note that $(\beta \gamma_Y \beta^{-1} ) \beta g = \beta g e^{Y}m_Y$ for all $\beta \in \Gamma$ i.e.
$\widetilde{\Psi}( \beta \gamma_Y \beta^{-1}  )=\widetilde{\Psi}( \gamma_Y )$.
It implies that the quotient map $\Psi$ is well-defined.

 Now let us prove the injectivity of the quotient map.
Consider $\gamma_Y$ as above and assume by contradiction there exists a distinct $\gamma_Y'\in \Gamma$ such that $\gamma_Y'g=g\exp(Y)m_Y'$ for some $m_Y'\in M$.
 Since $\gamma_Y\neq \gamma_Y'$ we deduce that $\gamma_Y^{-1}\gamma_Y'=gm_Y^{-1}m_Y'g^{-1}$ is not the identity. 
 This element $\gamma_Y^{-1}\gamma_Y'\in \Gamma$ fixes $gM$ in $G/M$, which contradicts that $\Gamma$ acts on $G/M$ freely.
 Hence $\Psi$ is injective.
\end{proof}

\subsection{The case of singular elements}
The following Proposition holds under the hypothesis that $\Gamma<G$ is torsion free and cocompact.
It is tautological for loxodromic elements.
The result should be well known for experts in the domain. 
We give the proof for singular elements in $\Gamma$ for completeness.
\begin{prop}\label{prop-singulier}
    Let $\Gamma<G$ be a cocompact lattice.
    Assume that the action of $\Gamma$ on the symmetric space $X=G/K$ is free.
    Then for all (non trivial) element $\gamma \in \Gamma$, its unipotent part in Jordan decomposition is trivial i.e.  there exists $h \in G$ and $k_\gamma \in Z_K(e^{\lambda(\gamma)})$ such that
    $$ \gamma = h e^{\lambda(\gamma)} k_\gamma h^{-1}.$$
\end{prop}
\begin{proof}
The relation is tautological for loxodromic elements in $\Gamma$. We prove it for singular elements.

Note that for every non trivial $\gamma \in \Gamma$, the injectivity radius of the manifold $\Gamma \backslash G/K$ is a lower bound of $\inf_{x\in X} \dis_X(x,\gamma x)$.
By hypothesis, $\Gamma \backslash G/K$ is a cocompact manifold.
Therefore its injectivity radius has a positive lower bound and $\inf_{x\in X} \dis_X(x,\gamma x)>0$ for any non trivial $ \gamma \in \Gamma$.

Fix a non trivial $\gamma \in \Gamma^{sing}$.
Consider a point that we denote by $p\in X$ where this infimum is realised i.e. $\dd_{X}(p, \gamma p)= \inf_{x\in X} \dis_X(x,\gamma x).$ 
We prove that the bi-infinite geodesic going through $p$ and $\gamma p$ is a translation axis of $\gamma$ on $X$.
Indeed, let $y\in [p,\gamma p]$, then on one hand $ \dis_X(y,\gamma y) \geq \dis_X(p,\gamma p)$.
On the other hand, by triangle inequality, left $G-$invariance of the distance $\dis_X$, we deduce that
    $$ \dis_X(y,\gamma y) \leq \dis_X(y,\gamma p)+ \dis_X(\gamma p, \gamma y) = \dis_X(y,\gamma p) + \dis_X(p,y).$$
Since $y$ is in the geodesic segment, we deduce that $\dis_X(y,\gamma y) \leq \dis_X(p,\gamma p)$, proving that it satisfies the minimizing equality.
By $G-$invariance over $\dis_X$ and gluing all the geodesic segments $ \gamma^\Z .[p,\gamma p]$ together, the same minimizing equality holds for every point in the geodesic $(p,\gamma p)$.  

Set $a_\gamma :=\exp \big( \underline{a}_p(\gamma) \big) \in A^+$. 
It is a non trivial element since $\Vert \underline{a}_p(\gamma) \Vert= \dis_X(p,\gamma p)$.
Recall that non trivial geodesics of $X$ based at $p$ are of the form $h e^{\R v}o$, where $v\in \mathfrak{a}^+$ is non zero and $h\in G$ is any element such that $ho=p$.
Hence, we fix $h \in G$ such that $ho=p$ and such that $he^{\R \underline{a}_p(\gamma) }o=( h a_\gamma^t o)_{t\in \R}$ denotes the bi-infinite geodesic $(p,\gamma p)$.  
Since the latter is a translation axis of $\gamma$, we deduce that for all integer $n\in \Z$, 
\begin{equation}\label{eq-prop-sing1}
     \gamma^n h o = h a_\gamma^n o.
\end{equation}

By the above relation, there exists $k_\gamma \in K$ such that $\gamma = h a_\gamma k_\gamma h^{-1}$.
We prove by induction that for all integer $n \in \Z$
\begin{equation}\label{eq-prop-sing2}
    a_\gamma^{-n} k_\gamma a_\gamma^n \in K.
\end{equation}
Note that the relation $a_\gamma k_\gamma = h^{-1} \gamma h$ yields the base case $n=0$.
Assume the relation is true up to some non negative integer $n \geq 0$.
By \eqref{eq-prop-sing1} on the one hand, $a_\gamma^{-n-2}  h^{-1} \gamma^{n+2} h \in K$.
On the other hand, 
$$ a_\gamma^{-n-2} \; \big( h^{-1} \gamma^{n+2} h \big)= a_\gamma^{-n-2} \; \overrightarrow{ \prod_{i=1}^{n+2} } h^{-1} \gamma h  = a_\gamma^{-n-2} \; \overrightarrow{ \prod_{i=1}^{n+2}} a_\gamma k_\gamma = \overrightarrow{ \prod_{i=1}^{n+2}}  a_\gamma^{-(n+2-i)} k_\gamma a_\gamma^{n+2-i} = \overrightarrow{ \prod_{k=n+1}^{0}} a_\gamma^{-k} k_\gamma a_\gamma^{k} .$$
By induction, the second until the last terms in multiplicative order are in $K$.
Furthermore, the entire product is in $K$. 
Consequently, the first term on the left, $a_\gamma^{-n-1} k_\gamma a_\gamma^{n+1}\in K$. 
Therefore, \eqref{eq-prop-sing2} is true for all non negative integers. 
For negative integers, using \eqref{eq-prop-sing1} similarly, $h^{-1} \gamma^{n+1}h \; a_\gamma^{-n-1} \in K$ for all $n>0$. 
For the computation, we notice the telescopic product
$$  \big( h^{-1} \gamma^{n+1} h \big) \; a_\gamma^{-n-1}  =  \Big( \overrightarrow{ \prod_{i=1}^{n+1} } h^{-1} \gamma h \Big) \; a_\gamma^{-n-1}
=\Big( \overrightarrow{ \prod_{i=1}^{n+1} } a_\gamma k_\gamma \Big) \; a_\gamma^{-n-1}
= \overrightarrow{ \prod_{i=1}^{n+1}} a_\gamma^{i} k_\gamma a_\gamma^{-i}  .$$
At each step, only the last term on the right is new, hence \eqref{eq-prop-sing2} holds for negative integers.

Since $K$ is compact, the sequence $(a_\gamma^{-n} k_\gamma a_\gamma^n)_{n\in \Z}$ is bounded. 
By Proposition \ref{prop-append-parab} (iii), we deduce that the compact element $k_\gamma$ is also commutes with $a_\gamma$ i.e. $k_\gamma \in Z_K(a_\gamma)$. 

To conclude, we have found a non trivial $a_\gamma \in A^+$, a commuting element $k_\gamma \in Z_K(a_\gamma)$ and $h\in G$ such that 
$$\gamma = h a_\gamma k_\gamma h^{-1}.$$
We recognize a Jordan decomposition of the singular element $\gamma$, where $h a_\gamma h^{-1}$ (resp. $h k_\gamma h^{-1}$) is the hyperbolic (resp. elliptic) part and the unipotent part is trivial.
\end{proof}
The above result implies that when $\Gamma$ is torsion free and uniform, every closed geodesic in the manifold $\Gamma \backslash G/K$ corresponds to a unique conjugacy class of non trivial elements in $\Gamma$. 
As a corollary, we deduce an upper bound for the distance between the Jordan projection and the Cartan projection.

\begin{corol}\label{lem-cartan-loxc}
Let $\Gamma<G$ be a torsion free, cocompact irreducible lattice. 
For every non trivial $\gamma\in \Gamma$, there exists $\gamma_0\in[\gamma]$ in its $\Gamma$-conjugacy class such that
\[\|\lambda(\gamma)-\underline{a}_o(\gamma_0) \|\leq 
C_\Gamma. \]

\noindent Furthermore, there exists an element $g\in G$ with $\|g\|\leq C_\Gamma$ and $k_\gamma \in Z_K(e^{\lambda(\gamma)})$ such that 
\[\gamma_0=g \exp(\lambda(\gamma))k_\gamma g^{-1}. \]
\end{corol}

\begin{proof}
    Fix a fundamental domain $\calD_
\Gamma  \subset X$ of diameter less than $2\,diam(\Gamma\backslash X)$ and containing $o$. 
Set $C_\Gamma:= 4 \,diam(\Gamma\backslash X).$
Fix a non trivial element $\gamma \in \Gamma$.

Since $\Gamma$ is cocompact and torsion free, the action of $\Gamma$ on the symmetric space $X$ is free.
Hence by the previous Proposition \ref{prop-singulier}, there exists $h\in G$ and $k_\gamma \in Z_K(e^{\lambda(\gamma)})$ such that 
$$\gamma=h \exp(\lambda(\gamma))k_\gamma h^{-1}. $$
Note that $he^{\R \lambda(\gamma)}o$ is a bi-infinite geodesic on $X$ and a translation axis for the action of $\gamma$.
Since $\calD_\Gamma$ is a fundamental domain for the left action of $\Gamma$ on $X$, there exists $\beta \in \Gamma$ such that 
$\beta he^{\R \lambda(\gamma)}o \cap \calD_\Gamma \neq \emptyset.$
We choose a time parameter $t\in \R$ on the geodesic such that 
$$\beta h e^{t \lambda(\gamma)}o \in \calD_\Gamma.$$
Set $g:= \beta he^{t \lambda(\gamma)}$ and consider $\gamma_0:= \beta \gamma \beta^{-1}\in [\gamma]$.
Then $$ \gamma_0 =  g \exp(\lambda(\gamma))k_\gamma g^{-1}.$$
Furthermore, $\Vert g\Vert \leq \dis_X(o,go) \leq C_\Gamma$.
Finally, applying Lemma \ref{lem-cartan-diff} to $\lambda(\gamma)=\underline{a}_{go}(\gamma_0)$ and $\underline{a}_{o}(\gamma_0)$, we deduce the first upper bound.
\end{proof}

\section{Equidistribution of flats}\label{sec:equidistribution}

For every loxodromic element $\gamma \in \Gamma^{lox}$, denote by $L_{\gamma}$ the quotient measure on $\Gamma \backslash G/M$ of $\calL_{[\gamma]}$ (Cf. \eqref{defin_mes-tore}).
Note that $L_{[\gamma]}$ is supported on $F_{[\gamma]}$ and is equal to the measure $L_{F_{[\gamma]}}$ given in the introduction. It is also given by the following construction: we push on $F_{[\gamma]}$, the restriction of $Leb_{\frak{a}}$ to any fundamental domain in $\frak{a}$ of the periods $\Lambda(F_{[\gamma]})$, by right $A$-action of the exponential of such a fundamental domain, starting from any base point on $F_{[\gamma]}$.
The construction is independent of both the choice of the fundamental domain and the base point on $F_{[\gamma]}$.

By Proposition \ref{lem-gaglox}, there is a bijection between $\calG_{lox}$ and $\calG(A)$. Let 
\[\calG_{lox}(\frakD_t)=\{[\gamma]\in\calG_{lox},\ \lambda(\gamma)\in\frakD_t \}. \]
By summing over the compact periodic orbits $F\in C(A)$ first, then summing over $Y \in \Lambda(F) \cap \frakD_t$, we deduce that
\begin{equation}\label{equ:conjugacy}
 \frac{1}{ \vol(D_t)}\sum_{[\gamma] \in \calG_{lox}(\frakD_t)}L_{[\gamma]}=\frac{1}{ \vol(D_t)}\sum_{F\in C(A)}|\Lambda(F)\cap \frakD_t |L_F,
 \end{equation}
 the measure on the right hand side is exactly the measure in the Theorem \ref{thm-introequid}. This formula is also a higher rank analogue of the first part of \eqref{equ-prime}.
Set
\[\calM^t_\Gamma:=\frac{ \vol(\Gamma\backslash G)}{  \vol( D_t)}  \sum_{[\gamma] \in \calG_{lox}(\frakD_t)} 
L_{[\gamma]}.\]

Let $m_{G/M}$ be the Haar measure on $G/M$, given by $\nu\otimes Leb_{\frak a}$ from Proposition \ref{prop-disintegration}.
Let $m_{\Gamma \backslash G/M}$ be the quotient measure on  $\Gamma\backslash G/M$. 
Theorem \ref{thm-introequid} is equivalent to the following one if $\Gamma$ is torsion free or if it acts on $G/M$ freely. 
\begin{theorem}\label{thm_cocompact-revet}
Let $\Gamma<G$ be a cocompact irreducible lattice which acts freely on $G/M$. Then there exists $u>0$ such that for any Lipschitz function $f$ on $\Gamma\backslash G/M$, as $t\rightarrow \infty$
\begin{equation}\label{eq_equ-plats}
\frac{ \vol(\Gamma\backslash G)}{ \vol(D_t)}\sum_{F\in C(A)}|\Lambda(F)\cap \frakD_t |\int f\dd L_F=\int f\ \dd \calM^t_\Gamma
=\int f\ \dd m_{\Gamma\backslash G/M}+O(e^{-u t}|f|_{Lip}),
\end{equation}
where the Lipschitz norm is with respect to the Riemannian distance $\dis_1$ on $\Gamma\backslash G/M$.
\end{theorem}

\begin{rem}
The constant $C_G$ equals $\|m_{\Gamma\backslash G/M} \|/ \vol(\Gamma\backslash G)$, which comes from the choice of $m_{G/M}=\nu\otimes Leb_{\frak a}$ and only depends on $G$. 

We can separate a Lipschitz function as the sum of its positive part and its negative part. So it is sufficient to prove Theorem \ref{thm_cocompact-revet} for non negative Lipschitz functions.
\end{rem}

We are going to prove Theorem \ref{thm_cocompact-revet} in this section.
Before starting the argument, we fix the parameters which will be used later. 
They come from Proposition \ref{prop-lemroblin}.
Choose $u_1>0$ small than $\min\{\epsilon_D,1\}/10$, where $\epsilon_D$ is the constant from Lemma \ref{volume_reste}. Set
\begin{equation}\label{equ-parameter0}
\varepsilon:= e^{-u_1 t} \text{ and }t_1 := 3 u_1 t. 
\end{equation}
Consider the decay rate function $u \mapsto \kappa(u) >0$ satisfying Lemma \ref{volume_reste} and the decay coefficient $\kappa>0$ given in Theorem \ref{theo-gorodnik-nevo}. Set 
\begin{equation}\label{equ-parameter}
    u_2 := \frac{1}{ 2\dim  (G/AM)} 
    \min \{ \delta_0 \kappa(6 u_1), \delta_0 \kappa, u_1 \}\text{ and } r:= e^{-u_2 t}.
\end{equation}

In this part we use $Lip_2$ to denote Lipschitz norm with respect to the product distance $d_2$ on $G/M$ or the product distance on $\calF^{(2)}$, according to which space the function lives on.

We lift everything to $G/M$ and prove a local version on $G/M$ in Section \ref{sec-loccor} and \ref{sec-corweyl}. This local version works for all irreducible lattices, which will be used in a forthcoming paper for non-cocompact lattices.
Then in Section \ref{sec-global}, we use the partition of unity to obtain a global version (Theorem \ref{thm_cocompact-revet}) on $\Gamma\backslash G/M$.

\subsection{Local convergence on corridors}\label{sec-loccor}

Recall the notation $\underline{a}_x(\gamma):= \adis(x,\gamma x)= \underline{a}(h_x^{-1} \gamma h_x )$.
For every $\gamma \in \Gamma$ such that $\underline{a}_x(\gamma)\in \frak{a}^{++}$, 
the geometric Weyl chamber based on $x$ containing $\gamma x$ (resp. $\gamma^{-1} x$)
determines $\gamma_x^+ \in \mathcal{F}$ (resp. $\gamma_x^-$). 

For $x\in X$ and $t>0$, we define the following measures on $\calF\times\calF$:

\begin{equation}\label{eq_nu1}
\nu_{x,1}^t:= \frac{ \vol(\Gamma\backslash G)}{ \vol(D_t)}\sum_{ \gamma \in \Gamma \cap D_t^{reg}(x)  } D_{\gamma_x^+} \otimes D_{\gamma_x^-},
\end{equation}

\begin{equation}\label{eq_nu2}
\nu_{x,2}^t:= \frac{ \vol(\Gamma\backslash G)}{ \vol(D_t)} \sum_{ \gamma \in \Gamma^{lox} \cap D_t^{reg}(x) } D_{\gamma^+} \otimes D_{\gamma^-}.
\end{equation}

Recall that $(\mu_x)_{x \in X}$ denotes the Patterson-Sullivan density given in Proposition \ref{prop-PS-BMS} and $\nu$ is the associated conformal measure on $\mathcal{F}^{(2)}$. 
Let $Lip^+_c( \mathcal{F}^{(2)}(x,r))$ be the space of positive compactly supported Lipschitz functions on $\calF^{(2)}(x,r)$.

\begin{lem}\label{lem-corridor-1GN}
Let $\Gamma$ be an irreducible lattice in $G$. Fix $x\in X$. Then for every test function $\psi \in Lip^+_c( \mathcal{F}^{(2)}(x,r))$ for every $t>C_4 \dis_X(o,x)$, there exists a function $E(t,\psi,x)$ such that 
\begin{equation}\label{eq-lem-corridor-1GN}
e^{-C_3r} \int \psi  \dd \nu - E(t,\psi, x) 
\leq \int \psi  \dd \nu_{x,1}^t
\leq \int \psi \dd \nu + E(t,\psi,x)
\end{equation}
where $E(x,\psi,t)=O( C_x Lip(\psi)  \vol(D_t)^{-\kappa})$ when $t \rightarrow \infty$.
\end{lem}
\begin{proof}
By Theorem \ref{theo-gorodnik-nevo}, we obtain the main term with the measure $\mu_x\otimes\mu_x$. Since $(\xi,\eta)\in\calF^{(2)}(x,r)$, so by Lemma \ref{lem-xietao}, we obtain 
\[ 1 \leq  f_x(\xi,\eta)^{-1} \leq e^{C_3r}. \]
Using the relation $\dd\nu(\xi,\eta)=\frac{\dd\mu_x(\xi)\dd\mu_x(\eta)}{f_x(\xi,\eta)}$, we deduce that $\int \psi \dd \mu_x \otimes \mu_x  \leq  \int \psi \dd \nu \leq  e^{C_3 r} \int \psi \dd \mu_x \otimes \mu_x.$
Hence the Lemma.
\end{proof}

\begin{lem}\label{lem-corridor-3RGN}
Let $\Gamma$ be a lattice in $G$.
Fix $x\in X$, for every $t \geq \frac{2 \log C_x}{u_1}$, for every test function $\psi \in Lip^+_c( \mathcal{F}^{(2)}(x,r))$,

$$
\bigg\vert \int \psi \dd \nu_{x,2}^t - \int \psi \dd \nu_{x,1}^t  \bigg\vert \leq 
\varepsilon Lip_2(\psi) \frac{\vert \Gamma \cap D_t(x) \vert  \vol(\Gamma \backslash G)}{ \vol(D_t)}
+ 3\Vert \psi \Vert_{\infty} \frac{ \vert \Gamma \cap D_t^{t_1}(x) \vert  \vol(\Gamma \backslash G) }{ \vol(D_t)},$$
where $\varepsilon$ and $t_1$ are given in \eqref{equ-parameter0}. 
\end{lem}

\begin{proof}
We split the difference between $\frac{ \vol(D_t)}{ \vol(\Gamma\backslash G)} \int \psi \dd \nu_{x,1}^t$ and $\frac{ \vol(D_t)}{ \vol(\Gamma\backslash G)} \int \psi \dd \nu_{x,2}^t$,
\begin{align*}
    \sum_{\gamma \in \Gamma \cap D_t^{reg}(x)} \psi (\gamma_x^+,\gamma_x^-) - \sum_{\gamma \in \Gamma^{lox} \cap D_t^{reg}(x)} \psi(\gamma^+,\gamma^-) 
    &=\sum_{\gamma \in \Gamma\cap D_t^{reg}(x)} \psi (\gamma_x^+,\gamma_x^-) - \sum_{\gamma \in \Gamma^{lox} \cap D_t^{reg}(x)} \psi(\gamma_x^+,\gamma_x^-)  \\
    & \hspace{2cm} + \sum_{\gamma \in \Gamma^{lox} \cap D_t^{reg}(x)} \psi(\gamma_x^+,\gamma_x^-) - \psi(\gamma^+,\gamma^-). 
\end{align*}

For the first term on the right hand side, note that $\Gamma^{lox} \subset \Gamma$, hence
$$  \sum_{\gamma \in \Gamma\cap D_t^{reg}(x)} \psi (\gamma_x^+,\gamma_x^-) - \sum_{\gamma \in \Gamma^{lox} \cap D_t^{reg}(x)} \psi(\gamma_x^+,\gamma_x^-) = \sum_{\gamma \in  (\Gamma \setminus \Gamma^{lox}) \cap D_t^{reg}(x)}\psi(\gamma_x^+,\gamma_x^-) .$$
Note that $t\geq t_1=3u_1 t>0$ since $u_1 \leq 1/10$, hence we have the following inclusion 
$$D_t^{reg}(x) \subset D_t^{t_1}(x) \sqcup \big( D_t(x) \setminus D_t^{t_1}(x)\big).$$ 
Using that $t \geq \frac{2 \log C_x}{u_1}$, we deduce that $t_1=3 u_1 t \geq t_0:= 2 \log C_x - 2\log \varepsilon = 2 \log C_x + 2 u_1 t$. 
Apply Proposition \ref{prop-lemroblin} to every 
every $\gamma \in D_t(x) \setminus D_t^{t_1}(x)$ such that $(\gamma_x^+,\gamma_x^-)\in\calF^{(2)}(x,r) $.
Any such element is loxodromic i.e. $D_t(x) \setminus D_t^{t_1}(x) \subset G^{lox}$.
Hence $\Gamma \cap \big( D_t(x) \setminus D_t^{t_1}(x) \big) \subset \Gamma^{lox}$ is a set of loxodromic elements. 
So the non-loxodromic must lie in $ \big(\Gamma \setminus \Gamma^{lox}\big) \cap D_t^{reg}(x) \subset D_t^{t_1}(x)$.
We deduce the following upper bound.
\begin{equation}\label{eq-lem-corridor-ineq1}
\bigg\vert \sum_{\gamma \in  (\Gamma \setminus \Gamma^{lox}) \cap D_t^{reg}(x)}\psi(\gamma_x^+,\gamma_x^-) \bigg\vert \leq  \Vert \psi \Vert_\infty \vert \Gamma \cap D_t^{t_1}(x) \vert .
\end{equation}

For the lower term, we split the sum over the partition $\Gamma^{lox}\cap (D_t(x) \setminus D_t^{t_1}(x))$ and $\Gamma^{lox} \cap D_t^{t_1}(x)$.
\begin{align*}
    \sum_{\gamma \in \Gamma^{lox} \cap D_t^{reg}(x)} \psi(\gamma_x^+,\gamma_x^-) -\psi(\gamma^+,\gamma^-)
    &= \underset{\gamma \in D_t(x) \setminus D_t^{t_1}(x)}{\sum_{\gamma \in \Gamma^{lox}}} \psi(\gamma_x^+,\gamma_x^-) -\psi(\gamma^+,\gamma^-) \\
    & \hspace*{1 cm} + \sum_{\gamma \in \Gamma^{lox} \cap D_t^{t_1}(x)\cap D_t^{reg}(x)} \psi(\gamma_x^+,\gamma_x^-) -\psi(\gamma^+,\gamma^-).
\end{align*}
We bound the lower term. 
\begin{equation}\label{eq-lem-corridor-ineq2}
\bigg\vert \sum_{\gamma \in \Gamma^{lox} \cap D_t^{t_1}(x)\cap D_t^{reg}(x)} \psi(\gamma_x^+,\gamma_x^-) -\psi(\gamma^+,\gamma^-) \bigg\vert \leq  2 \Vert \psi \Vert_\infty \vert \Gamma \cap D_t^{t_1}(x) \vert.
\end{equation}
By Proposition \ref{prop-lemroblin}, the elements $\gamma \in \Gamma \cap (D_t(x) \setminus D_t^{t_1}(x))$ with $(\gamma_x^+,\gamma_x^-)\in\calF^{(2)}(x,r) $ are loxodromic and their attractive and repelling points are at distance at most $\varepsilon$ of respectively $\gamma_x^\pm$.
Using that $\psi$ is Lipschitz and supported on $\calF^{(2)}(x,r)$, we bound above the last term.
\begin{equation}\label{eq-lem-corridor-ineq3}
 \bigg\vert \sum_{\gamma \in \Gamma \cap (D_t(x) \setminus D_t^{t_1}(x))} \psi(\gamma_x^+,\gamma_x^-) -\psi(\gamma^+,\gamma^-) \bigg\vert \leq \varepsilon Lip_2(\psi) \; \vert \Gamma \cap D_t(x)\vert .
\end{equation}
Finally, we use the triangle inequality, regroup the terms \eqref{eq-lem-corridor-ineq1}, \eqref{eq-lem-corridor-ineq2} and \eqref{eq-lem-corridor-ineq3}, then multiply everything by $\frac{ \vol(\Gamma \backslash G)}{ \vol(D_t)}$ to obtain the main upper bound.
\end{proof}

\subsection{From corridors to Weyl chambers}\label{sec-corweyl}

\begin{lem}\label{lem-couloirs-lipschitz}
Let $\widetilde{\psi} \in Lip^+_c( \widetilde{ \cal{F}^{(2)}} (x,r) )$ be a compactly supported non-negative, Lipschitz function and set 
$$ \psi := \int_{\frak{a}} \widetilde{ \psi} (.,. \; ; \; Y)  \dd Y.$$
Then $\psi \in Lip^+_c( \cal{F}^{(2)} (x,r))$ and the following norm bounds hold:
\begin{itemize}
\item[(a)] $Lip_2(\psi) \leq 2(2r)^{\dim  \frak{a}} Lip_2 (\widetilde{\psi}) $.
\item[(b)] $\Vert \psi \Vert_\infty \leq (2r)^{\dim  \frak{a}} \Vert \widetilde{\psi} \Vert_\infty.$
\end{itemize}
\end{lem}

For $x\in X$ and $t>0$, we define the following measure on $\calF^{(2)}\times\frak a$ by

\begin{equation}\label{eq_M2}
\cal{M}_{x,2}^t:= \frac{ \vol(\Gamma\backslash G)}{ \vol(D_t)}\sum_{ \gamma \in \Gamma^{lox}\cap D_t^{reg}(x) } \calL_\gamma = \nu_{x,2}^t \otimes Leb_{\frak{a}}. 
\end{equation}

\begin{lem}\label{lem_M2-haar}
Let $\Gamma$ be an irreducible lattice in $G$.
Fix $x\in X$, for every $t \geq \max\{\frac{2 \log C_x}{u_1}, C_4\dis_X(o,x) \}$, for every test function $\widetilde{ \psi }  \in Lip^+_c( \widetilde{\cal{F}}^{(2)}(x,r))$,  

\begin{align*}
\bigg\vert \int \widetilde{\psi} \;  \dd \cal{M}_{x,2}^t -
 \int & \widetilde{\psi} \; \dd m_{G/M} \bigg\vert \leq  \quad  C_3r \int \widetilde{\psi} \dd m_{G/M} \quad + \\
&  (2r)^{\dim  \frak{a}} \bigg( E(t,\widetilde{\psi},x) + 2\varepsilon Lip_2(\widetilde{\psi})  \frac{\vert \Gamma \cap D_t(x) \vert  \vol(\Gamma \backslash G) }{ \vol(D_t)} + 3 \Vert \widetilde{\psi} \Vert_\infty \frac{\vert \Gamma \cap D_t^{t_1}(x) \vert  \vol(\Gamma \backslash G)}{ \vol(D_t)}  \bigg), 
\end{align*}
where $E(x,\widetilde{ \psi},t)= O(C_x Lip(\widetilde{\psi})  \vol(D_t)^{-\kappa}  )$ as introduced in Lemma~\ref{lem-corridor-1GN} and $\varepsilon$, $t_1$ are given in \eqref{equ-parameter0}.
\end{lem}

\begin{proof}

We set $\psi (\xi^+,\xi^-) := \int_{\frak{a}} \widetilde{\psi}(\xi^+,\xi^- ;  v) \dd v.$
Using Fubini's theorem on the $\frak{a}$ coordinate and Proposition \ref{prop-disintegration} that $m_{G/M}=\nu \otimes Leb_{\frak a}$, we deduce that
\begin{align*}
 \int \widetilde{\psi} \dd \cal{M}_{x,2}^t = \int \psi \dd \nu_{x,2}^t \text{ and }  \int \widetilde{\psi} \dd m_{G/M} = \int \psi \dd \nu.
\end{align*}
We only need to bound $\int \psi \dd \nu_{x,2}^t - \int \psi \dd \nu$.
By definition of these measures,
$$
 \int \psi \dd \nu_{x,2}^t - \int \psi \dd \nu  = \quad \int \psi \dd \nu_{x,1}^t - \int \psi \dd \nu + \int \psi \dd \nu_{x,2}^t -  \int \psi \dd \nu_{x,1}^t .
$$
Using Lemma \ref{lem-corridor-3RGN} on the last term on the right, then Lemma \ref{lem-corridor-1GN}, the convexity inequality $e^{-r}-1 \geq -r$ and non-negativity of $\psi$ to the other term, we deduce the following bound. 
\begin{align*}
\bigg\vert \int \psi \dd \nu_{x,2}^t - \int \psi \dd \nu \bigg \vert \leq 
C_3 r &\int \psi \dd \nu + E(t, \psi, x) \\
&+ \varepsilon Lip_2(\psi) \frac{\vert \Gamma \cap D_t(x) \vert  \vol(\Gamma \backslash G)}{ \vol(D_t)}
+ 3\Vert \psi \Vert_{\infty} \frac{ \vert \Gamma \cap D_t^{t_1}(x) \vert  \vol(\Gamma \backslash G)}{ \vol(D_t)}.
\end{align*} 
By Lemma \ref{lem-couloirs-lipschitz} (a) (b), the Lipschitz constants and norms between $\psi$ and $\widetilde{\psi}$ satisfy 
$ Lip_2 (\psi) \leq 2 (2r)^{\dim  \frak{a}} Lip_2(\widetilde{\psi}) $
and $\Vert \psi \Vert_\infty \leq (2r)^{\dim  \frak{a}} \Vert \widetilde{\psi} \Vert_\infty$.  
We deduce the domination $E(t,\psi, x)= (2r)^{\dim  \frak{a}} O(Lip_2(\widetilde{\psi}) C_x  \vol(D_t)^{-\kappa})$ and abusing notation we write 
$$E(t, \psi, x) = (2r)^{\dim  \frak{a}} E(t, \widetilde{\psi},x). $$
Replacing the Lipschitz constants and norms in the upper bound by abuse of notation on $E(t,\psi,x)$ and lastly applying Fubini on the first term yields
\begin{align*}
\bigg\vert \int \psi \dd \nu_{x,2}^t - \int & \psi \dd \nu \bigg \vert \quad \leq \quad
C_3 r \int \widetilde{ \psi} \dd m_{G/M} \quad +\\
&(2r)^{\dim  \frak{a}} \bigg( E(t,\widetilde{\psi},x) +  2\varepsilon Lip_2(\widetilde{\psi}) \frac{\vert \Gamma \cap D_t(x) \vert  \vol(\Gamma \backslash G)}{ \vol(D_t)}
+ 3\Vert \widetilde{\psi} \Vert_{\infty} \frac{ \vert \Gamma \cap D_t^{t_1}(x) \vert  \vol(\Gamma \backslash G)}{ \vol(D_t)} \bigg).
\end{align*} 
\end{proof}

The measure in equidistribution is denoted by
\begin{equation}\label{eq_Mes}
\cal{M}^t:= \frac{ \vol(\Gamma\backslash G)}{ \vol(D_t)}  \sum_{ \gamma \in \Gamma^{lox}, \lambda(\gamma)\in\frakD_t}\calL_\gamma. 
\end{equation}

\begin{lem}\label{lem-cocom-weyl}
There exists $C>0$. Fix $x\in X$, for every test function $\widetilde{\psi}\in Lip^+_c( \widetilde{\cal{F}}^{(2)}(x,r))$,
\begin{equation}\label{eq-lem-cocom-weyl}
(1-Cr) \int  \widetilde{\psi} \dd \cal{M}_{x,2}^{t-2r} \leq  \int  \widetilde{\psi} \dd \cal{M}^t   \leq (1+Cr) \int  \widetilde{\psi} \dd \cal{M}_{x,2}^{t+2r}+\|\widetilde{\psi}\|_\infty\frac{|\Gamma\cap D_{t+2r}^{2r}(x)| \vol(\Gamma \backslash G)}{ \vol(D_t)}. 
\end{equation}
\end{lem}
\begin{proof}
By Lemma \ref{lem-cartan-lox}, for every loxodromic element $g \in G^{lox}$ such that $(g^+,g^-)\in \cal{F}^{(2)}(x,r)$ then 
$$ \Vert \lambda(g) - \underline{a}_x(g) \Vert \leq 2r .$$
Hence using triangle inequality we deduce the inclusions
\begin{align*}
\Gamma^{lox} \cap D_{t-2r}^{reg}(x)\cap \{\gamma|\ (\gamma^+,\gamma^-) &\in \cal{F}^{(2)}(x,r) \} \subset \\
& \big\lbrace \gamma \in \Gamma^{lox} \;\big\vert \; \lambda(\gamma)\in \frakD_t \; \text{and}\; (\gamma^+,\gamma^-) \in \cal{F}^{(2)}(x,r) \big\rbrace \\
& \hspace*{4cm} \subset (\Gamma^{lox} \cap D_{t+2r}^{reg}(x))\cup (\Gamma\cap D_{t+2r}^{2r}(x)) ,
\end{align*} 
here the set $\Gamma\cap D_{t+2r}^{2r}(x)$ is used to contain all the $\gamma$ in the middle set with $\underline{a}_x(\gamma)$ singular.
By integrating $\widetilde{\psi}$ over $\calL_\gamma$, summing and using that  $\widetilde{\psi}$ is supported on $\widetilde{\cal{F}}^{(2)}(x,r)$, we deduce
\begin{equation}\label{eq-lem-cocom-weyl-ineq1}
\frac{  \vol(D_{t-2r})}{ \vol(\Gamma \backslash G)} \int \widetilde{\psi} \dd \cal{M}_{x,2}^{t-2r}
\leq \frac{  \vol(D_t)}{ \vol(\Gamma \backslash G)} \int \widetilde{\psi} \dd \cal{M}^t
\leq \frac{  \vol(D_{t+2r})}{ \vol(\Gamma \backslash G)} \int \widetilde{\psi} \dd \cal{M}_{x,2}^{t+2r}
+\|\widetilde{\psi}\|_\infty|\Gamma\cap D_{t+2r}^{2r}(x)|.
\end{equation}
Finally, we multiply by $\frac{ \vol(\Gamma \backslash G)}{  \vol(D_t)}$, apply the local Lipschitz property of $t \mapsto \log ( \vol(D_t))$ (Lemma \ref{lem-vollip}).
\end{proof}

\begin{lem}\label{lem.changemetric}
Recall $\epsilon_0$ from Lemma \ref{lem-dist-compare}. For $0<s<\min\{\epsilon_0,(\log 2)/C_0 \}$ and any $z\in G/M$ and $x=\pi(z)\in X$, we have
\[B_1(z,s)\subset \widetilde{\calF^{(2)}}(x,s) \]
and for $\widetilde\varphi$ supported on $B(z,s)$
\[Lip_2\widetilde\varphi\leq C_x Lip \widetilde\varphi .\]
\end{lem}

\begin{proof}
By Lemma \ref{prop-corridor}, we have the first part.

By Lemma \ref{lem-dist-compare}, we have for $z_1,z_2\in B_1(z,s)$
\[\dis_1(z_1,z_2)\leq C_{\pi(z_1)}\dis_2(z_1,z_2)/4. \]
Now due to the definition of $C_x$, we have $C_{\pi(z_1)}\leq C_{\pi(z)}\exp(C_0\dis_X(\pi(z),\pi(z_1))\leq 2C_{\pi(z)}$. Therefore
\[\dis_1(z_1,z_2)\leq C_x \dis_2(z_1,z_2). \]
Then use the definition of Lipschitz norm.
\end{proof}

\paragraph{Local version}

\begin{prop}\label{prop:local}
Let $\widetilde\psi$ be a Lipschitz function supported on a ball $B(z,r)\subset G/M$ and let $x=\pi(z)\in X$. If $t> \max\{C_5, \frac{2(1-2\epsilon)}{\epsilon}, \frac{4(1-\kappa(2\epsilon)}{\kappa(2\epsilon)} \}\dis_X(o,x)$, then
\[ \bigg|\int\widetilde\psi\dd\calM^t-\int\widetilde\psi\dd m_{G/M} \bigg|=O\left(r\|\widetilde\psi\|_1+\big({C_x} \vol(D_t)^{-\kappa}+{\epsilon}+{ \vol(D_t)^{-\kappa(6u_1)}}\big)Lip_2(\widetilde\psi)\right). \]
\end{prop}

\begin{proof}
Due to Lemma \ref{lem_M2-haar} and \ref{lem-cocom-weyl}, we have

\begin{align*}
\pm \bigg(  \int &\widetilde{\psi} \dd \cal{M}^t  - \int \widetilde{\psi} \dd m_{G/M}  \bigg)  
\leq r (C_3+ C)  \int \widetilde{\psi}  \dd m_{G/M} \quad + \\
&  (2r)^{\dim  \frak{a}} \bigg( E(t\pm2r,\widetilde{\psi},x) +  
 2\varepsilon Lip_2(\widetilde{\psi})  \frac{\vert \Gamma \cap D_{t\pm 2r}(x) \vert  \vol(\Gamma \backslash G)}{ \vol(D_{t\pm 2r})} + 
 4 \Vert \widetilde{\psi} \Vert_\infty \frac{\vert \Gamma \cap D_{t \pm 2r}^{t_1}(x) \vert  \vol(\Gamma \backslash G)}{ \vol(D_{t \pm 2r})}  \bigg).
\end{align*}
Let's estimate the error term in the lower part. 
By Lemma \ref{lem-corridor-1GN} and Lemma \ref{lem-vollip}, we have
$$E(t\pm2r,\widetilde{\psi},x)=O(C_{x}Lip_2(\widetilde\psi) \vol(D_t)^{-\kappa}).$$

By Lemma \ref{lem-dtxt}, we have if $t>C_5\dis_X(o,x)$, then

\begin{equation}\label{equ-gammadt}
2\varepsilon Lip_2(\widetilde{\psi})  \frac{\vert \Gamma \cap D_{t\pm 2r}(x) \vert  \vol(\Gamma \backslash G)}{ \vol(D_{t\pm 2r})} 
= O \bigg({\varepsilon}  Lip_2(\widetilde{\psi})\bigg).
\end{equation}

Using that $t_1=3u_1t$, we get by applying Lemma \ref{lem-dtx}, for $t$ as in the hypothesis then
\begin{equation}\label{equ-gammadt0}
    3 \Vert \widetilde{\psi} \Vert_\infty \frac{\vert \Gamma \cap D_{t \pm 2r}^{t_1}(x) \vert  \vol(\Gamma \backslash G)}{ \vol(D_{t \pm 2r})} 
    = O \big( \|\widetilde\psi_\Gamma\|_\infty   \vol(D_t)^{-\kappa(6u_1)} \big) .
\end{equation}
By Combining the above inequalities, we complete the proof.
\end{proof}

\subsection{Proof of the equidistribution}\label{sec-global}
\begin{center}
\fbox{
\begin{minipage}{.7 \textwidth}
From now on, to the end of this section, we suppose that $\Gamma$ is an irreducible \textit{cocompact lattice} in $G$ which acts freely on $G/M$.
\end{minipage}
}
\end{center}

Fix a non-negative test function $\widetilde{\psi}_\Gamma \in Lip^+_c(\Gamma \backslash G/M)$.  
We want to prove the following convergence and dominate its rate
$$ \int \widetilde{\psi}_\Gamma \dd \cal{M}_\Gamma^t  \xrightarrow[t \rightarrow + \infty]{} \int \widetilde{\psi}_\Gamma \dd m_{\Gamma\backslash G/M}.$$

\paragraph{Partition of unity}

By applying Vitali's covering lemma to the collection $\{B(y,r/10)\}_{y\in \Gamma\backslash G/M}$, there exists a finite set $\{y_i\}_{i\in I}$ such that $B(y_i,r/10)$ are pairwisely disjoint and $\cup_{i\in I}B(y_i,r/2)$ is a covering of $\Gamma\backslash G/M$. By disjointness, we know $|I|\ll r^{-\dim(G/M)}$. Fix a partition of unity of $\frac{1}{r}$-Lipschitz functions associated to the open cover $\cup_{i\in I}B(y_i,r)$. 
For the function $\widetilde\psi_\Gamma$ on $\Gamma\backslash G/M$, we can write it as $\widetilde\psi_\Gamma=\sum_{i\in I}\widetilde\psi_{\Gamma,i}$ using the partition of unity.
For each $y_i$, we can find a lift $z_i$ in $G/M$ such that $\dis(o,z_i)$ is less than the diameter of $\Gamma\backslash G/M$. By Lemma \ref{lem.changemetric}, we know that for $x_i=\pi(z_i)\in X$
\[B(z_i,r)\subset \widetilde{\calF^{(2)}}(x_i,r). \]
We can take $t$ large such that $r=e^{-u_2t}$ is smaller then the injectivity radius of $\Gamma\backslash G/M$. Then the two balls $B(z_i,r)$ and $B(y_i,r)$ are homeomorphic.
Let $\widetilde\psi_i$ be the lift of $\widetilde\psi_{\Gamma,i}$ on $B(z_i,r)$. 

Furthermore, for every $i\in I$, the function $\widetilde{\psi}_i$ is Lipschitz and satisfies the following norm bounds:
\begin{itemize}
\item[(p1)] $Lip_2(\widetilde\psi_i)\leq C_{x_i} Lip\widetilde\psi_i \leq C_{x_i}(Lip\widetilde\psi_\Gamma+\frac{1}{r}\|\widetilde\psi_\Gamma\|_\infty)\leq \frac{C_{x_i}}{r} |\widetilde\psi_\Gamma |_{Lip},$
\item[(p2)] $\Vert \widetilde{\psi}_i \Vert_\infty  \leq \Vert \widetilde{\psi}_\Gamma \Vert_\infty,$
\item[(p3)] $\sum_{i\in I} \Vert \widetilde{\psi}_i \Vert_1  \leq \Vert \widetilde{\psi}_\Gamma \Vert_1,$
\end{itemize}
where the first inequality is due to Lemma \ref{lem.changemetric}.

For every $i \in I$, we can apply Proposition \ref{prop:local}. Then we use (p1) and (p2) to replace the Lipschitz norm of $\widetilde{\psi}_i$ by $\widetilde{\psi_\Gamma}$. By compactness, the $x_i$'s are in a bounded set, therefore the constants $\lbrace C_{x_i} \rbrace_{i\in I}$ are uniformly bounded.
Therefore we have for $t$ large
\begin{equation}\label{equ-et}
    |\int\widetilde\psi_i\dd\calM^t-\int\widetilde\psi_i\dd m_{G/M}|=O\left(r\|\widetilde\psi_i\|_1+\frac{1}{r}( \vol(D_t)^{-\kappa}+{\epsilon}+{ \vol(D_t)^{-\kappa(6u_1)}})|\widetilde\psi_\Gamma|_{Lip}\right).
\end{equation}

\paragraph{Global domination}
By the partition of unity, we have
\[\int \widetilde\psi_\Gamma \dd \calM_\Gamma^t=\sum_i\int\widetilde\psi_{\Gamma,i} \dd \calM_\Gamma^t=\sum_i\int\widetilde\psi_i \dd \calM^t \]
and
\[\int \widetilde\psi_\Gamma \dd m_{\Gamma\backslash G/M}=\sum_i\int\widetilde\psi_{\Gamma,i} \dd m_{\Gamma\backslash G/M}=\sum_i\int\widetilde\psi_i \dd m_{G/M}.\]
Therefore, by local dominations, $|I|\ll r^{-\dim(G/M)}$ and \eqref{equ-et}, we obtain
\begin{align*}
\int \widetilde\psi_\Gamma d\calM_\Gamma^t-\int \widetilde\psi_\Gamma \dd &m_{\Gamma\backslash G/M} = O \Bigg(  r  \sum_{i\in I}  \Vert \widetilde{\psi}_i \Vert_1
\\
&+  r^{-\dim(G/AM)} \bigg(\frac{ \vol(D_t)^{-\kappa}}{r}  | \widetilde{\psi}_{\Gamma} |_{Lip}
+ \frac{\varepsilon}{r} | \widetilde{\psi}_{\Gamma} |_{Lip}+\frac{ \vol(D_t)^{-\kappa(6 u_1)}}{r} \vert \widetilde{\psi}_\Gamma \vert_{Lip} \bigg) \Bigg). 
\end{align*}
Using (p3) and $\Vert \widetilde{\psi}_\Gamma \Vert_1 \leq \|m_{\Gamma\backslash G/M} \| \; \vert \widetilde{\psi}_\Gamma \vert_{Lip}$, we deduce that
$$
\int \widetilde\psi_\Gamma d\calM_\Gamma^t-\int \widetilde\psi_\Gamma \dd m_{\Gamma\backslash G/M} =
O\Bigg( \bigg(  r  
+  \frac{ \vol(D_t)^{-\kappa} + \varepsilon + \vol(D_t)^{-\kappa(6 u_1)}
}{r^{ \dim  (G/AM)+1 }} 
  \bigg) \vert \widetilde{\psi}_\Gamma \vert_{Lip} \Bigg). 
$$

Recall the choice of parameter in \eqref{equ-parameter} where $\varepsilon=e^{-u_1t}$ and $r=e^{-u_2t}$.
Collecting all the error terms together, we obtain that there exists $u>0$ such that 
\[ \Big|\int \widetilde\psi_\Gamma \dd\calM_\Gamma^t-\int \widetilde\psi_\Gamma \dd m_{\Gamma\backslash G/M} \Big|= O(e^{-u t}|\widetilde\psi_\Gamma|_{Lip}). \]

\section{Counting conjugacy classes}\label{sec:conjugacy}
\begin{center}
\fbox{
\begin{minipage}{.7 \textwidth}
In this section, we only consider ball domains, i.e. $D_t=K\exp(B_{\frak a}(0,t))K$.
\end{minipage}
}
\end{center}
\paragraph{Centralizer of singular hyperbolic elements}
We need to introduce $L_\Theta$ to study the structure of the centralizer of a semisimple element in $G$. See for example \cite[Section 8.2]{bochiAnosovRepresentationsDominated2019}.

Let $\Theta$ be a subset of simple roots $\Pi$. Taking the convention $\Theta^\complement:= \Pi \setminus \Theta$,
we set $$\mathfrak{p}_{\Theta}:= \mathfrak{g}_0 \oplus \bigoplus_{\alpha \in \Sigma^+} \mathfrak{g}_\alpha
\oplus \bigoplus_{\alpha \in \langle \Theta^\complement \rangle} \mathfrak{g}_{-\alpha} ,$$
where $ \langle \Theta^\complement \rangle$ is the set of weights generated by $\Theta^\complement$.
Denote by $P_\Theta$ the associated standard parabolic subgroup.
For the opposite parabolic subgroup, $P_\Theta^-$, its Lie algebra is given by
$$\mathfrak{p}_{\Theta}^-= \mathfrak{g}_0 \oplus \bigoplus_{\alpha \in \Sigma^+} \mathfrak{g}_{-\alpha}
\oplus \bigoplus_{\alpha \in \langle \Theta^\complement \rangle} \mathfrak{g}_{\alpha}. $$
The Lie algebra of the Levi group $L_\Theta = P_\Theta \cap P_\Theta^-$ is given by
$$\mathfrak{l}_{\Theta}:=\mathfrak{p}_{\Theta} \cap \mathfrak{p}_{\Theta}^- = \mathfrak{g}_0
\oplus \bigoplus_{\alpha \in \langle \Theta^\complement \rangle} \mathfrak{g}_{\alpha} \oplus \mathfrak{g}_{-\alpha}. $$
Let's define the $\Theta$-singular subspace $$\mathfrak{a}_{\Theta}= \cap_{\alpha \in \Theta^\complement} \ker \alpha,$$
which has real dimension $\vert \Theta\vert$. Recall $\mathfrak{m}= \mathfrak{k} \cap \mathfrak{g}_0$ and $\mathfrak{g}_0= \mathfrak{m} \oplus \mathfrak{a}$. 
Denote by
$$ \mathfrak{h}_\Theta= \mathfrak{m} \oplus \mathfrak{a}_{\Theta}^\perp \oplus \bigoplus_{\alpha \in \langle \Theta^\complement \rangle} \mathfrak{g}_{\alpha} \oplus \mathfrak{g}_{-\alpha} $$
the subalgebra where $\mathfrak{a}_{\Theta}^\perp$ is the orthogonal of $\mathfrak{a}_\Theta$ in $\mathfrak{a}$ for the Killing form and by $H_\Theta$ its associated reductive Lie group.
Then $A_\Theta$ and $H_\Theta$ commute and 
$$L_\Theta=A_\Theta H_\Theta.$$
Let $\delta_0(H_\Theta):=  \max_{Y\in B_{\mathfrak{a}}(0,1)} \sum_{\alpha\in \langle \Theta^\complement \rangle} \alpha(Y) $ be defined using the root space of $H_\Theta$.
Since $\Theta$ is non empty, there is a uniform gap i.e. there exists
$c_G>0$
such that 
\begin{equation}\label{equ-cG}
\delta_0(H_\Theta)\leq \delta_0-c_G
\end{equation}
for all non empty $\Theta \subset \Pi$.
\begin{proof}[Proof of Theorem \ref{corol_conjugacy_count}]
We want an upper estimate, when $t$ is large, of
$$[\Gamma](t)= \lbrace [\gamma] \in [\Gamma] \; \vert \; \lambda(\gamma) \in B_{\mathfrak{a}}(0,t) \rbrace.$$
First we estimate the number of conjugacy classes of singular (non-loxodromic) elements whose Jordan projection has norm less than $t$.
By Corollary \ref{lem-cartan-loxc} and Lemma \ref{count-reste} we have when $t$ is large enough
\begin{equation}\label{equ:vol-singular-part}
    \big|[\Gamma^{sing}](t)\big|\leq |\Gamma\cap D_{t+C}^{C}|\ll e^{(\delta_0-\epsilon) t}.
\end{equation}
Let us now provide an upper estimate for the number of loxodromic elements.

Recall that
$$[\Gamma^{lox}](t)= \lbrace [\gamma] \in [\Gamma^{lox}] \; \vert \; \lambda(\gamma) \in B_{\mathfrak{a}}^{++}(0,t) \rbrace.$$
Set $\kappa:=2\delta_0/c_G$, where $c_G$ is defined in \eqref{equ-cG}.
Recall that the set $[\Gamma^{lox}](t)$ is in bijection with
$$ \calG^t(A)= \big\lbrace (Y,F) \; \big\vert \; F \in C(A) \text{ and } Y\in \Lambda(F) \cap B_{\mathfrak{a}}^{++} (0,t)\big\rbrace.$$
We consider the subset of balanced periodic tori
$$\mathcal{B}^t(\kappa):= \Big\lbrace (Y,F)\in \calG^t(A) \; \Big\vert \; \Lambda(F)\cap B\Big(0, \frac{t}{\kappa} \Big) = \{0 \}  \Big\rbrace $$
and unbalanced periodic tori
$$\mathcal{U}^t (\kappa):= \Big\lbrace (Y,F)\in \calG^t(A) \; \Big\vert \; \Lambda(F)\cap B\Big(0, \frac{t}{\kappa} \Big) \neq \{0 \}  \Big\rbrace .$$
Note that $\mathcal{B}^t(\kappa)$ (resp. $\mathcal{U}^t(\kappa)$ ) projects into a subset of periodic tori in $C(A)$ of systole larger (resp. smaller) than $\frac{t}{\kappa}$ : the balanced (resp. unbalanced) tori.

 We prove that the amount of unbalanced tori is negligible compared to the balanced ones.
 Then, using Theorem \ref{thm_cocompact-revet} below will allow us to deduce the upper estimate
	\begin{equation}\label{equ:vol}
	\sum_{[\gamma]\in\cal [\Gamma^{lox}] (t) } \vol(F_{[\gamma]})=\vol(D_t)(1+O(e^{-\epsilon t})). 
	\end{equation}

Abusing notations, we identify the elements of $\mathcal{B}^t(\kappa)$ and $\mathcal{U}^t(\kappa)$ with the corresponding elements in $[\Gamma^{lox}](t)$. 
	\paragraph{For the balanced part} 
 By definition, for every $[\gamma] \in \mathcal{B}^t(\kappa)$, its periodic torus $F_{[\gamma]}$ is balanced i.e. its systole is greater than $t/\kappa$. 
 Hence, there exists $c>0$ such that
	\[\vol(F_{[\gamma]})>ct^{r}. \]
 Consequently, by \eqref{equ:vol}, we deduce that
\begin{equation}\label{eq:A-number}
	| \mathcal{B}^t(\kappa) |  \leq \sum_{[\gamma]\in [\Gamma^{lox}]  (t) } \frac{\vol(F_{[\gamma]})}{ct^{r}}\ll \frac{ \vol(D_t)}{t^r}.
\end{equation}

\paragraph{For the unbalanced part}
We prove that there is a negligible amount of unbalanced elements. 
Recall that $[\Gamma]( \frac{t}{\kappa} )= \lbrace [\gamma] \in [\Gamma] \; \vert \; \lambda(\gamma) \in B(0,\frac{t}{\kappa}) \rbrace$.
Since any unbalanced periodic torus has a period of size less than $\frac{t}{\kappa}$, the number of unbalanced periodic tori is bounded above by $ n_W\big\vert [\Gamma] \big(\frac{t}{\kappa}\big) \big\vert $ where $n_W \geq 1$ is the number of Weyl chambers in $\mathfrak{a}$.
Then by summing over first the unbalanced periodic tori and then their regular periods, we deduce the following upper bound
\begin{equation}\label{equ:unbalanced}
 \vert \mathcal{U}^t(\kappa) \vert \ll \sum_{ [\beta] \in [\Gamma](\frac{t}{\kappa}) } \big\vert \Lambda(F_{[\beta]}) \cap B_{\mathfrak{a}}^{++}(0,t) \big\vert .
 \end{equation}

Since $\Gamma$ is cocompact, Corollary \ref{lem-cartan-loxc} provides an upper bound for the summation term
\begin{equation}\label{equ:beta gamma}
\Big|[\Gamma]\Big(\frac{t}{\kappa}\Big)\Big|\leq \big|\Gamma\cap D_{\frac{t}{\kappa}+C}\big|\ll t^{r}e^{\delta_0 t/\kappa}.
\end{equation}

Now for the summand, for each $[\beta] \in [\Gamma](\frac{t}{\kappa})$,
there exists a unique non-empty $\Theta \subset \Pi$ such that $\lambda(\beta) \in \mathfrak{a}_{\Theta}^{++}$.
It is given by the biggest subset $\Theta \subset \Pi$ such that $\alpha (\lambda(\beta))>0$ for all $\alpha \in \Theta$.
Denote by $G_\beta$ the centralizer of $\beta$ in $G$.
It is conjugated to a closed subgroup of the Levi $L_\Theta$ (which is reductive) i.e. there exists $g \in G$ such that $G_\beta < g^{-1} L_\Theta g$. By Corollary \ref{lem-cartan-loxc} and since $Z_K(\exp(\lambda(\beta))\subset H_{\Theta}$, we may assume that $\|g\|\leq C_\Gamma$ by choosing an appropriate element in the conjugacy class of $\beta$ and abusing notations.
By Selberg's lemma (\cite[Lemma 1.10]{pr}), $\Gamma_\beta$ is a cocompact lattice of $G_\beta$. 

Now, counting only the loxodromic elements, by Corollary \ref{lem-cartan-loxc}, we have 
$$\big\vert \Lambda(F_{[\beta]}) \cap B_{\mathfrak{a}}^{++}(0,t) \big\vert\leq |\Gamma_\beta \cap D_{t+C}| = |\Gamma_\beta\cap (D_{t+C}\cap G_\beta)|.$$ 
Then take some small ball $\calO_\epsilon \subset G$ of injective image in $\Gamma \backslash G$, we have
\[ |\Gamma_\beta\cap (D_{t+C}\cap G_\beta)| \leq \frac{ \vol(\calO_\epsilon D_{t+C}\cap G_\beta)}{ \vol(\calO_\epsilon \cap G_{\beta})}\leq \frac{ \vol(D_{t+C'}\cap G_\beta)}{ \vol(\calO_\epsilon\cap G_\beta)}. \]
It remains to estimate $\vol(D_{t+C'}\cap G_\beta)$.
Since $g G_\beta g^{-1} < L_\Theta$ and $\|g\|\leq C_\Gamma$, it is dominated by the volume growth of the Levi, i.e. there exists $C''>0$ such that
$$\vol(D_{t+C'}\cap G_\beta) \leq \vol(D_{t+C''}\cap g G_\beta g^{-1}) \leq \vol( D_{t+C''} \cap L_\Theta ).$$
By the same computation as in \cite[Theorem 6.2]{knieper_asymptotic_1997}, we obtain that 
\[ \vol(D_{t+C''}\cap L_\Theta)\ll t^{r}\exp\big(\delta_0(H_\Theta) t\big), \]
where $H_\Theta$ is the semisimple part of the Levi $L_\Theta$. 
Due to \eqref{equ-cG}, for all $[\beta] \in [\Gamma](\frac{t}{\kappa})$,
\begin{equation}\label{equ:gamma beta}
  \big\vert \Lambda(F_{[\beta]}) \cap B_{\mathfrak{a}}^{++}(0,t) \big\vert \ll t^{r}\exp((\delta_0-c_G) t). 
\end{equation}

Finally, by combining \eqref{equ:unbalanced}, \eqref{equ:beta gamma} and \eqref{equ:gamma beta} and our choice of $\kappa$, we get
\begin{equation}\label{eq:B}
|\mathcal{U}^t (\kappa)|\ll t^{2r} \exp \Big( t\big( \delta_0+ \frac{\delta_0}{\kappa} -c_G \big) \Big)\ll \vol(D_t)^{1-\epsilon}.
\end{equation}

\paragraph{Back to the main estimate}
	
	Combining \eqref{equ:vol-singular-part}, \eqref{eq:A-number} and \eqref{eq:B}, we deduce the upper estimate because the singular elements $\big|[\Gamma^{sing}](t) \big|$ and the unbalanced periodic tori $\big\vert \mathcal{U}^t(\kappa) \big\vert$ are negligible compared to the balanced periodic tori
	\[ \big|[\Gamma](t) \big| \sim  \big\vert \mathcal{B}^t(\kappa)   \big\vert  \ll \frac{\vol(D_t)}{t^r} \ll t^{-\frac{r+1}{2}} e^{\delta_0 t}  . \]
	Therefore, the proof is complete.
\end{proof}

\section{Appendix}
\label{appendix}

\subsection{Weyl subgroups and parabolic subgroups of G}

Recall the notion of parabolic subgroups $P_\Theta$, Levi subgroups $L_\Theta$ and $A_\Theta$ is the group corresponding to the Lie algebra $a_\Theta$ with $\Theta$ a subset of simple roots $\Pi$ from Section \ref{sec:conjugacy}. Let $W_\Theta$ be the Weyl subgroup generated by reflections $s_\alpha$ for $\alpha\in\Theta^\complement$.
\begin{rem}
The conventions are made such that $P_\Pi$ is the Borel subgroup, which is different from \cite{bourbaki_lie_2004}.

$$\begin{array}{c l c l c}
    \Pi & \supset  & \Theta & \supset  & \emptyset  \\
    \lbrace e_W \rbrace =W_{\Pi} & \subset & W_{\Theta} & \subset & W_\emptyset = W \\
    B= P_{\Pi}   & \subset & P_{\Theta} = B W_\Theta B  & \subset & P_\emptyset = G \\
    AM = L_{\Pi} & \subset & L_\Theta & \subset &L_\emptyset = G \\
    &&&& \\
    A^+=A^{+}_\Pi & \supset & A^{+}_{\Theta} & \supset &    A^{+}_{\emptyset} = \lbrace e_A \rbrace 
\end{array}$$
\end{rem}

\begin{prop}[See \cite{bourbaki_lie_2004}{Chap IV, \S2 section 5 Proposition 2}]\label{prop-bruhat}
Let $\Theta_1, \Theta_2 \subset \Pi$ and $w\in W$.
Then $$ P_{\Theta_1}wP_{\Theta_2} = BW_{\Theta_1} w W_{\Theta_2}B.$$
\end{prop}

\begin{corol}\label{corol-bruhat}
For all $\Theta_1, \Theta_2 \subset \Pi$, the map $w\in W \rightarrow BwB \in B \backslash G/B$ induces, by passing the quotient, the following bijection.
\begin{align*}
    W_{\Theta_1}\backslash W / W_{\Theta_2} &\longrightarrow P_{\Theta_1} \backslash G/ P_{\Theta_2} \\
    W_{\Theta_1}w W_{\Theta_2} &\longmapsto P_{\Theta_1} w P_{\Theta_2}.
\end{align*} 

\end{corol}

Let $\tau$ be the Cartan involution of $G$ (see \cite{helgason1978differential}): it is an automorphism of $G$ which acts on $\frak a$ by $-id$, and on $A$ by $a\mapsto a^{-1}$. The involution $\tau$ induces an involution $\iota:\Pi\to\Pi$, such that $\tau(P_\Theta)$ is conjugated to $P_{\iota(\Theta)}$. We have $P_\Theta^-=\tau(P_\Theta)$; by definition it is a parabolic subgroup of type $\iota(\Theta)$.

\begin{prop}\label{prop-bruhat-dyn}\label{prop-append-parab}
Let $\Theta \subset \Pi$. 
Then
\begin{itemize}
    \item[(i)] The sequence $\big(a^{-n}ga^n\big)_{n \geq 1}$ is bounded for all $a\in A^{++}_\Theta$ if and only if $g\in P_\Theta$.
    \item[(i')] The sequence $\big(a^{n}g'a^{-n}\big)_{n \geq 1}$ is bounded for all $a\in A^{++}_\Theta$ if and only if $g'\in P^-_\Theta$.
    \item[(ii)] The sequence $\big(a^{-n}ga^n\big)_{n \in \mathbb{Z}}$ is bounded for all $a\in A^{++}_\Theta$ if and only if $g\in L_\Theta$.
\end{itemize}
\end{prop}
\begin{proof}
This direction $(\Leftarrow)$ is well-known. 

For (i), $(\Rightarrow)$, due to Corollary \ref{corol-bruhat},
we write $g=p_1 w_g p_2$ where $p_1,p_2 \in P_\Theta$ and $w_g\in W$.
Due to $(\Leftarrow)$ of (i), the sequences $a^{-n}p_ia^n$ are bounded as $n\rightarrow+ \infty$. 
The behavior of the sequence $a^{-n}ga^n$ when $n \rightarrow +\infty$ is the same as $a^{-n} w_g a^n$.
We write $a=e^v$ with $v\in \mathfrak{a}^{++}_\Theta$, then
$$ a^{-n}w_g a^n = e^{ n( \mathrm{Ad}(w_g)v - v )} w_g .$$
We conclude that the sequence $a^{-n}ga^n$ is bounded when $n \rightarrow +\infty$ if $\mathrm{Ad}(w_g)v=v$ i.e. $w_g \in W_\Theta$.
We finish by noticing $g\in p_1 W_\Theta p_2 = P_\Theta$.

For (i'), applying $\tau$ to (i) we obtain that the sequence $(\tau(a)^{-n}\tau(g) \tau(a)^n)_{n\geq 1}$ is bounded for all $g\in P_\Theta$. Due to $\tau(a)=a^{-1}$ and $\tau(P_\Theta)=P_{\Theta}^-$ we obtain (i'). 

The point (ii) follows from (i) and (i') by using $P_\Theta\cap P_\Theta^- = L_\Theta$.
\end{proof}

\subsection{Proof of Theorem \ref{theo-gorodnik-nevo}}

We give a proof of Theorem \ref{theo-gorodnik-nevo} by redoing the proof of Theorem 7.1 in \cite{gorodnikCountingLatticePoints2012} for Lipschitz functions. Here we have one notation issue, the quotient $\Gamma$ is on the left $\Gamma\backslash G$ to be consistent with the main part of the article, which is different from that in \cite{gorodnikCountingLatticePoints2012}.  
Fix notation $\vol$ and $\dd m_{\Gamma\backslash G}=\dd \vol/\vol(\Gamma\backslash G)$, which is a probability measure.

\paragraph{Quantitative mean ergodic theorem}
The main engine to obtain equidistribution is the quantitative mean ergodic theorem on $L^2(\Gamma\backslash G)$. For an absolutely continuous probability measure $\beta$ on $G$, let $\pi(\beta)f=\int \pi(g)f d\beta(g)$ where $\pi(g)$ is the right representation of $G$ on $\Gamma\backslash G$. By Theorem 4.5 in \cite{gorodnikCountingLatticePoints2012}, we have
\begin{equation}\label{equ.meanergodic}
\bigg\|\pi(\beta)f-\int f\bigg\|_{2}\leq C_q\|\beta\|_{q}^{1/n(G,\Gamma)} \|f\|_{2}, 
\end{equation}
where $n(G,\Gamma)$ is an integer depending on $G$, $\Gamma$ and $q$ is any constant in $[1,2)$ such that $\|\beta\|_q<\infty$. We explain  
why Theorem 4.5 in \cite{gorodnikCountingLatticePoints2012} works in the setting of irreducible lattices.

\paragraph{Verification of conditions in Theorem 4.5 in \cite{gorodnikCountingLatticePoints2012}}
 
There are two conditions, the group is simply connected as an algebraic group, the lattice satisfies that the representation of $G$ on $L^2_0(\Gamma\backslash G)$ is $L^{p+}$ for some $p\geq 2$.

For the first condition and for real linear algebraic semisimple Lie groups, we do not need that the group is simply connected. This condition is only required for the $p$-adic case, as can be observed in the proof of Theorem 4.5. 

Then the crucial condition is the second one. From the parameter $p$ we can compute the rate $n(G,\Gamma)$ in \eqref{equ.meanergodic}, which equals 1 if $p=2$ and $2\lceil p/4\rceil$ if $p>2$. In \cite{ohUniformPointwiseBounds2002}, an explicit estimate on $p$ is provided for certain cases. Additionally, in \cite[Remark 4.6]{gorodnikCountingLatticePoints2012}, the authors explained several cases where the second condition holds. We explain this condition also holds when $G$ is a connected real linear algebraic semisimple Lie group and $\Gamma$ is an irreducible lattice. This fact is certainly known among experts in the field and we include it for the sake of completeness.

The proof that the representation of $G$ on $L^2_0(\Gamma\backslash G)$ is $L^{p+}$ is a two-step process. The first step is to prove that we have a strong spectral gap, that is, each simple factor $G_j$ of $G$ has no almost invariant vector on $L^2_0(\Gamma\backslash G)$. The second step is to use the strong spectral gap to prove the representation is $L^{p+}$, which is well explained in \cite[Theorem 3.4]{kleinbock_margulis} and references therein. Hence we only need to explain why the strong spectral gap holds.

In Kelmer-Sarnak \cite[Page 284-285]{kelmerStrongSpectralGaps2009}, they explained the strong spectral gap for $G'=G_1'\times\cdots \times  G_r'$, where each $G_j'$ is a non-compact simple Lie group with trivial centre and $\Gamma$ an irreducible lattice. We shall employ \cite[Lemma 3.1]{kleinbock_margulis} (due to Furman-Shalom and Kleinbock-Margulis) to transfer the spectral gap to finite coverings, thereby deducing the strong spectral gap for semisimple Lie group $G$ without compact factor from this version.

Return to a connected real linear algebraic semisimple group $G$ without compact factor. There exist non-compact simple Lie groups $G_1,\cdots, G_r$ and a map $\pi_1:G_1\times \cdots\times G_r\rightarrow G$ with finite central kernel (see for example \cite[\S 22]{borel_linear_1991}). There also exists a quotient map $\pi_2$ from $G_1\times \cdots\times G_r$ to $G':=G_1'\times \cdots\times G_r'$ such that $G_j'$ has trivial centre. Letting $\Gamma'=\pi_2\pi_1^{-1}\Gamma$, we obtain an irreducible lattice $\Gamma'$ in $G'$. Then $\Gamma\backslash G\simeq \pi_1^{-1}\Gamma\backslash(G_1\times\cdots\times G_r)$ is a finite covering of $\Gamma'\backslash G'$. Applying the results of \cite{kelmerStrongSpectralGaps2009}, we know that each simple factor of $G_j$ has no almost invariant vector in $L^2_0(\Gamma'\backslash G')$. Therefore, invoking Lemma 3.1 in \cite{kleinbock_margulis}, we deduce that $G_j$ has no almost invariant vector in $L^2_0(\Gamma\backslash G)$.

\paragraph{Lipschitz well rounded domains}

For all $\epsilon>0$, denote by $\calO_\epsilon$ the ball of radius $\epsilon$ centered at identity in $G$.
Let $\epsilon_\inj>0$ be a constant such that for all $\epsilon\in (0,\epsilon_\inj)$, the map $\calO_\epsilon  \rightarrow \calO_{\epsilon} \Gamma \subset \Gamma\backslash G  $ is injective.

For a family of domains $(S_t)_{t>0}$, we call it Lipschitz well rounded if 
there exist $\epsilon_0>0,\ C>1$ such that for all $\epsilon<\epsilon_0$, there exist domains $S^+_t$, $S^-_t$ and  for all $t>1$
\begin{align}
   \label{equ.domain} & S^-_{t-\epsilon}\subset \cap_{g,h\in\calO_\epsilon}gS_th,\  \calO_\epsilon S_t\calO_\epsilon\subset S^+_{t+\epsilon}\\
    \label{equ.lip}&\vol(S^+_{t+\epsilon}-S^-_{t-\epsilon})\leq C\epsilon \vol(S_t).
\end{align}

\paragraph{Angular equidistribution for regular elements}
Let $\widetilde{A}^\delta=\{\exp(a),\ a\in\frak a^{++},\ d(a,\partial \frak a^{++})\geq\delta \}$.
\begin{theorem}\label{theo-gorodnik-nevo1}
\hypG
Let $\Gamma <G$ be an irreducible lattice. There exist $\kappa>0$ and $C_6>0$ only depending on $n(G,\Gamma)$ (from \eqref{equ.meanergodic}) and $G$. Let 
$(S_t)_{t >0}$ be Lipschitz well rounded and $S_t\subset K\widetilde{A}^\delta K$. 
There exists $C_7>0$ depending on $n(G,\Gamma)$, $G$ and the family $(S_t)_{t>0}$.
Then for all Lipschitz test functions $\psi \in Lip(\cal{F}^{(2)})$, there exists $E(t,\psi)=O(Lip(\psi)  \vol(S_t)^{-\kappa})$ when $t >\{ C_6 |\log\epsilon_\inj|,C_7\}$ such that

$$ \frac{1}{ \vol(S_t)} \sum_{\gamma \in S_t\cap \Gamma} \psi(\gamma_o^+,\gamma_o^-) = \frac{1}{ \vol(\Gamma \backslash G)}   \int_{ \calF\times\calF  } \psi \dd \mu_o\otimes\mu_o + E(t,\psi) ,$$
where all the implied constants only depending on $G$ and $n(G,\Gamma)$.
\end{theorem}

In order to obtain the domains we are interested, we need to add singular elements.

\begin{corol}\label{theo-gorodnik-nevo2}
\hypG
Let $\Gamma <G$ be an irreducible lattice. There exist $\kappa>0$ and $C_6>0$ only depending on $n(G,\Gamma)$ (from \eqref{equ.meanergodic}) and $G$. Let 
$(D_t)_{t >0}$ be one of the two type of domains. There exists $C_7>0$ depending on $n(G,\Gamma)$, $G$ and the family $(D_t)_{t>0}$.
 
Then for all Lipschitz test functions $\psi \in Lip(\cal{F}^{(2)})$, there exists $E(t,\psi)=O(Lip(\psi)  \vol(D_t)^{-\kappa})$ when $t >\{ C_6 |\log\epsilon_\inj|,C_7\}$ such that

$$ \frac{1}{ \vol(D_t)} \sum_{\gamma \in D_t^{reg}\cap \Gamma} \psi(\gamma_o^+,\gamma_o^-) = \frac{1}{ \vol(\Gamma \backslash G)}   \int_{ \calF\times\calF  } \psi \dd \mu_o\otimes\mu_o + E(t,\psi) ,$$
where all the implied constants only depending on $G$ and $n(G,\Gamma)$.
\end{corol}

\begin{proof}[Proof that Corollary \ref{theo-gorodnik-nevo2} $\Rightarrow$ Theorem \ref{theo-gorodnik-nevo}]
Due to \eqref{equ.gammax} $\gamma_x^+=h_x(h_x^{-1}\gamma h_x)^+_o$, we apply Theorem \ref{theo-gorodnik-nevo2} to the lattice $h_x^{-1}\Gamma h_x$ and the Lipschitz function $\psi'(\cdot,\cdot):=\psi(h_x\cdot,h_x\cdot)$. This is the reason that we need a uniformed version for lattices $h_x^{-1}\Gamma h_x$ and we made dependence of constants in Theorem \ref{theo-gorodnik-nevo2} more transparent. The constant $n(G,h_x^{-1}\Gamma h_x)$ is the same as $n(G,\Gamma)$ due to invariance of the Haar measure.
For $\epsilon_\inj$ of $h_x^{-1}\Gamma h_x$, we have
\[\inf_{\gamma\in\Gamma-\{e\}}\dis_G(o,h_x^{-1}\gamma h_x)\geq e^{-C\dis_X(o,x)}\inf_{\gamma\in\Gamma-\{e\}}\dis_G(o,\gamma). \]
By Lemma \ref{lem-actiong}, the action of $h_x$ on $\calF$ is $C_x$ Lipschitz. From these, we obtain Theorem \ref{theo-gorodnik-nevo}.
\end{proof}

\textbf{Step 1}: 
The first step is to transfer the counting problem to integrals, which can be treated by the mean ergodic theorem.

\begin{lem}[Effective Cartan decomposition, Proposition 7.3 in \cite{gorodnikCountingLatticePoints2012}, first appeared in \cite{gorodnikStrongWavefrontLemma2010}]\label{lem_effective_cartan}
There exist $\delta>0$ and $l_0,\epsilon_1>0$. If $\epsilon<\epsilon_1$, then for $g=k_1 ak_2\in K \widetilde{A}^\delta K$, we have
\[\calO_\epsilon g\calO_\epsilon\subset (\calO_{l_0\epsilon}\cap K)k_1M(\calO_{l_0\epsilon}\cap A)ak_2(\calO_{l_0\epsilon}K). \]
\end{lem}
For ease of notation, when there is no confusion, we will use $k_1,a,k_2$ to denote elements come from the Cartan decomposition  $g=k_1ak_2$. 
Notice that by identifying $\calF$ with $K/M$, we have $k_1M=\gamma_o^+$ and $k_2^{-1}M=\gamma_o^-$.
Let $$\rho_t(g)=\mathbbm{1}_{S_t}(a)\psi(k_1,k_2),$$
where $\psi(k_1,k_2)=\psi(k_1M,k_2^{-1}M)=\psi(g_o^+,g_o^-)$.

We introduce two auxiliary functions, which is the replacement of Lipschitz well-roundness of sets in \cite{gorodnikCountingLatticePoints2012}. Recall 
\[Lip \;\psi=\max\bigg\{|\psi|_\infty, \sup_{x\neq y}\frac{|\psi(x)-\psi(y)|}{d(x,y)} \bigg\}.\] 
Let
\begin{align*}
    \rho_{t,\epsilon}^+(g)&=\mathbbm{1}_{S^+_{t+\epsilon}}(g)(\psi(k_1,k_2)+(Lip\psi)l_0\epsilon)\\
    \rho_{t,\epsilon}^-(g)&=\mathbbm{1}_{S^-_{t-\epsilon}}(g)\max\{\psi(k_1,k_2)-(Lip\psi)l_0\epsilon,0\}.
\end{align*}
From the definition, we know $\rho^-_{t,\epsilon}\leq \rho_t\leq \rho^+_{t,\epsilon}$. 
\begin{lem}
For $g\in\calO_\epsilon\gamma\calO_\epsilon$ with $\epsilon\leq \epsilon_1$ we obtain
\begin{equation}\label{equ_rho-+}
\rho^-_{t,\epsilon}(g)\leq \rho_t(\gamma)\leq \rho^+_{t,\epsilon}(g).
\end{equation}
\end{lem}
\begin{proof}
If $\rho^-_{t,\epsilon}(g)\neq 0$, then $g\in S^-_{t-\epsilon}$. By $\gamma\in \calO_\epsilon g\calO_\epsilon$, we obtain 
\[a(\gamma)\in (\calO_{l_0\epsilon}\cap A)a(g)\cap S_t. \]
By \eqref{equ.domain}, we know $\mathbbm{1}_{S_t}(a(\gamma))=1$. By Lemma \ref{lem_effective_cartan} and Lipschitz property of $\psi$, we obtain  
\[\psi(k_1(\gamma),k_2(\gamma))\geq \psi(k_1(g),k_2(g))-(Lip\psi)l_0\epsilon. \]
This proves the left hand side. For the other side, the proof is similar.
\end{proof}
Take $\mathbbm{1}_\epsilon=\frac{1}{\vol(\calO_\epsilon)}\mathbbm{1}_{\calO_\epsilon}$ be the normalized characteristic function of $\calO_\epsilon$.
Let $\varphi_\epsilon(g\Gamma)=\sum_{\gamma\in\Gamma}\mathbbm{1}_\epsilon(g\gamma)$. The counting is connected to integral by the following.
\begin{lem}\label{lem_rhot}
For $h$ in $\calO_\epsilon$ with $\epsilon\leq \epsilon_1$, we have
\begin{equation}
    \int \varphi_\epsilon(g^{-1}h\Gamma)\rho^-_{t,\epsilon}(g)\dd \vol(g)\leq \sum_{\gamma\in \Gamma}\rho_t(\gamma)\leq \int \varphi_\epsilon(g^{-1}h\Gamma)\rho^+_{t,\epsilon}(g)\dd \vol(g).
\end{equation}
\end{lem}
\begin{proof}
By using \eqref{equ_rho-+}, the proof is almost the same as Lemma 2.1 in \cite{gorodnikCountingLatticePoints2012}.
\end{proof}
\textbf{Step 2:} This step will estimate the error terms in the mean ergodic theorem.

We want to apply the mean ergodic theorem to probability measures $\frac{\rho_{t,\epsilon}^{\pm}}{\int\rho_{t,\epsilon}^{\pm}}$. Before doing so, we need to compute some integrals. The computation is a bit tedious. \textbf{This step is to verify similar stable mean ergodic theorems, the main consequence is \eqref{equ_rho-} and \eqref{equ_rho+}}.

Let's first compute the difference.

\begin{lem} We have for $\epsilon<\epsilon_0$
\begin{equation}\label{equ_diff}
\int \rho^+_{t,\epsilon} \dd \vol-\int \rho^-_{t,\epsilon} \dd \vol\ll \epsilon \bigg(  \int\psi+l_0 (Lip\psi) \bigg)\vol(S_t).
\end{equation}
\end{lem}

\begin{proof} By definition
\begin{equation*}
\begin{split}
    &\int \rho^+_{t,\epsilon} \dd \vol-\int \rho^-_{t,\epsilon} \dd \vol\\
    &\leq \vol\big(S^+_{t+\epsilon}\big)\bigg(\int\psi+l_0\epsilon(Lip\psi)\bigg)-\vol\big(S^-_{t-\epsilon}\big)\bigg(\int\psi-l_0\epsilon(Lip\psi)\bigg)\\
    &=\bigg( \vol\big(S^+_{t+\epsilon}\big)-\vol\big(S^-_{t-\epsilon}\big)\bigg) \int\psi+l_0\epsilon(Lip\psi)\Big(\vol\big(S^+_{t+\epsilon}\big)+\vol\big(S^-_{t-\epsilon}\big)\Big)\\
    &\ll \bigg(\epsilon  \int\psi+l_0\epsilon(Lip\psi)\bigg)(\vol(S_t)+\epsilon ),
\end{split}
\end{equation*}
where the last inequality is from Lipschitz well-roundness \eqref{equ.lip}. 
\end{proof}

Let $\beta_{t,\epsilon}^{\pm}:=\rho_{t,\epsilon}^{\pm}/\int \rho_{t,\epsilon}^{\pm}$.

\begin{lem}
For $\epsilon<\min\{\int\psi/2l_0Lip\psi, \epsilon_0, 1/2C\}$, $t>1$ and $f\in L^2(\Gamma\backslash G)$
\begin{equation}\label{equ_rho-}
    \|\pi(\beta^{-}_{t,\epsilon})f-\int f\|_{2}\leq E(t) \|f\|_{2},
\end{equation}
with 
\begin{equation}\label{equ_et}
E(t)=(\frac{C}{\vol(S_t)^{q-1}}\frac{(Lip \psi)^q}{(\int\psi)^q})^{\kappa_2},
\end{equation}
 $\kappa_2=1/qn(G,\Gamma)$ and $C>0$ only depending on $G$.

For $\epsilon\leq \epsilon_0$, $t>1$ and $f\in L^2(\Gamma\backslash G)$
\begin{equation}\label{equ_rho+}
    \|\pi(\beta^{+}_{t,\epsilon})f-\int f\|_{2}\leq  E(t) \|f\|_{2}.
\end{equation}
\end{lem}

The main difference between the above two inequalities is that for $\beta^+_{t,\epsilon}$, we don't need an extra condition of $\epsilon$ depending on $\psi$.

\begin{proof}
We compute the integral of $\rho_{t,\epsilon}^-$. We have
\begin{align*}
    \int \rho_{t,\epsilon}^-\dd \vol\geq \vol(S^-_{t-\epsilon})(\int\psi-(Lip\psi)l_0\epsilon).
\end{align*}
Due to \eqref{equ.lip}, we obtain
\[  \vol(S^-_{t-\epsilon})\geq (1-C\epsilon )\vol(S_t).\]
Hence if $\epsilon\leq 1/2C$, then
\[ \int \rho_{t,\epsilon}^-\dd \vol\gg \vol(S_t)(\int\psi-(Lip\psi)l_0\epsilon). \]
Therefore if $\epsilon\leq \min\{\int\psi/2l_0 (Lip\psi),1/2C\}$, we obtain
\begin{equation}\label{equ_rhot-}
    \int \rho_{t,\epsilon}^-\dd \vol\gg \vol(S_t)\int\psi.
\end{equation}

By \eqref{equ.domain} 
\begin{equation}\label{equ_dt+}
\vol(S^+_{t+\epsilon})\geq \vol(S_t).
\end{equation}
Therefore
\begin{equation}\label{equ_rhot+}
    \int \rho_{t,\epsilon}^+\dd \vol\geq \int \rho_t\dd \vol\geq 
    \vol(S_t)\int\psi.
\end{equation}

After these preparation, we can start to compute the integral appears in error term of mean ergodic theorem. By \eqref{equ_rhot-}, we obtain when $\epsilon\leq \min\{\int\psi/2l_0 (Lip\psi),1/2C\}$
\begin{align*}
    \|\rho^{-}_{t,\epsilon}\|_q^q/(\int\rho_{t,\epsilon}^-)^q \ll \int |\rho_{t}|^q/(\vol(S_t)\int\psi)^q\leq  \frac{1}{\vol(S_t)^{q-1}}(Lip\psi)^q/(\int\psi)^q.
\end{align*}
For $\rho_{t,\epsilon}^+$, by \eqref{equ_rhot+} and \eqref{equ_dt+} we have
\begin{align*}
    \|\rho^{+}_{t,\epsilon}\|_q^q/(\int\rho_{t,\epsilon}^+)^q \ll \frac{1}{\vol(S_t)^{q-1}} \int(\psi+(Lip\psi)l_0\epsilon)^q/(\int\psi)^q.
\end{align*}
We obtain if $t>t_1$,
\begin{equation}
     \|\rho^{+}_{t,\epsilon}\|_q^q/(\int\rho_{t,\epsilon}^+)^q\ll \frac{1}{\vol(S_t)^{q-1}}(Lip\psi)^q/(\int\psi)^q.
\end{equation}

Applying the above formulas for $\widetilde\beta^{\pm}_{t,\epsilon}$, combined with mean ergodic estimate \eqref{equ.meanergodic}, we obtain the lemma.
\end{proof}

\textbf{Step 3:} 
The mean ergodic theorem only gives an estimate of $L^2$ norm, but what we need is an estimate at some points. So we need to use the Chebyshev inequality. The remaining work is to collect the error terms. This part is similar to the proof of Theorem 1.9 in \cite{gorodnikCountingLatticePoints2012}.
\begin{proof}[Proof of Theorem \ref{theo-gorodnik-nevo1} ]
Suppose $\epsilon\leq \epsilon_{inj}$. Applying \eqref{equ_rho+} to $f=\varphi_\epsilon$, by Chebyshev's inequality, we obtain for any $\eta>0$
\begin{equation}
    m_{\Gamma\backslash G}\{h\, : \, | \pi(\beta_{t,\epsilon}^{+})(\varphi_\epsilon)(h\Gamma)-\int\varphi_\epsilon|>\eta \}\leq (\frac{E(t)\|\varphi_\epsilon\|_{L^2}}{\eta})^2.
\end{equation}
If $(E(t)\|\varphi_\epsilon\|_{L^2}/\eta)^2<m_{\Gamma\backslash G}(\calO_\epsilon)/2=\vol(\calO_\epsilon)/2V(\Gamma)$,
(here we need $\|\varphi_\epsilon\|^2_{L^2(\Gamma\backslash G)}=\vol(\calO_\epsilon)/V(\Gamma)$.)
for example we can take $\eta=\frac{2E(t)}{\vol(\calO_\epsilon)}$, then due to $\calO_\epsilon$ injective there exists $h\in\calO_\epsilon$ such that \[\pi(\beta_{t,\epsilon}^{+})(\varphi_\epsilon)(h\Gamma)<\eta+\int\varphi_\epsilon. \]
Then by Lemma \ref{lem_rhot},
\begin{align*}
    \sum_{\gamma\in \Gamma\cap S_t}\psi(\gamma_o^+,\gamma_o^-)&=\sum_{\gamma\in\Gamma}\rho_t(\gamma)\leq\pi(\beta_{t,\epsilon}^+)(\varphi_\epsilon)(h\Gamma)\int\rho_{t,\epsilon}^+\leq(\eta+\frac{1}{V(\Gamma)})\int\rho_{t,\epsilon}^+\\
    &=\frac{\int\rho_t}{V(\Gamma)}(1+\eta V(\Gamma))+O(\epsilon (Lip\psi)\vol(S_t)),
\end{align*}
where the last inequality is due to \eqref{equ_diff}.
Therefore
\[ \frac{\sum_{\gamma\in \Gamma\cap S_t}\psi(\gamma_o^+,\gamma_o^-)}{\int\rho_t}-\frac{1}{V(\Gamma)} \leq\frac{E(t)}{2\vol(\calO_\epsilon)}+\epsilon\frac{Lip\psi}{\int\psi}\frac{1}{V(\Gamma)}\ll \frac{E(t)}{\epsilon^{d_0}}+\epsilon\frac{Lip\psi}{\int\psi},\]
where $d_0$ is the dimension of group $G$. Hence
\begin{equation}\label{equ_phileq}
\begin{split}
   \sum_{\gamma\in \Gamma\cap S_t}\psi(\gamma_o^+,\gamma_o^-)-\frac{\vol(S_t)}{V(\Gamma)}\int\psi\ll \vol(S_t)(\int\psi\left( \frac{E(t)}{\epsilon^{d_0}}+\epsilon\frac{Lip\psi}{\int\psi}\right)).
\end{split}
\end{equation}

In order to optimize the error term, we take $$\epsilon=(E(t)\int\psi/Lip\psi)^{1/(1+{d_0})},$$
then the error term in the above formula is
\[E(t)^{1/(1+{d_0})}(\frac{Lip\psi}{\int\psi})^{{d_0}/(1+{d_0})}\ll \vol(S_t)^{-\zeta}(\frac{Lip\psi}{\int\psi})^{({d_0}+q\kappa_2)/(1+{d_0})}\leq \vol(S_t)^{-\zeta}(\frac{Lip\psi}{\int\psi}),\]
where the last equality is due to \eqref{equ_et} and $q\kappa_2=1/n(G,\Gamma)\leq 1$, and where $\zeta=(q-1)\kappa_2/(1+{d_0})$. Here $\epsilon$ should be less than $\epsilon_1,\epsilon_\inj$, but 
\begin{equation}\label{equ_eps}
    \epsilon\leq \left(\frac{C}{\vol(S_t)^{(q-1)\kappa_2}}\frac{\int\psi}{Lip\psi} \right)^{1/(1+{d_0})}\leq \left(\frac{C}{\vol(S_t)^{(q-1)\kappa_2}} \right)^{1/(1+{d_0})}.
\end{equation}
The condition on $\epsilon$ is satisfied if $t$ is greater than some constant $t_2=C'|\log\epsilon_1|>0$ and $C_6|\log\epsilon_\inj|$.
Therefore by \eqref{equ_phileq}, we obtain one part of Theorem \ref{theo-gorodnik-nevo1} for $t>t_0=\max\{C_6|\log\epsilon_{inj}|,t_2\}$, with $t_0$ not depending on $\psi$.

For $\rho_{t,\epsilon}^-$, we can obtain the same bound with extra condition that $\epsilon<\min\{\int\psi/2l_0Lip\psi,1/2C \}$. Due to \eqref{equ_eps}, if $t$ is large than some constant $C_7$ only depending on $n(G,\Gamma)$, $G$ and $\vol(S(t))$, then we have $\epsilon<1/2C$. For the other inequality, if not then we have $\epsilon\geq \int\psi/2l_0Lip\psi$, by \eqref{equ_eps}, which implies
\begin{equation}
Lip\psi\gg \vol(S_t)^{\zeta_2}\int\psi,    
\end{equation}
with $\zeta_2=(q-1)\kappa_2/{d_0}$. Therefore by non-negativeness of $\psi$
\[\frac{\vol(S_t)}{V(\Gamma)}\left(\int\psi-C \vol(S_t)^{-\zeta_2}Lip\psi\right)\leq 0\leq \sum_{\gamma\cap S_t}\psi(k_1(\gamma),k_2(\gamma)). \]
By taking $$\kappa=\min \{\zeta, \zeta_1/2,\zeta_2 \}=\min\{\zeta_1/2,\frac{(q-1)}{q(1+d_0)n(G,\Gamma)}  \} ,$$
the proof is complete.
\end{proof}

\begin{rem}
Theorem \ref{theo-gorodnik-nevo1} is exactly Theorem 7.2 in \cite{gorodnikCountingLatticePoints2012} with an explicit error term, where no proof of Theorem 7.2 is given. But we cannot obtain this Theorem directly from Theorem 7.1 for Lipschitz well-rounded sets in \cite{gorodnikCountingLatticePoints2012} by approximating Lipschitz functions by level sets because the level sets of a Lipschitz function may not be uniformly Lipschitz well rounded. For one-dimensional cases, (i.e. $\rm{SL}_2(\R)$, Lipschitz function on $\rm{SO}(2)$), we can take a Lipschitz function $\psi$ as the distance to a Cantor set. Then the level sets $\{\psi<1/n\}$ approximate the Cantor set. Each set is Lipschitz well-rounded, but the constant in Lipschitz well-rounded blow up as $n$ tends to infinity because the number of intervals in $\{\psi<1/n \}$ goes to infinite.
\end{rem}

\subsection{Explicit cases: ball domain and parallelotope domain}

\paragraph{Verifying Lipschitz well roundness}
We only consider ball and parallelotope domains.
We take
\[S_t=D_t\cap K\tilde{A}^\delta K,\  S_t^+=D_t\cap K\tilde{A}^{\delta-\epsilon} K,\ S_t^-=D_t\cap K\tilde{A}^{\delta+\epsilon} K, \]
where $\delta$ are from Lemma \ref{lem_effective_cartan}.
By Lemma \ref{lem-cartan-diff}, we know that this choice of $S_t^\pm$ satisfies \eqref{equ.domain}.

\begin{figure}[h!]
    \centering
    \includegraphics[width=8cm]{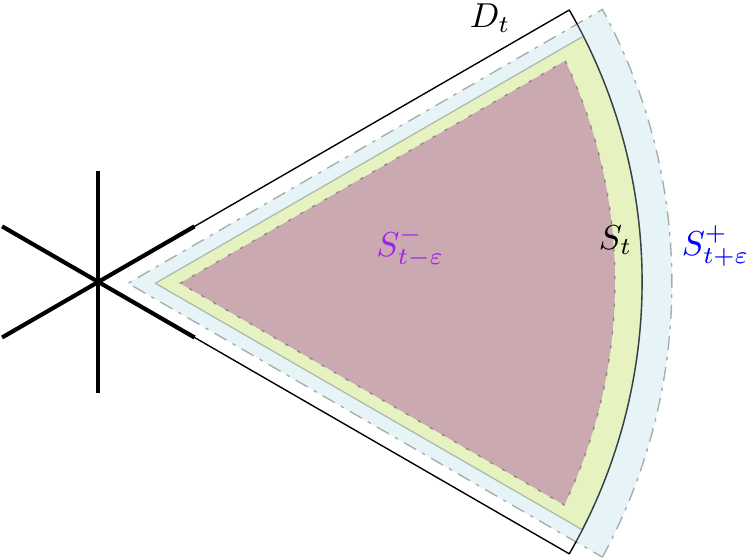}
    \caption{Ball domain, $S^+_{t+\varepsilon}$ is the blue outer layer delimited by the dash dots, $S_t$ in green yellow is the mid layer delimited by the gray line, $S^-_{t-\varepsilon}$ is the innermost layer delimited by the gray dotted line}
    \label{fig:DLdomain}
\end{figure}
\begin{figure}[h!]
    \centering
    \includegraphics[width=10cm]{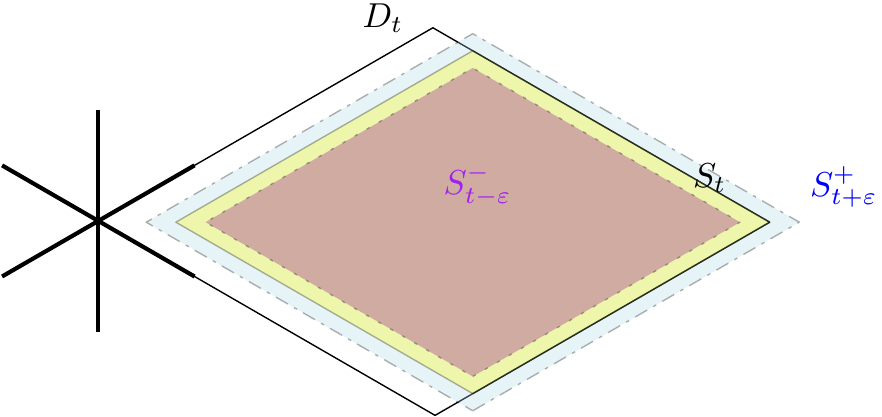}
    \caption{Parallelotope domain, $S^+_{t+\varepsilon}$ is the blue outer layer delimited by the gray dash dots, $S_t$ in green yellow is the mid layer delimited by the gray line and a portion of the black line, $S^-_{t-\varepsilon}$ is the innermost layer delimited by the gray dotted line }
    \label{fig:DGSdomain}
\end{figure}

We need to verify \eqref{equ.lip}. We observe that
\begin{align*}
\vol(S^+_{t+\epsilon}-S^-_{t-\epsilon})&\leq \vol(S^+_{t+\epsilon}-S^-_{t+\epsilon})+\vol(S^-_{t+\epsilon}-S^+_{t-\epsilon})+\vol(S^+_{t-\epsilon}-S^-_{t-\epsilon})\\
 &\leq \vol(S^+_{t+\epsilon}-S^-_{t+\epsilon})+\vol(S_{t+\epsilon}-S_{t-\epsilon})+\vol(S^+_{t-\epsilon}-S^-_{t-\epsilon}).
\end{align*}
 Both the case will be verified through local Lipschitz property of logarithm of volume (second term) and estimates of volume near the boundary of the region (first and third term).
 
 By admitting Lemma \ref{lem_volrest_proof} and Lemma \ref{lem_vollip_proof}, we obtain
 \[\vol(S^+_{t+\epsilon}-S^-_{t-\epsilon})\leq C\epsilon\vol(S_t), \]
 which is exactly \eqref{equ.lip}.
 
 It remains to prove Lemma \ref{lem_volrest_proof} and Lemma \ref{lem_vollip_proof}.
 
\paragraph{Boundary estimate}
 
We recall the Harish-Chandra formula
\[\vol(K\exp(\frak D)K)=\int_{\frakD}\prod_{\alpha\in\Sigma^+}\sinh(\alpha(Y))^{m_\alpha}\dd Leb(Y), \]
where $m_\alpha\in\N$ is the multiplicity of the positive root $\alpha$ and $\frakD$ is measurable subset of $\frak a^+$. To simplify the notation, we write $\frakS_t\subset \frak a^+$ for $S_t=K\exp(\frakS_t)K$. Similarly for $S_t^{\pm}$ and $D_t$.

Due to $\sup_{Y\in\frakD_t}2\rho(Y)\leq \delta_0t $, by the Harish-Chandra formula and $\sinh(\alpha(Y))\leq e^{\alpha(Y)}$, we obtain
\begin{equation}
    \vol(S_t)\leq \vol(D_t)\leq \int_{\frakD_t}e^{2\rho(Y)}\dd Y \ll e^{\delta_0t}t^{r_G}.
\end{equation}

Then we do the rest cases.
\begin{lem}\label{lem_volrest_proof}
For ball and parallelotope domains, we have
\[e^{\delta_0t}\ll \vol(S_t)\leq \vol(D_t),\  \vol(S^+_t-S^-_t)\ll \epsilon \vol(S_t). \]
\end{lem}

\begin{proof}
For the upper bound, 
by the Harish-Chandra formula, we obtain
\[\vol(S_t^+-S_t^-)\leq \int_{\frakS_t^+-\frakS_t^-}e^{2\rho(Y)}\dd Y  \ll \epsilon t^{r_G-1}e^{c_0t}, \]
where $c_0=\sup_{Y\in \frakS_t^+-\frakS_t^-}2\rho(Y)/t<\delta_0$ due to the choice of the domains.

For the lower bound of ball domain, we use volume estimates from \cite{knieper_asymptotic_1997}, \cite[Thm 5.8]{helgasonGroupsGeometricAnalysis2000}, \cite[Thm 6.1]{gorodnikIntegralPointsSymmetric2009} to obtain that $\vol(S_t)\asymp t^{(r_G-1)/2}e^{\delta_0 t}$.

For the lower bound of parallelotope domain, due to Lemma \ref{lem-volume-paral}, we obtain
\[ \vol(S_t)\geq \vol(D_t)-\vol(D_t\backslash S_t)\gg e^{\delta_0 t}. \]
\end{proof}

\paragraph{Local Lipschitz property of logarithm of volume}
\begin{lem}\label{lem_vollip_proof}
There exists $C>0$ such that for $\epsilon<1/C$ and $t>1$
\[ \vol(S_{t+\epsilon})-\vol(S_t)\leq C\epsilon \vol(S_t),\ \vol(D_{t+\epsilon})-\vol(D_t)\leq C\epsilon \vol(D_t) .\]
\end{lem}
\begin{proof}
We use a similar computation as in Proposition 7.1 in \cite{gorodnikErgodicTheoryLattice2010}. We use the polar coordinate $(r,\theta)\in \R^+\times \frak a_1^+$ with $\frak a_1^+=\{Y\in\frak a^+,\ \|Y\|=1 \}$. Then we can rewrite the Harish-Chandra formula. For ball and parallelotope domains, using the cone shape of domains, we obtain
\begin{equation}
    \vol(D_{t+\epsilon}-D_t)=\int_{(r,\theta)\in \frakD_{t+\epsilon}-\frakD_t}\xi(r,\theta)\dd r\dd \theta=\int \dd\theta\int_{t(\theta)}^{t(\theta)+\epsilon} \xi(r,\theta)\dd r.
\end{equation}
Since $\xi(r,\theta)$ is a continuous function, we have
\[\int_{t(\theta)}^{t(\theta)+\epsilon}\xi(r,\theta)\dd r=\epsilon\xi(r(\theta),\theta), \]
with some $r(\theta)\in[t(\theta),t(\theta)+\epsilon]$.
Lemma A.3 in \cite{eskin_unipotent_1996} implies that there exists $C>0$ such that for $r>1$
\begin{equation*}
    \xi(r,\theta)\leq C\int_0^r\xi(s,\theta)\dd s.
\end{equation*}
We have
\begin{equation*}
    \xi(r(\theta),\theta)\leq C\int_0^{r(\theta)}\xi(s,\theta)\dd s\leq C\int_0^{t(\theta)+\epsilon}\xi(s,\theta)\dd s
\end{equation*}
Therefore, we have
\begin{equation*}
    \vol(D_{t+\epsilon}-D_t)\leq C\int \dd\theta\left(\epsilon\int_0^{t(\theta)+\epsilon} \xi(r,\theta)\dd r\right)=C\epsilon \vol(D_{t+\epsilon}).
\end{equation*}
By taking $\epsilon$ small such that $C\epsilon<1/2$, we obtain
\[  \vol(D_{t+\epsilon}-D_t)\leq C'\epsilon \vol(D_t) \]
for some new constant $C'>0$.

For both ball domain and parallelotope domain, due to definition and boundary estimates (Lemma \ref{volume_reste}) and by setting $C':= \frac{C}{1- C \epsilon}$ we have
\begin{align*}
    \vol(S_{t+\epsilon})-\vol(S_t)&=\vol(D_{t+\epsilon})-\vol(D_t)-\vol((D_{t+\epsilon}-D_t)\cap A^\delta)\\
    &\leq \vol(D_{t+\epsilon})-\vol(D_t)\leq C\epsilon \vol(D_t)\leq C'\epsilon \vol(S_t).
\end{align*}

\end{proof}

\subsection{Proof of Corollary \ref{theo-gorodnik-nevo2}}
The domains we are interested in may have singular elements. We use estimates of singular elements to obtain Corollary \ref{theo-gorodnik-nevo2}. 
\begin{proof}[Proof of Corollary \ref{theo-gorodnik-nevo2}]

Let $S_t^\delta=D_t^{reg}-S_t$ for ball and parallelotope domains, then
\begin{align*}
    &\frac{1}{ \vol(S_t\cup S_t^\delta)} \sum_{\gamma \in (S_t\cup S_t^\delta)\cap \Gamma} \psi(\gamma_o^+,\gamma_o^-)-\frac{1}{ \vol(S_t)} \sum_{\gamma \in S_t\cap \Gamma} \psi(\gamma_o^+,\gamma_o^-) \\
    =&\frac{1}{ \vol(S_t\cup S_t^\delta)} \sum_{\gamma \in  S_t^\delta\cap \Gamma} \psi(\gamma_o^+,\gamma_o^-)+(\frac{1}{ \vol(S_t\cup S_t^\delta)}-\frac{1}{ \vol(S_t)}) \sum_{\gamma \in S_t\cap \Gamma} \psi(\gamma_o^+,\gamma_o^-)\\
    \leq&\left( \frac{|S_t^\delta\cap\Gamma|}{ \vol(S_t)}+\frac{|S_t\cap\Gamma|\vol(S_t^\delta)}{ \vol(S_t)^2}\right)|\psi|_\infty.
\end{align*}
By Lemma \ref{lem-cartan-diff} and Lemma \ref{volume_reste}

\begin{equation*}
    |S_t^\delta\cap\Gamma|\leq \frac{ \vol(S_t^\delta\calO_{\epsilon_{inj}})}{ \vol(\calO_{\epsilon_{inj}})}\leq \vol(S_{t+\epsilon_{inj}}^{\delta+\epsilon_{inj} })\epsilon_{inj}^{-\dim G}\ll \vol(S_t)^{1-\kappa}\epsilon_{inj}^{-\dim G}.
\end{equation*}

For the term $|S_t\cap\Gamma|$, by a similar estimate and Lemma \ref{lem_vollip_proof}, we obtain

\begin{equation*}
    |S_t\cap\Gamma|\ll \vol(S_t)\epsilon_{inj}^{-\dim  G}.
\end{equation*}

Therefore, we have

\begin{align*}
    &\frac{1}{ \vol(S_t\cup S_t^\delta)} \sum_{\gamma \in (S_t\cup S_t^\delta)\cap \Gamma} \psi(\gamma_o^+,\gamma_o^-)-\frac{1}{ \vol(S_t)} \sum_{\gamma \in S_t\cap \Gamma} \psi(\gamma_o^+,\gamma_o^-) \\
    & \ll \vol(S_t)^{-\kappa}\epsilon_{inj}^{-\dim G}|\psi|_\infty.
\end{align*}
Therefore, if $t\geq C|\log\epsilon_{inj}|$ for some constant $C>0$, then we obtain the result.
\end{proof}

\clearpage

\bibliographystyle{alpha}

\bigskip
 \noindent 
	\it{Institut f\"ur Mathematik, Universit\"at Z\"urich, 8057 Z\"urich; \textsc{current}: CNRS-Centre de Math\'ematiques Laurent Schwartz ,\'Ecole Polytechnique, Palaiseau, France}  \\
	email: {\tt jialun.li@polytechnique.edu} 
		
		\bigskip   
		
\noindent 
		\it{Fakult\"at f\"ur Mathematik und Informatik
Universit\"at Heidelberg,
69120 Heidelberg;
\textsc{current}: MCF-Laboratoire de Math\'ematiques d'Orsay, Facult\'es des Sciences, Universit\'e Paris-Saclay, Orsay, France
}  \\
			email: {\tt nguyen-thi.dang@universite-paris-saclay.fr}

\end{document}